\documentclass{article}
\usepackage{pdfsync}
\usepackage{cite}
\usepackage{amsthm,amssymb,amsmath,epic,rotating,hyperref}
\newtheorem{theorem}{Theorem}
\newtheorem*{theorem5}{Theorem 5}
\newtheorem*{theorem2}{Theorem 2}

\newtheorem*{theorem1}{Theorem 1}
\newtheorem{proposition}{Proposition}

\newtheorem{lemma}{Lemma}

\newtheorem{corollary}{Corollary}
\newtheorem*{corollary1}{Corollary}

\newtheorem*{definition1}{Definition}
\begin{document}
\title{ON THE REES---SUSHKEVICH VARIETY}
\author{Stanislav Kublanovsky\\
TPO "Severnyi ochag"\\ 30, B. Konjushennaja 15,\\
191186, St. Petersburg, Russia\\
e-mail: stas@norths.spb.su}
\date{19.07.1997}

\maketitle

\begin{abstract}
We study varieties of semigroups related to completely 0-simple semigroup. We present here an algorithmic descriptions of these varieties in terms of "forbidden" semigroups. We also describe residually completely 0-simple varieties of semigroups in terms of identities.
\end{abstract}

\section{Introduction.}

One of the most important classes of semigroups is the class of completely 0-simple semigroups. Recall that a semigroup is called 0-simple if it does not have ideals except itself and possibly 0. A 0-simple semigroup is called completely 0-simple if it has a minimal non-zero idempotent. Such semigroups play a very important role in semigroup theory. In 1940 Rees \cite{Rees1} showed that completely 0-simple semigroups may be described via the same construction as that used by Sushkevich \cite{Sushkevich1}.

It is shown in \cite{Hall1} that the identities
\begin{align}
\label{1}
& x^{2}=x^{n+2}\\
\label{2}
& xyx=(xy)^{n+1}x\\
\label{3}
& (xhz)^nxyz=xyz(xhz)^n
\end{align}
form an identity basis of the variety generated by all completely 0-simple
semigroups whose subgroups have exponent dividing $n$.

\begin{definition1}
Any semigroup variety satisfying, for some $n$, the identities \eqref{1}, \eqref{2}, \eqref{3} is called a Rees---Sushkevich variety.
\end{definition1}

We are interested in the following problem: given a finite set $\Sigma$ of identities, is the semigroup variety defined by $\Sigma$ a Rees---Sushkevich variety?

Clearly, this problem is a special case of the general problem of sequence of
identities. It was shown by  Murskii in 1968 \cite{Murskii1} that this problem is undecidable for semigroups. However we are able to prove that the problem of recognizing Rees---Sushkevich varieties is decidable.

Consider the following semigroup presentations:
\begin{enumerate}
  \item $\langle x,y \mid x=x^{2}; y^{2}=0; xy=yx\rangle$
  \item $\langle x,y \mid x^{2}=0; y^{2}=0; xyx=yxy\rangle$
  \item $\langle x,y \mid x^{2}=x^{3}; xy=y; yx^{2}=0; y^{2}=0\rangle$
  \item $\langle x,y \mid x^{2}=x^{3}; yx=y; x^{2}y=0; y^{2}=0\rangle$
  \item $\langle x \mid x^{3}=0\rangle$
  \item $\langle x,y \mid x^{2}=0; y=y^2; yxy=0\rangle$
  \item $\langle x,y \mid x^{2}=0; y^{2}=y^{n+2}; yxy=0; xy^{q}x=0, (q=2,\cdots,n); xyx=xy^{n+1}x\rangle$
  \item $\langle x,y \mid xy=xyx=xy^{2}; yx=yxy=yx^{2}; x^{2}=x^{2}y=x^{3}; y^{2}=y^{2}x=y^{3} \rangle$
  \item $\langle x,y \mid xy=yxy=x^{2}y; yx=xyx=y^{2}x; x^{2}=yx^{2}=x^{3}; y^{2}=xy^{2}=y^{3}\rangle$
  \item $\langle a,x,y \mid a^{2}=x^{2}=y^{2}=xy=yx=0; ax=axax; ay=ayay; xa=xaxa; ya=yaya; xay=xax; yax=yay\rangle$
  \item $\langle a,x,y \mid a^{2}=x^{2}=y^{2}=xy=yx=0; xa=xaxa; ya=yaya; ax=axax; ay=ayay; xay=yay; yax=xax\rangle$
  \item $\langle a,x,y \mid x=x^{2}; y=y^{2}; a=a^{2}; xy=x; yx=y; ax=xa=x; ay=ya=y\rangle$
  \item $\langle a,x,y \mid x=x^{2}; y=y^{2}; a=a^{2}; xy=y; yx=x; ax=xa=x; ay=ya=y\rangle$
\end{enumerate}
It can be shown that each of this presentation defines a finite semigroup. Note that the presentation $7$ defines in fact a series of semigroups depending on
an integer $n$.

\begin{theorem1}\label{MainTheorem}
A semigroup variety $V$ is a Rees---Sushkevich variety if and only if it does not contain any of the semigroups $(1)$-$(13)$.
\end{theorem1}

At first sight, Theorem \ref{MainTheorem} is not effective because the class of the semigroups forbidden by Theorem \ref{MainTheorem} is infinite. But in fact, it can be used to recognize in finite number of steps whether the variety $V$ defined by a finite set of identities $\Sigma$ is a Rees---Sushkevich variety. Call an identity balanced if the multiplicity of occurrence of each varible in the left and the right sides of the identity is the same. If $\Sigma$ consists only of balanced identities then any commutative semigroup belongs to $V$ and, consequently, $V$ is not a Rees---Sushkevich variety. Assume that $\Sigma$ contains an unbalanced identity; it is obvious that this identity implies an identity $x^{m}=x^{m+n}$. But this last identity fails in all semigroups $K_{r}$ with $r>n$. Therefore, to recognize if the variety $V$ satisfies the requirements of Theorem \ref{MainTheorem} it remains to test only a finite number of forbidden semigroups $1-6$, $8-13$, $K_{1}$, $K_{2}$ $\cdots$ $K_{n}$. Thus, we obtain the following corollary

\begin{corollary}\label{corollary1}
  There exists an algorithm determining if the variety defined by a given finite system of identities is a Rees---Sushkevich variety.
\end{corollary}

\begin{corollary}\label{corollary2}
For a semigroup variety $V$, if every finite semigroup of $V$ satisfies the identities \eqref{1}, \eqref{2}, \eqref{3} then each semigroup of $V$ satisfies these identities.
\end{corollary}

It is easy to deduce from Theorem \ref{MainTheorem} the following corollary.

\begin{corollary}\label{corollary3}
The algorithm that determines if the variety defined by a given finite system of identities coincides with the variety of semigroups generated by all completely 0-simple semigroups over the periodic groups of exponent $n$ exists if and only if the equational theory of the variety of periodic groups of exponent $n$ is decidable.
\end{corollary}

Well-known results by Adian \cite{Adian1} and Ivanov \cite{Ivanov1} on the Burnside problem imply therefore that such an algorithm exists for all odd $n\ge 665$ and for all $n>2^{48}$ divisible by $2^9$.

By definition, any Rees---Sushkevich variety is contained in the variety
generated by all completely 0-simple semigroups over periodic groups
of bounded exponent. However it is easy to see that not every
Rees---Sushkevich variety is generated by its completely 0-simple semigroups.
To characterize varieties generated by its completely 0-simple semigroups
among Rees---Sushkevich varieties we need following semigroups:

$N_{2}=\langle a\mid a^2 =0\rangle$

$A_{0}=\langle x,y\mid x^2 =x;\; y^2=y;\; yx=0\rangle $

$A_{2}$, $B_{2}$ are well-known finite semigroups of matrix (see \cite{Shevrin1}).

These semigroups can be defined by presentations:

$A_2=\langle a,b\mid aba=a;\; bab=b;\; a^2=a;\; b^2=0\rangle$

$B_2=\langle a,b\mid aba=a;\; bab=b;\; a^2=b^2=0\rangle $

\begin{theorem2}\label{exacttheorem}
  Let $V$ be a periodic semigroup variety. The variety $V$ is generated by completely 0-simple semigroups if and only if $V$ is a Rees---Sushkevich variety and one of the following condition is fulfilled:
  \begin{enumerate}
    \item $B_{2}\in V$, $A_{0}\notin V$;
    \item $A_{2}\in V$;
    \item $N_{2}\notin V$.
  \end{enumerate}
\end{theorem2}

For brevity, we call any variety generated by completely 0-simple semigroups whose subgroups have a bounded exponent \emph{exact}. It is not hard to show that every exact variety is generated by a single completely 0-simple semigroup.

\begin{corollary}\label{corollary4}
  There exists an algorithm that determines if the variety defined by a given finite system of identities (or by a given finite semigroup) is exact.
\end{corollary}

We have also the following surprising corollary:

\begin{corollary}\label{corollary5}
  The class of all exact semigroup varieties forms a sublattice of the lattice of semigroup varieties.
\end{corollary}

In Theorems 3-4 we describe the varieties of semigroups where every semigroup is the subdirect product of completely 0-simple semigroup (of principal factors).

In Theorem 5 we describe the varieties of semigroups where every semigroups is
embedable into direct product of completely 0-simple semigroups.

\begin{theorem5}
  For every variety of semigroup $V$ the following conditions are equivalent.
  \begin{enumerate}
    \item Every semigroup from $V$ is embeds into a direct product of completely 0-simple semigroups.
    \item Every finitely generated semigroup from $V$ is embeds into a direct product of completely 0-simple semigroups.
    \item $V$ satisfies one of following systems of identities:
    \begin{enumerate}
      \item $[xy = x^{n+1}y^{n+1}, (axyb)^{n} = (ayxb)^{n}]$;
      \item $[xy = xy^{n+1}, axay = ayax, abxy = abyx]$;
      \item $[xy = x^{n+1}y, xaya = yaxa, xyab = yxab]$.
    \end{enumerate}
  \end{enumerate}
\end{theorem5}

Theorem 6,7 are an analog of the Theorem 5 for some natural subclasses of completely 0-simple semigroups.

Remark that these varieties are in many cases closely connected with the varieties of residually finite semigroups which were independently described by E. Golubov and M. Sapir \cite{Golubov1}, R. McKenzie \cite{McKenzie2} and the author \cite{Kublanovsky6,Kublanovsky7}.

\section{Notation and terminology}

Let $B_{n}$ be the class of all groups of exponent $n$. We denote by $CS^{0}(B_{n})$ the class of all completely 0-simple semigroups over groups from $B_{n}$. Let $VarCS^{0}(B_{n})$ be  the semigroup variety generated by the class $CS^{0}(B_{n})$. Its identity basis quoted in \cite{Hall1} consists of three identities \eqref{1}, \eqref{2}, \eqref{3} in four variables.

We call a Rees---Sushkevich variety any semigroup variety $V$ in which the identities  \eqref{1}, \eqref{2}, \eqref{3} hold (for a positive integer n). Thus, Rees---Sushkevich varieties are exactly the varieties contained in $VarCS^{0}(B_{n})$ for some $n$.

Classic results by Sushkevich and Rees (see \cite{Clifford1}) show that completely 0-simple semigroups have the following structure. Let $G$ be a group, let $L$ and $R$ be two non-empty sets and let $P$ be an $R\times L$-matrix over the group $G$ with $0$ adjoined such that every row and every column of $P$ contains a non-zero element. Let $M^{0}(G; L, R, P)$ be the set $(L\times G\times R)\cup\{0\}$ with the following binary operation:
$$a\cdot0=0\cdot a=0 \text{ for all } a\in(L\times G\times R)\cup\{0\}, \text{ and}$$
$$(l_{1},g_{1},r_{1})\cdot(l_{2},g_{2},r_{2})=\begin{cases}
(l_{1},g_{1}p_{r_{1}l_{2}}g_{2},r_{2}), &\text{if } p_{r_{1}l_{2}}\neq0\\
0, &\text{otherwise}\end{cases}$$
Then $M^{0}(G; L, R, P)$ is a completely 0-simple semigroup and every completely 0-simple semigroup is isomorphic to $M^{0}(G; L, R, P)$ for some $G,L,R,P$. The semigroup $M^{0}(G; L, R, P)$ is called the \emph{Rees matrix semigroup} over $G$ with $0$ adjoined.

A semigroup $T$ is called a \emph{divisor} of a semigroup $S$ if $T$ is a homomorphic image of a subsemigroup $S$.

A semigroup $S$ is called \emph{periodic} if all its one-generated subsemigroups are finite. A theorem by Mann \cite{Clifford1} shows that every periodic 0-simple semigroup is completely 0-simple semigroup.

Let $S$ be a semigroup, $x\in S$. Then $J(x)$ denotes the principal ideal generated by $x$ while $I_{x}=\{y\in S \mid x\notin J(y)\}$. Then $K_{x}$ denotes the principal factor $J(x)/(I_{x}\cap J(x))$ (see \cite{Clifford1}). As usual, $S^{1}$ denotes $S$ with identity adjoined, $\mathcal{J}$, $\mathcal{L}$, $\mathcal{R}$, $\mathcal{D}$, $\mathcal{H}$ denote the Green relations on $S$ (see \cite{Clifford1}).

Let us fix notation. The free semigroup over an alphabet  $X$, i.e., the set of all words over $X$ with the operation of concatenation, will be denoted by $X^{+}$. For a word $\mathcal{U}\in X^{+}$ we denote by $\chi(\mathcal{U})$ the set of variables occurring in $\mathcal{U}$ and by $|\mathcal{U}|$ the length of the word $\mathcal{U}$. Further, for a word $\mathcal{U}=x_{1}x_{2}\cdots x_{n}$ and an integer $s\leq n=|\mathcal{U}|$ we denote by $l_{s}(\mathcal{U})$ the word $x_{1}x_{2}\cdots x_{s}$ and by $r_{s}(\mathcal{U})$ the word $x_{n-s+1}\cdots x_{n-1}x_{n}$; in particular $l(\mathcal{U})=l_1(\mathcal{U})$ is the first and $r(\mathcal{U})=r_{1}(\mathcal{U})$ is the last letter of the word $\mathcal{U}$.

Following \cite{Mashevitzky6}, we call a word of length $>1$ \emph{covered by cycles} if each its subword of length $2$ is contained in a subword in which the first and the last letters are the same.

For a class $C$ of semigroups we denote by $Var(C)$ and $qVar(C)$ the variety and the quasivariety of semigroups generated by $C$. We shall say that a semigroup $S$ is approximated by semigroups from $C$ or that it is residually $C$-semigroup (or residually in the class $C$) if for any two distinct elements $\mathcal{U}, \mathcal{V}\in S$ there exists a semigroup $T\in C$ and a homomorphism $f:S\rightarrow T$ such that $f(\mathcal{U})\not= f(\mathcal{V})$.
If $C$ is the class of all finite semigroups then $S$ is called residually
finite. If $C$ is the class of all completely 0-simple semigroups then $S$
is called residually completely 0-simple semigroup etc. It is easy to see that $S$ is a residually $C$-semigroup if and only if $S$ can be embedded into the direct product of semigroups from $C$.

We shall say that a semigroup $S$ is finitely approximated by Green relations $\rho$ ($\rho=\mathcal{L},\mathcal{R},\mathcal{H}$) if for any two elements $\mathcal{U}, \mathcal{V}\in S$ such that $(\mathcal{U}, \mathcal{V})\notin \rho_{S}$ there exists finite semigroup $T$ and a homomorphism $f:S\rightarrow T$ such that $(f(\mathcal{U}), f(\mathcal{V}))\notin \rho_{T}$ (see \cite{Kublanovsky7}).

For a semigroup $S$ we denote by $S^{1}$ the semigroup $S$ with the adjoined
outer unity. By $R_{n}$ and $L_{n}$ we denote the $n$-element semigroups of
right and left zeros. Further, we denote by $A_{2}$, $B_{2}$ the Rees
semigroups of matrix type over the 1-element group $E=\{1\}$ with the
sandwich-matrices $\left(
                    \begin{array}{cc}
                      1 & 1 \\
                      0 & 1 \\
                    \end{array}
                  \right)$,
                  $\left(
                    \begin{array}{cc}
                      1 & 0 \\
                      0 & 1 \\
                    \end{array}
                  \right)$.

It is easy to see that these semigroups coincide with the semigroups $A_{2}$, $B_{2}$ defined above (see Introduction).

A completely 0-simple semigroup $M=M^{0}(G; L, R, P)$ is called $B_{2}$-semigroup if the semigroup $A_{2}$ is not a divisor of $M$.

It is not difficult to show that a semigroup $M$ is a $B_{2}$-semigroup if and only if its matrix $P$ is equivalent \cite{Clifford1} to a matrix composed of several diagonal blocks each of which does not contain any $0$.

The following semigroups play a very important role in our study.

\begin{enumerate}
\item $A=\{ e,y,z,0 \mid e^2=e; ey=ye=z; ez=ze=z \}$
\item $B=\{x,y,a,b,c,0 \mid xy=a;yx=b; ax=xb=ya=by=c\}$
\item $C_{\lambda}=\{x,e,y,z,0 \mid x^2=xe=ex=e; e^2=e; xy=ey=y; yx=z; xz=ez=z\}$
\item $C_{\rho}$ is a semigroup antiisomorphic to $C_{\lambda}$
\item $N_{3}=\{a,b,0 \mid a^2=b\}$
\item $D=\{x,e \mid e^2=e; x^2=0; exe=0\}$
\item $K_{n}=\{x,y\mid y^2=y^{2+n}; x^2=0; yxy=0; xy^qx=0\,\,\, (q=2,\cdots,n);xyx=xy^{n+1}x\}$  ($n$ is an arbitrary positive integer)
\item  $F_{\lambda}=\{x,y\mid xy=xyx; yx=yxy; xy=xy^2; yx=yx^2;
x^3=x^2y=x^2 ;  y^3 =y^2  x=y^2 \}$
\item  $F_{\rho}$ --- the semigroup antiisomorphic to $F_\lambda $
\item  $ W_{\lambda}=\{a,x,y \mid a^2=0; x^2 = 0; y^2=0; xy=0; yx=0; axay=ax;
ayax=ay; xay=xax; yax=yay; xaxa=xa;  yaya=ya \}$
\item $W_\rho $ --- the semigroup antiisomorphic to $W_\lambda $
\item $L_{2}^{1}=\{a,b,e\mid a^2=ab=ae=ea=a; b^2=ba=eb=be=b; e^2=e\}$ (the 2-element semigroup of left zeroes with the outer unity)
\item $R_{2}^{1}=\{a,b,e\mid a^2=ba=ae=ea=a; b^2=ab=eb=be=b; e^2=e\}$ (the 2-element semigroup of left zeroes with the outer unity)
\end{enumerate}
All products which are not mentioned in the definitions (1)-(5) are equal to
0.

It is easy to prove that these semigroups coincide with the semigroups (1)-(13) defined above (see Introduction). These semigroups are called \emph{indicator Burnside semigroups}.

Introduce notation for some more semigroups:

$$S_{0}=\{a,b,c,0 \mid ab=ba=c\};$$
$$S_{1l}=\{e,a,0 \mid e=e^{2}, ae=a\};$$
$$S_{2l}=\{a,u,v,e,f \mid a=ue=vf=ae=af, e=e^{2}=ef, f=f^{2}=fe\};$$
$$S_{3l}=\{x,e,f,g \mid yz=z\text{ for } y,z\in S_{3l}, z\neq x, ex=fx=x^{2}=e, gx=f\}.$$
All products of the elements of the semigroups $S_{0}$, $S_{1l}$ which are not mentioned in the above definitions are equal to $0$. Further, by $S_{1r}$, $S_{2r}$, $S_{3r}$ we denote the semigroups antiisomorphic to the semigroups $S_{1l}$, $S_{2l}$, $S_{3l}$. All these semigroups were introduced by the author in \cite{Kublanovsky6} in order to investigate residually finite varieties of semigroups.

Let $CS$ ($CS(V)$) denote the class of all completely simple semigroups (of the variety $V$); recall that $CS^{0}$ ($CS^{0}(V)$) stands for the class of all completely 0-simple semigroups (of the variety $V$).

\section{Auxiliary statements}

To prove these theorems let us prove at first some Lemmas and other auxiliary facts.

\begin{proposition}\label{p1}
  If an identity $xy=\mathcal{U}$ holds in a semigroup variety $V$ and $|\mathcal{U}|>2$, then an identity of one of the following forms holds in $V$ as well:
\begin{subequations}\label{P1}
\begin{align}
\label{P1.1}
xy=(xy)^{n+1};\\
\label{P1.2}
xy=xy^{n+1};\\
\label{P1.3}
xy=x^{n+1}y
\end{align}
\end{subequations}
(here $n>1$ and $x$, $y$ are different variables).\qed
\end{proposition}

This fact was proved independently by several authors (see, for example, \cite{Kublanovsky7}, Proposition 7).

\begin{proposition}\label{p2}
  Each of identities \eqref{P1.1}, \eqref{P1.2}, \eqref{P1.3} implies the identities \eqref{1}, \eqref{2}.
\end{proposition}
\begin{proof}
  \eqref{P1.3}$\Rightarrow$\eqref{1}: make in \eqref{P1.3} the substitution $x\mapsto y$.

  \eqref{P1.3}$\Rightarrow$\eqref{2}: make in \eqref{P1.3} the substitution $x\mapsto xy$, $y\mapsto x$.

  Similarly it can be proven that \eqref{P1.2}$\Rightarrow$\eqref{1},\eqref{2}.

  \eqref{P1.1}$\Rightarrow$\eqref{1}: make in \eqref{P1.1} substitutions $y\mapsto x$, $y\mapsto x^{2}$; we obtain identities $x^{2}=x^{2n+2}$, $x^{3}=x^{3n+3}$, and it follows:
  $$x^{2}=x^{2n+2}=x^{3}x^{2n-1}=x^{3n+3}x^{2n-1}=(x^{2n+2})^{2}x^{n-2}=(x^2)^{2}x^{n-2}=x^{n+2}.$$

  \eqref{P1.1}$\Rightarrow$\eqref{2}: multiply \eqref{P1.1} on the right by $x$.
\end{proof}

\begin{proposition}\label{p3}
  For a periodic semigroup variety the following conditions are equivalent:
  \begin{enumerate}
    \item Any regular semigroup from $V$ is the subdirect product of its principal factors.
    \item Any regular semigroup from $V$ is the subdirect product of completely 0-simple semigroups.
    \item Any regular semigroup from $V$ is a residually completely 0-simple semigroup.
    \item There is a positive integer $n$ such that the identity $(axaya)^{n}=(ayaxa)^{n}$ holds in any regular semigroup of the variety $V$.
    \item The variety $V$ does not contain the semigroups $L_{2}^{1}$, $R_{2}^{1}$, $B_{2}^{1}$.
  \end{enumerate}
\end{proposition}
\begin{proof}
Remark that \textit{(1)} $\Rightarrow$ \textit{(2)} $\Rightarrow$ \textit{(3)} for any periodic variety. Let \textit{(3)} be fulfilled in the variety $V$. Let $n$ be the period of the variety $V$ and assume that $(axaya)^{n}\neq (ayaxa)^{n}$ for a regular semigroup $S\in V$ and elements $a,x,y\in V$. According to \textit{(3)} there exists a homomorphism $f:S\rightarrow M\in CS^{0}$ such that $f(\mathcal{U})\neq f(\mathcal{V})$, $\mathcal{U} =(axaya)^{n}, \mathcal{V}=(ayaxa)^{n}$. Recall that every idempotent of the regular semigroup $f(S)$  is 0-primitive, and (as it is shown in \cite{Kublanovsky6}) $f(S)$ decomposes into the direct product of completely 0-simple semigroups.  Therefore, we can assume $f(S)$ to be a completely 0-simple semigroup. But $(axaya)^{n}=(ayaxa)^{n}$ in any completely
0-simple semigroup of period $n$, i.e., $f(\mathcal{U})=f(\mathcal{V})$ in contradiction with the assumption. Hence, \textit{(3)} $\Rightarrow$ \textit{(4)}.

On the other hand, \textit{(4)} $\Rightarrow$ \textit{(5)}, since the identity $(axaya)^{n}=(ayaxa)^{n}$ is not fulfilled in the semigroups $L_{2}^{1}$, $R_{2}^{1}$, $B_{2}^{1}$. It remains to show, that \textit{(5)} implies \item1.
Let \textit{(5)} hold in the variety $V$. Let $S$ be a semigroup of the variety $V$ and $e,f\in S \quad (e \neq f)$ be idempotents belonging to the same $D$-class of the semigroup $S$, where $\mathcal{D}$ is the Green's relation (see \cite{Clifford1}). But in a regular periodic semigroup $\mathcal{D}$-classes coincide with $\mathcal{J}$-classes. Therefore, $e,f$ lay in the same principal factor $K$ of the semigroup $S$ and are not equal to 0. According to Mann's theorem, $K$ is a completely 0-simple semigroup.

Assume that the elements $e$ and $f$ have a common two-sided unity "1".
Suppose at first that $(ef)^{2} \neq 0$. If $e=fe$ then $ef= fef=(fe)^n=f$, and
consequently, $R_{2}^{1}\approx \{e,f,1\}\subset S$, i.e., $R_{2}^{1}\approx
\{e,f,1\}\subset S$ in contradiction with \item5. If $e\neq ef$ then $L_{2}^{1}\approx \{e,(fe)^{n},1\}\subset S$, and $L_{2}^{1}\in V$ in contradiction with \item5. Therefore, $(ef)^2=0$. Then $ef=0$ or $fe=0$ in $K$. If $ef=0$ and $fe=0$, then $B_{2}^{1}\subset K\in V$ in contradiction with the hypothesis. If $ef=0$ and $fe\neq 0$  (or \textit{vice versa}), then $A_{2}^{1}$ is a divisor of $K$ and, consequently, $A_{2}^{1}\in V$. But it is easy to show that $B_{2}^{1}\in Var(A_{2}^{1})$; therefore, $B_{2}^{1}\in V$ and it is again a contradiction.

The above argument shows that idempotents from the same $D$-class can not have a common two-sided unity. In this case the regular semigroup $S$ decomposes into the subdirect product of its principal factors (see \cite{Kublanovsky6}). Hence, \textit{(5)} $\Rightarrow$ \textit{(1)}.
\end{proof}

\begin{lemma}\label{l1}
  If an identity $\mathcal{U}=\mathcal{V}$ fails in the semigroup  $A$ then this identity implies an identity of the form
  \begin{equation}\label{L1}
    x^{k}yx^{l}=PyQyR
  \end{equation}
(here $P,Q,R$ are  words or empty symbols, $x,y$ are different variables, $k,l\geq0$).
\end{lemma}
\begin{proof}
  Let us consider all possible cases.
  \begin{enumerate}
    \item $\chi(\mathcal{U})\neq\chi(\mathcal{V})$. Replace all letters in $\mathcal{U}$ and all letters but one in $\mathcal{V}$ by $x$, and replace the remaining letter in $\mathcal{V}$ by $y$; we obtain an identity of the form $x^{k}=PyQ$. It follows that $x^{k}yx^{l}=PyQyx^{l}=PyQyR \quad (R=x^{l})$.
    \item $\chi(\mathcal{U})=\chi(\mathcal{V})$. Let $\chi(\mathcal{U})=\chi(\mathcal{V})=\{x_{1},x_{2},\cdots,x_{n}\}$. Since $\mathcal{U}=\mathcal{V}$ does not hold in the semigroup $A$, there exists a mapping $f:\{x_1,x_2,\cdots ,x_n\}\rightarrow A$ such that
\begin{equation}\label{IN}
\mathcal{U}(f(x_1),\cdots,f(x_n))\neq\mathcal{V}(f(x_1),\cdots,f(x_n)).
\end{equation}
We can conclude that either right-hand or left-hand side of this inequality
is distinct from 0. Assume, for example, that $\mathcal{U}(f(x_{1}),\cdots,f(x_{n}))\neq0$. Since $e$, $y$, $z$ are the only nonzero elements of the semigroup $A$, solely the following cases are possible:
  \begin{enumerate}
    \item $\mathcal{U}(f(x_1),\cdots,f(x_n))=e$. It follows from the definition of the semigroup $A$ that it is possible only if $f(x_1)=\cdots=f(x_n)$; then $\mathcal{V}(f(x_1)\cdots f(x_n))=e$ in contradiction with \eqref{IN}. Hence, this case is impossible.
    \item $\mathcal{U}(f(x_1)\cdots f(x_n))=y$. In this case $|\mathcal{U}|=1$ and $n=1$ (because $y \notin A^{2}$). Then $|\mathcal{V}|>1$ (otherwise $\mathcal{U}=\mathcal{V}$ should became the tautology $x_{1}=x_{1}$). This means that $\mathcal{U}=\mathcal{V}\Rightarrow y=y^{m}$ (for some $m>1$)  and, consequently,  $xyx=xy^mx$; it is an identity of the form \eqref{L1}.
    \item $\mathcal{U}(f(x_{1}),\cdots,f(x_{n}))=z$. This equality is possible only if there exists an index $i$ such that $f(x_{i})=y$ or $f(x_{i})=z$, the letter $x_{i}$ occurs in $\mathcal{U}$ only once, and $f(x_{j})=e$  for all $j\neq i$. If $x_{i}$ occurs in $\mathcal{V}$ only once as well, then $\mathcal{V}(f(x_1)\cdots f(x_n))=z$ in contradiction with \eqref{IN}. Hence, $n_{x_{i}}(\mathcal{V})>1$. Setting $x_{i}=y, x_{j}=x$ we obtain an identity of the form \eqref{L1}.
  \end{enumerate}
  \end{enumerate}
\end{proof}

\begin{lemma}\label{l2}
  If an identity $\mathcal{U}=\mathcal{V}$ fails in the semigroup  $B$ then this identity implies an identity of one of the following forms:
\begin{subequations}\label{L2}
\begin{align}
\label{L2.1}
xyx=P(xy)^{2}Q \text{ or } xyx=P(yx)^{2}Q;\\
\label{L2.2}
xyx=Py^{2}xQ \text{ or } xyx=Pxy^{2}Q;\\
\label{L2.3}
xyx=Px^{2}yQ \text{ or } xyx=Pyx^{2}Q,
\end{align}
\end{subequations}
where $P,Q$ are words or empty symbols and $x,y$ are different variables.
\end{lemma}
\begin{proof}
  Remark first of all that the semigroup $B$ can be given by generators and defining relations in the following way:
  $$B = \{x,y \mid x^2=0; y^2=0; xyx=yxy; (xy)^2=(yx)^2\}.$$
  Assume that the identity $\mathcal{U}=\mathcal{V}$ does not hold in the semigroup $B$. Let $\mathcal{U}, mathcal{V}$ be words over $\{x_1,x_2,\cdots,x_{n}\}$. Then there exists a mapping $f:\{x_1,x_2,\cdots,x_n\}\rightarrow B$ such that \eqref{IN} holds. For any $i$ the element $f(x_i)$ can be equal only to $x$ or $y$ or $xy$ or $yx$ or $xyx$ or 0. But $0=x^{2}$; therefore, substituting $x_i\mapsto f(x_i))$ in $\mathcal{U}$ and $\mathcal{V}$ we obtain a new identity in two variables (not more) which does not hold in $A$. Therefore, without loss of generality we can consider $\mathcal{U}$ and $\mathcal{V}$ as words over the two-letters alphabet $\{x_1,x_2\}$. Let us consider all possible cases.
 \begin{enumerate}
   \item $|\mathcal{U}|>3$ and $|\mathcal{U}|>3$. This means that at least one of the products $x^2$ or $y^2$ or $(xy)^2$ or $(yx)^2$ occurs in each of the words $\mathcal{U}$, $\mathcal{V}$. Making use of the definition of the multiplication in the semigroup $B$ we see now that in this case both sides of \eqref{IN} are  equal to 0. It contradicts the inequality \eqref{IN}; therefore, this  case is impossible.
   \item $|\mathcal{U}|\leq3$ and $|\mathcal{U}|>3$.
        \begin{enumerate}
          \item $|\mathcal{U}|=1$. In this case the word $\mathcal{U}$ consists of one letter: $\mathcal{U}=z$. Replace all letters in $\mathcal{V}$ by $z$; then the identity $\mathcal{U}=\mathcal{V}$ will turn into the identity $z=z^m$  $(m>1)$. Multiply the last identity on the left and on the right by $x$ and substitute $z\mapsto y$; we shall obtain so the identity $xyx=xy^mx$ of the form \eqref{L2.2}.
          \item $|\mathcal{U}|=2$. If $\mathcal{U}=x_{i}^{2}$, then $\mathcal{U}(f(x_1)\cdots f(x_n))=f(x_{i})^{2}=0$, $\mathcal{V}(f(x_1)\cdots f(x_n))=0$ because $|\mathcal{V}|>3$. It contradicts \eqref{IN}. Therefore, $\mathcal{U}$ is the product of two different variables, $\mathcal{U}=x_1x_2$, $x_1\neq x_2$.
              \begin{enumerate}
                \item $\chi(\mathcal{V})=\{x_1,x_2\}$. Since $|\mathcal{V}|>3$, there are only the following possibilities:
                    \begin{itemize}
                      \item $\mathcal{V}$ contains $x_1^2$; then $\mathcal{U}x_1=\mathcal{V}x_1$ is an identity of the form \eqref{L2.3}.
                      \item $\mathcal{V}$ contains $x_2^2$; then $\mathcal{U}x_1=\mathcal{V}x_1$ is an identity of the form \eqref{L2.2}.
                      \item $\mathcal{V}$ contains $(x_1x_2)^2$ or $(x_2x_1)^2$;  then $\mathcal{U}x_1=\mathcal{V}x_1$ is an identity of the form \eqref{L2.1}.
                    \end{itemize}
                \item  $\chi(\mathcal{V})=\{x_2\}$. Then our identity $\mathcal{U}=\mathcal{V}$ is the identity $x_1x_2=x_2^m$. Multiplying on the right by $x_1$  we obtain the identity $x_1x_2x_1=x_2^mx_1$ of the form \eqref{L2.2}.
                \item $\chi(\mathcal{V})=\{x_1\}$. As in the previous case we see that an identity of the form \eqref{L2.3} follows from $\mathcal{U}=\mathcal{V}$.
              \end{enumerate}
          \item $|\mathcal{V}|=3$. If $\mathcal{U}$ contains a square of any variable, then
              $$\mathcal{U}(f(x_1)\cdots f(x_n))=\mathcal{V}(f(x_1)\cdots f(x_n))=0$$
              in contradiction with \eqref{IN}. Hence, $\mathcal{U}=x_{i_1}x_{i_2}x_{i_3}$, where $x_{i_1}\neq x_{i_2}$, $x_{i_2}\neq x_{i_3}$. Since $x_{i_1},x_{i_2},x_{i_3}\in \{x_1,x_2\}$, it follows that $x_{i_3}= x_{i_1}$. Setting $x_{i_1}=x, x_{i_2}=y$ and taking into account that $\mathcal{V}$ is a word over $\{x_1,x_2\}$, the length of which is at least 4, we see that the so obtained identity has one of the forms \eqref{L2.1}, \eqref{L2.2} or \eqref{L2.3}.
        \end{enumerate}
   \item $|\mathcal{U}|\leq3$ and $|\mathcal{V}|\leq3$. Without loss of generality we can assume that $|\mathcal{U}|\leq|\mathcal{V}|$.
       \begin{enumerate}
         \item $|\mathcal{U}|=1$. As in the above we see that the identity $\mathcal{U}=\mathcal{V}$ implies an identity of the form \eqref{L2.2}.
         \item $|\mathcal{U}|=2$.
            \begin{enumerate}
              \item $\mathcal{U}=x_i^2$. If $\mathcal{V}$ contains a square of any variable, then
                  $$\mathcal{U}(f(x_1)\cdots f(x_n))=\mathcal{V}(f(x_1)\cdots f(x_n))=0.$$
                  Otherwise, $\mathcal{V}=x_ix_j$ or $\mathcal{V}=x_jx_i$ or $\mathcal{V}=x_{i_1}x_{i_2}x_{i_3}$. In all cases an identity of the form  \eqref{L2.2} or \eqref{L2.3} follows from the identity $\mathcal{U}=\mathcal{V}$.
              \item $\mathcal{U}=x_{i_1}x_{i_2}$, $x_{i_1}\neq x_{i_2}$. If $\mathcal{V}$ contains a square of any variable, then we obtain an identity of the form \eqref{L2.2} or \eqref{L2.3}. Otherwise, there are the following possibilities:
                  \begin{itemize}
                    \item $\mathcal{V}=x_{i_2}x_{i_1}$; then $\mathcal{U}=\mathcal{V}$ $\Rightarrow$ $xy=yx$ $\Rightarrow$ $xyx=yx^2$, and the last identity has the form \eqref{L2.3};
                    \item $\mathcal{V}=x_{i_1}x_{i_2}x_{i_1}$; then $\mathcal{U}=\mathcal{V}$ $\Rightarrow$ $xy=xyx$ $\Rightarrow$ $xyx=xyx^2$, and the last identity has the form \eqref{L2.3};
                    \item $\mathcal{V}=x_{i_2}x_{i_1}x_{i_2}$; then $\mathcal{U}=\mathcal{V}$ $\Rightarrow$ $xy=xy=yxy$ $\Rightarrow$ $xy=(yx)^2$, and the last identity has the form \eqref{L2.1}.
                  \end{itemize}
            \end{enumerate}
         \item $|\mathcal{U}|=|\mathcal{V}|=3$. If both $\mathcal{U}$ and $\mathcal{V}$ contain a square of any variable, then as above we obtain a contradiction with \eqref{IN}. Therefore, at least one of the words $\mathcal{U}$ and $\mathcal{V}$ does not contain any square. Assume that $\mathcal{U}$ does not contain any square; then $\mathcal{U}=x_{i_1}x_{i_2}x_{i_1}$. If $\mathcal{V}$ does not contain squares, then $\mathcal{V}=x_{i_1}x_{i_2}x_{i_1}$ or $\mathcal{V}=x_{i_2}x_{i_1}x_{i_2}$. But both identities $xyx=xyx$, $xyx=yxy$ hold in the semigroup $B$; therefore, $\mathcal{V}$ contains a square, and $\mathcal{V}$ is equal to $x_{i_1}^2x_{i_2}$ or $x_{i_2}^2x_{i_1}$ or $x_{i_1}x_{i_2}^2$ or $x_{i_2}x_{i_1}^2$ or $x_{i_1}^3$ or $x_{i_2}^3$. In the first 4 cases the identity $\mathcal{U}=\mathcal{V}$ has the form \eqref{L2.2} or \eqref{L2.3}. If $\mathcal{V}=x_{i_1}^3$, we obtain the identity $xyx=x^3$; replacing $x$ by $x^4$ we have: $x^3=x^6$ and $xyx=x^3=x^6=xyxxyx$, and this identity has the form \eqref{L2.3}. Similarly, if $\mathcal{V}=x_{i_2}^3$, we obtain the identity $xyx=y^3$; replacing $x$ by $y^3$ we have: $y^3=y^7$ and $xyx=y^3=y^7=xyxxyxy$, and this identity has the form \eqref{L2.3} as well.
       \end{enumerate}
 \end{enumerate}
\end{proof}

\begin{lemma}\label{l3}
  If an identity $\mathcal{U}=\mathcal{V}$ fails in the semigroup $C_{\rho}$ then this identity implies an identity of one of the following forms:
  \begin{subequations}\label{L3}
  \begin{align}
  \label{L3.1}
  x^myx=x^nyx^p \quad (p>1);\\
  \label{L3.2}
  x^myx=PyQyR,
\end{align}
\end{subequations}
where $P,Q,R$ are words or empty symbols and $x,y$ are different variables.
\end{lemma}
\begin{proof}
    Let us consider all possible cases.
    \begin{enumerate}
      \item $\chi(\mathcal{U})\neq\chi(\mathcal{V})$. Then $\mathcal{U}=\mathcal{V}$ $\Rightarrow$ $x^m=PyQ$. From the last identity we deduce: $x^myx=PyQyx$, and this is an identity of the form \eqref{L3.2}.
      \item $\chi(\mathcal{U})=\chi(\mathcal{V})= \{x_1, x_2,\cdots,x_n\}$. Since the identity  $\mathcal{U}=\mathcal{V}$ does not hold in the semigroup $C_\rho$, there exists a mapping $f:\{x_1, x_2,\cdots ,x_n\}\rightarrow C_\rho$ such that \eqref{IN} holds. Consider all possible cases.
          \begin{enumerate}
            \item $(f(x_i)\in\{x,e\}$ for all $i$. If $|\mathcal{U}|>1$ and $|\mathcal{V}|>1$, then $f(\mathcal{U})=f(\mathcal{V})=e$ in contradiction with \eqref{IN}.

                If  $|\mathcal{U}|=1$ and $|\mathcal{V}|=1$, then the identity  $\mathcal{U} =\mathcal{V}$ has the form  $x_i=x_j$ ($x_i$ and $x_j$ are different variable). It is clear that this implies \eqref{L3.1}.

                If $|\mathcal{U}|=1$ and $|\mathcal{V}|>1$ (or \textit{vice versa}), then  $\mathcal{U}=\mathcal{V}$ $\Rightarrow$ $x=x^p$, where $p>1$, and an identity of the form \eqref{L3.1} can be easily deduced from the last identity.
            \item There is an index $i$ such that $(f(x_i)\notin\{x,e\}$. If $x_i$ enters in both words $\mathcal{U}$ and $\mathcal{V}$ at least twice, then $f(\mathcal{U}) = f(\mathcal{V})=0$, because $f(x_i)\in\{y,z,0\}$, and $\{y,z,0\}$ is an ideal of the semigroup $C_\rho$, in which the product of every two elements equals 0. But it contradicts \eqref{IN}. Hence, $x_i$ enters in $\mathcal{U}$ or in  $\mathcal{V}$ only once. Without loss of generality let us assume that $x_i$ enters in $\mathcal{U}$ only once.
                \begin{enumerate}
                  \item $x_i$ enters in $\mathcal{V}$ at least twice. Then the identity $\mathcal{U} =\mathcal{V}$ has the form  $Sx_iT=Px_iQx_iR$ where $x_i\notin\chi(S)\cup\chi(T)$. If $|T|>1$, then $f(x_i)\cdot f(T)=0$, and $f(\mathcal{U})=f(\mathcal{V})=0$ in contradiction with \eqref{IN}. If $|T|=1$, then the substitution $x_i\mapsto y$, $x_j\mapsto x$ for $j\neq i$, gives us an identity of the form \eqref{L3.1}. At last, if $|T|=0$, the same substitution followed by the multiplication on the right by $x$ produces an identity of the form \eqref{L3.1}.
                  \item $x_i$ occurs in $\mathcal{V}$ only once. The identity $\mathcal{U} =\mathcal{V}$ has the form $Sx_iT=Dx_iE$, and $x_i \not\in \chi(S)\cup \chi(T)\cup \chi(D)\cup\chi(E)$. If $|T|>1$, $|E|>1$, then $f(x_i) \cdot f(T)=f(x_i) \cdot f(E)=0$ and $f(\mathcal{U})=f(\mathcal{V})=0$ in contradiction with \eqref{IN}. Therefore, either $|T|=1$ or $|E|=1$. Without loss of generality assume that $|T|=1$. Suppose that $|E|=1$ as well and that $T=E=x_j$ $(x_i \neq x_j)$. If exists an index $k \neq i$ such that $f(x_k) \in \{y,z,0\}$, then $f(\mathcal{U})=f(\mathcal{V})=0$ in contradiction with \eqref{IN}. Thus, $f(x_k)\in\{e,x\}$ for all $k$, $k\neq i$; then $f(S) \in \{e,x\}$, $f(D) \in \{x,e\}$, and $f(S) \cdot f(x_i)=f(D) \cdot f(x_i) \Rightarrow  f(S) \cdot f(x_i) \cdot f(x_j)=f(D) \cdot f(x_i) \cdot f(x_j) \Rightarrow f(\mathcal{U})=f(\mathcal{V})$ in contradiction with \eqref{IN}. Thus $|E|>1$, or $E=x_k$, $T=x_j$, $k\neq i,j$, $j \neq i$. Make the substitution $x_i \mapsto y$, $x_l\mapsto x$ for $l \neq i$ in the first case and the substitution $x_i\mapsto y$, $x_k\mapsto x^2$, $x_l \mapsto x$ for $l \neq i,k$ in the second case. We shall obtain so an identity of the form \eqref{L3.1}.
                \end{enumerate}
          \end{enumerate}
    \end{enumerate}
\end{proof}

The left analog of Lemma \ref{l3} is valid for the semigroup $C_{\lambda}$.

\begin{lemma}\label{l4}
  If an identity $\mathcal{U}=\mathcal{V}$ fails in the semigroup  $D$ then this identity  implies an identity of the form
\begin{equation}\label{L4}
  xy^nx=PxQ,
\end{equation}
where $P,Q\neq\emptyset$ are nonempty words and $x,y$ are different variables.
\end{lemma}
\begin{proof}
  There exists a mapping $f:\{x_1,x_2,\cdots ,x_n\} \rightarrow D$ such that \eqref{IN} holds. If $\chi(\mathcal{U})\neq\chi(\mathcal{V})$, then $\mathcal{U}=\mathcal{V}$ $\Rightarrow$ $y^n=P^*xQ^*$ $\Rightarrow$ $xy^nx=xP^*xQ^*x$, and the last identity has the form \eqref{L4}. Let     $\chi(\mathcal{U})=\chi(\mathcal{V})$. As above we shall look over all possible cases. Without loss of generality we can assume that $|\mathcal{U}|\leq|\mathcal{V}|$.
  \begin{enumerate}
    \item $|\mathcal{U}|<3$.
        \begin{enumerate}
          \item $\mathcal{U}=x_i$. Then  the identity $\mathcal{U}=\mathcal{V}$ takes the form  $x_i=x_i^m$ $(m>1)$. Set $x_i =xyx$ and obtain an identity $xyx=(xyx)^m$ of the form \eqref{L4}.
          \item $\mathcal{U}=x_i^2$. In this case  the identity $\mathcal{U}=\mathcal{V}$ takes the form $x_i^2=x_i^m$ $(m>2)$; but this identity holds in the semigroup $D$ in contradiction with the hypothesis.
          \item $\mathcal{U}=x_ix_j$ $(x_i\neq x_j)$. If $|\mathcal{V}|=2$, i.e., $\mathcal{V}=x_jx_i$ then the identity  $\mathcal{U}=\mathcal{V}$ $\Rightarrow$ $xy=yx$ implies the identity $xyx=yx^2$ of the form \eqref{L4}. If $|\mathcal{V}|\geq3$ then according to Propositions \eqref{p1} and \eqref{p2}, $\mathcal{U}= \mathcal{V}$ implies the identity $xyx=(xy)^{n+1}x$ of form \eqref{L4}.
        \end{enumerate}
    \item $|\mathcal{U}|\geq3$, $|\mathcal{V}|\geq3$. Let $\mathcal{U}=aPb$, $\mathcal{V}=cQd$, where $a,b,c,d$ are variables and $P,Q$ are nonempty words $(a,b,c,d\in\{x_1, x_2,\cdots ,x_n\})$.
        \begin{enumerate}
          \item $\{a,b\} \subset \chi(P)$, $\{c,d\} \subset \chi(Q)$. Then for any mapping $f:\chi(\mathcal{U})\rightarrow\chi(\mathcal{V})$ we have:
              $$f(\mathcal{U})=f(aPb)=f(\mathcal{V})=f(cQd)=\begin{cases}
                e, &\text{ if } f(x_i)=e \text{ for all }i,\\
                0, &\text{ otherwise; }
              \end{cases}$$
              Thus this case is impossible.
          \item $a\notin\chi(P)$ or $b\notin\chi(Q)$ or $c\notin\chi(Q)$ or  $d\notin\chi(Q)$. All these cases can be considered analogously; examine the first one. There exist the following possibilities:
              \begin{enumerate}
                \item $a \neq c$. If  $a\in\chi(Q)$, then sending $a$ into $x$, $b$ into $y$ (if $b \neq a$) and all letters from $P$ into $y$ we obtain an identity of the form \eqref{L4} or an identity $xy^n=P^*xQ^*$ \ ($P^*\neq0$), which implies the identity $xy^nx=P^*xQ^*x$ of the form \eqref{L4}. If $a\notin\chi(Q)$, then $a=d$ because  $\chi(\mathcal{U})=\chi(\mathcal{V})$, and the identity takes the form $aPb=cQa$, where $a\notin\chi(P)\cup\chi(Q)$ and $a\neq c$. If $a=b$ then $\mathcal{U}=\mathcal{V}$ $\Rightarrow$ $aPa=cQa$ $\Rightarrow$ $ac^na=c^ma$ $\Rightarrow$ $ac^ma=a^2c^na$, and the last identity is an identity of the form \eqref{L4}. If $a \neq b$ then the substitution $b\mapsto c$ transforms the identity $\mathcal{U}=\mathcal{V}$ into the identity $aP^*c=cQ^*a$, in which $a\notin\chi(P^*)$ and $a \notin\chi(Q^*)$; from this identity follows the identity $ac^n=c^ma$ which implies the identity $ac^na=c^ma^2$ of the form \eqref{L4} $(n,m\geq1$.
                \item $a=c$. In this case our identity takes the form $aPb=aQd$. Assume that $a=b=d$. If $a \in \chi(Q)$, then $\mathcal{U}=\mathcal{V}$ is an identity of the form \eqref{L4}. If $a\notin\chi(Q)$ then tt follows from the definition of the semigroup $D$ that $f(P)=e$ only if $f(x_i)=e$ for each letter $x_i\in\chi(P)$. Therefore, if $f(P)=e$ then $f(x_i)=e$ for any letter $x_i\neq a$. Since $a\notin\chi(Q)$, we have also $f(Q)=e$. Hence, $f(P)=e\Rightarrow f(Q)=e$ and the same argument proves that $f(Q)=e\Rightarrow f(P)=e$. Further,
                    $$f(\mathcal{U})=f(aPa)=\begin{cases}
                    xex, &\text{if }  a=x \text{ and } f(P)=e; \\
                    e, &\text{if } a=e \text{ and } f(P)=e; \\
                    0 &\text{otherwise,}\end{cases}$$
                    and the same formulas, in which $P$ is replaced with $Q$, give the value of $f(\mathcal{V})$.  But in our case $f(P)=e$ $\Leftrightarrow$ $f(Q)=e$, and, consequently, $f(\mathcal{U})=f(\mathcal{V})$ in contradiction with \eqref{IN}.

                     Assume that $a=b\neq d$. Then if $a\in\chi(Q)$, then, obviously, $\mathcal{U}=\mathcal{V}$ implies an identity of the form \eqref{L4}. Otherwise, $\mathcal{U}=\mathcal{V} \Rightarrow ad^ma=ad^n$ $\Rightarrow$ $ad^na=ad^ma^2$, and this is an identity of the form \eqref{L4}.

                     Assume that $a\neq b$ and $b\neq d$. Then $\mathcal{U}=\mathcal{V}$ $\Rightarrow$ $aPb=aQd$ $\Rightarrow$ $ab^n=aQa$ $\Rightarrow$ $ab^na=aQa^2$, and this is an identity of the form \eqref{L4}.

                     Assume that $a\neq b=d$. Then the identity $\mathcal{U}=\mathcal{V}$ takes the form $aPb=aQb$.

                     If $a\in\chi(Q)$, then $\mathcal{U}=\mathcal{V}$ $\Rightarrow$ $ab^n=aQ^*b$ $\Rightarrow$ $ab^na=aQ^*ba$ $\Rightarrow$ \eqref{L4}.

                     If $a\notin\chi(Q)$ and $b\in\chi(P)\backslash\chi(Q)$, then by setting $a\mapsto x;b\mapsto x; f(Q)=y^n$ we obtain the identity $xP^*(x)x=xy^nx$ which implies an identity of the form \eqref{L4}.

                     If $a\notin\chi(Q)$ and $b\in\chi(Q)\backslash\chi(P)$, then by setting $a\mapsto x; b\mapsto x; f(P)=y^n$ we obtain $xy^nx=xQ^*(x)x$ which implies an identity of the form \eqref{L4}.

                     If $a\notin\chi(Q)$ and $b\in\chi(Q)\cap\chi(P)$, then
                     $$f(\mathcal{U})=f(aPb)=\begin{cases}
                     e,  &\text{if }  f(x_i)=e \text{ for all }i;\\
                     xe,  &\text{if } f(a)=x,\, f(b)=e,\, f(P)=e; \\
                     0 &  \text{otherwise.}
                     \end{cases}$$
                     The same formulas (with $P$ replaced by $Q$) give the value of $f(\mathcal{V})$. Remark that in this case $\chi(P)=\chi(Q)$, therefore, $f(P)=e$ $\Leftrightarrow$ $f(Q)=e$, i.e., $f(\mathcal{U})=f(\mathcal{V})$ in contradiction with \eqref{IN}. Thus, this situation is impossible.

                     If $a\notin\chi(Q)$ and $b\notin\chi(Q)\cup\chi(P)$, then $$f(\mathcal{U})=f(aPb)=\begin{cases}
                     e, & \text{if } f(x_i)=e \text{ for all }i;\\
                     xe, & \text{if } a=x,\, b=e,\, f(P)=e; \\
                     ex, & \text{if } a=e,\, b=x,\, f(P)=e; \\
                     xex, & \text{if } a=x,\, b=x,\, f(P)=e; \\
                     0  & \text{otherwise.}
                     \end{cases}$$
                     The similar formulas give the value of $f(\mathcal{V})$.
                     Remark that in this case $\chi(P)=\chi(Q)$, therefore, $f(P)=e$ $\Leftrightarrow$ $f(Q)=e$, i.e., $f(\mathcal{U})=f(\mathcal{V})$ in contradiction with \eqref{IN}. Thus, this situation is impossible, too.

                     So, in all possible cases we have obtained an identity of the form \eqref{L4}.\end{enumerate}\end{enumerate}\end{enumerate}\end{proof}
Recall that
$$K_n=\{x,y \mid y^2=y^{2+n};x^2=0; yxy=0; xy^qx=0 (q=2,\cdots ,n);xyx=xy^{n+1}x\}.$$
It follows from this definition that the semigroup $K_n$ is finite; more precisely,
$$K_n=\{x,y^k,xy^k, y^kx, xyx, 0 \mid 1 \leq k \leq n+1\}.$$
Indeed, each element of $K_n$ is a word in $\{x,y\}$-alphabeth. The relation $yxy=0$ shows that every nonzero element of $K_n$ is equal to $x^iy^j$ or $y^jx^i$ or $x^iy^jx^k$; the relations $x^2=0$ and $xy^qx=0$ show that $xyx$ is
the only nonzero element of this semigroup in which $x$ occurs more than once and the relation $y^2=y^{2+n}$ shows that we can consider only the words in which $y$ enters not more than $n+1$ times. Remark, that all elements of $K_n$ from the above list are different.

\begin{lemma}\label{l5}
  If an identity $\mathcal{U}=\mathcal{V}$ fails in the semigroup  $K_n$ then this identity implies an identity of one of the following forms:
\begin{subequations}\label{L5}
\begin{align}
\label{L5.1}
x=PyQ;\\
\label{L5.2}
xy^{m_{1}}x=PxQ;\\
\label{L5.3}
xy^{m_{1}}x=xy^{m_{2}}x,
\end{align}
\end{subequations}
where $P,Q$ are words or empty symbols and $m_{1}\equiv1$ (mod $n$), $m_{2}\not\equiv1$ (mod $n$) and x,y are different variables.
\end{lemma}
\begin{proof}
  If $\chi(\mathcal{U}) \neq\chi(\mathcal{V})$, then the identity $\mathcal{U}=\mathcal{V}$ implies  an identity of the form \eqref{L5.1}. Let  $\chi(\mathcal{U})=\chi(\mathcal{V})=\{x_1,x_2,\cdots,x_n\}$. Since the identity $\mathcal{U}=\mathcal{V}$ does not hold in $K_n$, there exists a mapping $f:\{x_1,x_2,\cdots ,x_n\} \rightarrow K_n$ such that \eqref{IN} holds. As in the proofs of the precedent Lemmas, we shall examine all cases item-by-item.
  \begin{enumerate}
    \item $f(x_i)=x$ for all $i$. Then \eqref{IN} $\Rightarrow$ $|\mathcal{U}|=1$ or $|\mathcal{V}|=1$ $\Rightarrow$ $x=x^m$ ($m>1$) $\Rightarrow$ $xyx=xy^mx$, and the last identity has the form \eqref{L5.2}.
    \item There exists $i$ such that $f(x_i)$ is divisible by  $y$ and $f(x_j)$ is not divisible by $y$ for any $j\neq i$. This means that $f(x_j)=x$. We can assume without loss of generality that
        \begin{equation}\label{IN2}
          f(\mathcal{U})\neq0.
        \end{equation}
        If the letter $x_i$ enter in $\mathcal{U}$ more than once, then by the definition of $K_n$ the word $\mathcal{U}$ has one of the following forms:
        $$\dagger\quad x_i^qx_j, \quad \text{or} \quad \ddagger \quad x_jx_i^q\,\,\,\,(q>2),\quad\text{or} \quad \star \quad x_jx_i^mx_k \,\,(m\equiv 1\!\!\!\!\pmod{n})\,.$$
        If $\mathcal{U}$ has the form $\dagger$ and \eqref{L5.2} does not hold, then $\mathcal{V}$ necessarily has the form of $x_i^lx_j$ ($l>2$). Then $q \not\equiv l\pmod{n}$, and it follows that the identity $\mathcal{U}=\mathcal{V}$ implies \eqref{L5.3}. The case $\ddagger$ can be studied  in the same way. If we have the case $\star$ and \eqref{L5.2} does not hold then $\mathcal{V}$ has necessarily the form  $x_jx_i^lx_k$ or the form $x_kx_i^lx_j$, where $m \not\equiv l\pmod{n}$, i.e., the identity $\mathcal{U}=\mathcal{V}$ implies \eqref{L5.3}. Thus, $x_i$ enter in $\mathcal{U}$ only once. Further, for $j \neq i,k$ the letter $x_j$ can not follow immediately the letter $x_k$ because otherwise $f(\mathcal{U})=0$ in contradiction with \eqref{IN2}. Consider all remaining cases.
        \begin{enumerate}
          \item $\mathcal{U}=x_jx_ix_j$, $x_j \neq x_i$.
            \begin{enumerate}
              \item The letter $x_j$ enters in the  word $\mathcal{V}$ only once. Then $\mathcal{V}=x_jx_i^k$ or $\mathcal{V}=x_i^kx_j$ or $\mathcal{V}=x_i^kx_jx_i^l$ for some $k,l\geq1$. Hence, $\mathcal{U} =\mathcal{V}$ implies $xyx=xy^k$ or $xyx=y^kx$ or$xyx=y^kxy^l$. It remains to remark, that each of these identities implies an identity of the form \eqref{L5.2}.
              \item $x_j$ occurs in $\mathcal{V}$ at least twice and either the first or the last letter of the word $\mathcal{V}$ is not equal to $x_j$. Then $\mathcal{U} =\mathcal{V}$ implies \eqref{L5.2}.
              \item $\mathcal{V}=x_jP^*x_j$ and $x_j\notin \chi (P^*)$. In this case  $\mathcal{V}=x_jx_i^mx_j$. But $f(\mathcal{U})\neq f(\mathcal{V})$, therefore, $m>1$.

                  If $f(x_i)$ is not divisible by $x$, then $f(\mathcal{U})=f(x_j)f(x_i)f(x_j)=xf(x_i)x$ is divisible by $x^2=0$ or by $yxy=0$, i.e., $f(\mathcal{U})=0$ in contradiction with \eqref{IN2}. So, $f(x_i)$ is not divisible by $x$. This means that  $f(x_i)=y^k$ for some $k\geq1$, i.e.,  $f(\mathcal{U})=xy^{km}x$.

                  If $m\equiv 1\pmod{n}$ then $km-k=k(m-1)\equiv 0\pmod{n}$ and $xy^kx=xy^{km}x$ by the definition of the semigroup $K_n$; it follows that $f(\mathcal{U})=f(\mathcal{V})$ in contradiction with $(*)$. Thus, $m \not\equiv 1\pmod{n}$, and $\mathcal{U}=\mathcal{V}$ implies the identity $xyx=xy^mx$ of the form \eqref{L5.3}.
            \end{enumerate}
          \item $\mathcal{U}=x_jx_ix_k$ $(x_j\neq x_k; x_j \neq x_i; x_i\neq x_k)$.
              \begin{enumerate}
                \item $x_j$ or $x_k$ occurs inside the word $\mathcal{V}$, i.e., it is neither the first nor the last letter of the word $\mathcal{V}$. Then $\mathcal{U} =\mathcal{V} \Rightarrow xyx=PxQ$, where $ P\neq\emptyset$,  $Q\neq\emptyset$, and this is an identity of the form \eqref{L5.2}.
                \item Neither $x_j$ nor $x_k$ occurs inside $\mathcal{V}$. Then $\mathcal{V}=x_jx_i^mx_k$ or $\mathcal{V}=x_kx_i^mx_j$ and $f(\mathcal{U})=xf(x_i)x$. Applying again arguments of above cases we prove that $m\not\equiv1\pmod{n}$. Then $\mathcal{U}=\mathcal{V}$ implies the identity $xyx=xy^mx$ of the form \eqref{L5.3}.
              \end{enumerate}
          \item $\mathcal{U}=x_ix_j$ or $\mathcal{U}=x_jx_i$ or $\mathcal{U}=x_i$. Since the identity $\chi(\mathcal{U}=\chi(\mathcal{V})$ is not a tautology, it implies the identity $xy=yx$ or the identity $xy=\mathcal{V}$ with $|\mathcal{V}|\geq3$ and $\chi(\mathcal{V})=\{x,y\}$. Obviously,  each of them implies an identity of the form \eqref{L5.2}.
        \end{enumerate}
  \end{enumerate}
\end{proof}

\begin{lemma}\label{l6}
  If an identity $\mathcal{U}=\mathcal{V}$ fails in the semigroup  $F_\lambda $ then this identity  implies an identity of the form
  \begin{equation}\label{L6}
    xyP=x^{k}yQ
  \end{equation}
\end{lemma}
where $P,Q$ are words or empty symbols, $x,y$ are different variables and $k>1$.
\begin{proof}
    Let  $\mathcal{U} =\mathcal{V}$ be an identity which does not hold in $F_\lambda$; let us study all possible cases. Without loss of generality assume that $|\mathcal{U}|\leq|\mathcal{V}|$.
    \begin{enumerate}
      \item $|\mathcal{U}|=1${\rm and }$|\mathcal{V}|=1$. As the identity  $\mathcal{U}=\mathcal{V}$ is not a tautology, $\mathcal{U}=\mathcal{V}$ takes the form of $x=y$, i.e., we come to an identity of the form \eqref{L6}.
      \item $|\mathcal{U}|=1${\rm  and }$|\mathcal{V}|>1$. $\mathcal{U}=\mathcal{V}$ $\Rightarrow$ $x=x^k$ ($k>1$) $\Rightarrow$ $xy=x^ky$, and the last identity is an identity of the form \eqref{L6}.
      \item $|\mathcal{U}|>1${\rm  and }$|\mathcal{V}|>1$. If $l_{1}(\mathcal{U})l_{2}(\mathcal{U})=l_{1}(\mathcal{V})l_{2}(\mathcal{V})$ then $\mathcal{U} =\mathcal{V}$ is fulfilled in $F_\lambda$ by the definition of this semigroup. Let $l_{1}(\mathcal{U})l_{2}(\mathcal{U})\neq l_{1}(\mathcal{V})l_{2}(\mathcal{V})$. Then the following cases are possible.
          \begin{enumerate}
            \item $\mathcal{U}=xyP $, $\mathcal{V}=xzQ$. The substitution $z\mapsto x^{k-1}y$ transforms the identity $\mathcal{U}=\mathcal{V}$ into an identity of the form \eqref{L6}.
            \item $\mathcal{U}=x^2P$, $\mathcal{V}=xyQ$. Replace all variable except $x$ by $y$.
            \item $\mathcal{U}=yP $,  $\mathcal{V}=zQ $. The identity  $\mathcal{U}=\mathcal{V}$ implies the identity $xyP=xzQ$, and we have seen in previous cases that the last identity implies an identity of the form \eqref{L6}.
          \end{enumerate}
    \end{enumerate}\end{proof}

The right analog of Lemma \eqref{l6} is valid for the semigroup $F_{\rho}$.

\begin{lemma}\label{l7}
  If an identity $\mathcal{U}=\mathcal{V}$ fails in the semigroup  $W_\lambda$ then this identity  implies an identity of one of the following  forms:
  \begin{subequations}\label{L7}
\begin{align}
\label{L7.1}
(xy)^{n+1}=Py^{2}Q;\\
\label{L7.2}
(xy)^{n+1}=Px^{2}Q;\\
\label{L7.3}
(ax)^{k_{1}}(ay)^{k_{2}}\cdots(ax)^{k_{s}}=(ay)^{l_{1}}(ax)^{l_{2}}\cdots(ay)^{l_{t-1}}(ax)^{l_{t}};\\
\label{L7.4}
(ax)^{k_{1}}=(ax)^{l_{1}}a
\end{align}
\end{subequations}
or the identity \eqref{L5.1}, where $P,Q$ are words or empty symbols, $k_1,\cdots,k_s,l_1,\cdots,l_t\geq1$,
and $a,x,y$ are different variables.
\end{lemma}
\begin{proof}
  Since  $\mathcal{U}=\mathcal{V}$  does not hold in the semigroup $W_\lambda$, there exists a mapping $f:\{x_1,x_2,\cdots ,x_n\} \rightarrow  W_\lambda$, such that  $f(\mathcal{U}) \neq f(\mathcal{V})$. For each $j$ the element $f(x_j)$ is a word in the alphabet $\{a,x,y\}$; let $X_j$ be the word obtained from $f(x_j)$ by replacing letters $a, \,x,\,y$ with  letters $A,\,X,\,Y$. We consider $A$, $X$, $Y$ as independent variables. Let $\mathcal{U}_1$, $\mathcal{V}_1$ be words obtained from $\mathcal{U}$, $\mathcal{V}$ by replacing $x_j$ with the words $X_j$. Then the identity $\mathcal{U}_1=\mathcal{V}_1$ follows from the identity $\mathcal{U}=\mathcal{V}$ and it is not fulfilled in $W_\lambda$ because
  \begin{equation}\label{IN3}
    g(\mathcal{U}_1)=f(\mathcal{U})\neq f(\mathcal{V})=g(\mathcal{V}_1),
  \end{equation}
  where $g$ is the mapping $\{A,X,Y\}\rightarrow W_\lambda $ defined by the formulae $g(A)=a$, $g(X)=x$, $g(Y)=y$.

  Without loss of generality assume that
  \begin{equation}\label{IN4}
    g(\mathcal{U}_1)\neq0.
  \end{equation}
  Consider all possible cases.
  \begin{enumerate}
    \item $\chi(\mathcal{U}_1) \neq\chi(\mathcal{V}_1)$. Then $\mathcal{U}_1=\mathcal{V}_1$ implies \eqref{L5.1}.
    \item $\chi(\mathcal{U}_1) =\chi(\mathcal{V}_1)\subseteq \{A,X,Y\}$.
        \begin{enumerate}
          \item The first letters of the words $\mathcal{U}_1$ and $\mathcal{V}_1$ are different. Then $\mathcal{U}_1=\mathcal{V}_1\Rightarrow xP(x,y)x=yQ(x,y)x$, and the substitution $x\mapsto ax$,  $y \mapsto ay$ produces an identity of the form \eqref{L7.3}.
          \item The words $\mathcal{U}_1$ and $\mathcal{V}_1$ begin with the same letter $T\in \{A,X,Y\}$. Let $|\mathcal{U}_1|=1$. Since $\mathcal{U}_1=\mathcal{V}_1$ is not a tautology, the identity $\mathcal{U}_1=\mathcal{V}_1$ implies the identity $x=x^k$ of the form \eqref{L7.2} ($k=|\mathcal{V}_1|>1$).

              If $|\mathcal{U}_1|>1$ then $g(\mathcal{U}_1)$ is divisible by $a$. Otherwise $g(\mathcal{U}_1)$should be divisible by $x^2$ or by $y^2$ or by $xy$ or by $yx$; in all cases $g(\mathcal{U}_1)=0 $ in contradiction with \eqref{IN4}.
              \begin{enumerate}
                \item $T=A$. In this case $\mathcal{U}_1=AT_1AT_2\cdots AT_{k-1}AT_k$, where $T_1,T_2,\cdots ,T_{k-1} \in \{X,Y\}$, $T_k \in\{X,Y\}$ or $T_k$ is the empty word, $g(\mathcal{V}_1)=AP_1AP_2\cdots AP_n$, where $P_i$ are words over $\{X,Y\}$ (maybe empty).

                    Assume that $g(\mathcal{V}_1)\neq0$. It is possible only if $P_1,\cdots P_{n-1} \in \{X,Y\}$, $P_n \in \{X,Y\}$ or $P_n=\emptyset$.   If exactly one of the words $T_k$ and $P_n$ is empty then the identity $\mathcal{U}_1=\mathcal{V}_1$ implies an identity of the form \eqref{L7.4}. Let $T_1=P_1$;  if both symbols $T_k$, $P_n$ are empty, then it follows from the description of the semigroup $W_\lambda $ that $g(\mathcal{U}_1)=ag(T_1)a=ag(P_1)a=g(\mathcal{V}_1)$, and this equality contradicts \eqref{IN3}. If  $T_1=P_1$, $T_k,P_n\in \{X,Y\}$, then $g(\mathcal{U}_1)=ag(T_1)=ag(P_1)=g(\mathcal{V}_1)$ in contradiction with \eqref{IN3}.

                    It remains to study the situation, in which $T_1\neq P_1$ and either both words $T_k$ and $P_n$ are empty or both words are not empty.  The second letters $T_1$, $P_1$ of the words $\mathcal{U}_1$, $\mathcal{V}_1$ are different and belong to the set $\{X,Y\}$; therefore the second letter of one of the words $\mathcal{U}_1$, $\mathcal{V}_1$ is $X$, and the second letter of another of these words is $Y$. Hence, the identity $\mathcal{U}_1=\mathcal{V}_1$ or the identity $\mathcal{U}_1X=\mathcal{V}_1X$ is an identity of the form \eqref{L7.3}.

                    Assume that $g(\mathcal{V}_1)=0$. This means that  $g(\mathcal{V}_1)$ contains a square of one of the letters $a,x,y$ or $xy$ or $yx$. Replacing $X$, $Y$ by $x$ and $A$ by $y$ we come either to the identity $(xy)^{n+1}=Px^2Q$ of the form \eqref{L7.1} or to the identity $(xy)^{n+1}=Py^2Q$ of the form \eqref{L7.2}.
                \item $T=X$ or $T=Y$. In this case $\mathcal{U}_1=TAT_1\cdots T_{k-1}AT_k$, where $T_1,\cdots ,T_{k-1}\in\{X,Y\}$, $T_k\in \{X,Y\}$ or $T_k=\emptyset$, because otherwise $g(\mathcal{U}_1)=0$, and $\mathcal{V}_1=TP_0AP_1\cdots AP_{n-1}AP_n$, where  $P_0,P_1,\cdots ,P_n$ are words over $\{X,Y\}$ or empty symbols.

                    Assume that $g(\mathcal{V}_1)\neq0$. Then $P_0=\emptyset$ and $P_1,\cdots , P_{n-1} \in \{X,Y\}$, $P_n\in \{X,Y\}$ or $P_n=\emptyset$. If exactly one of the symbols $T_k$, $P_n$ is empty, then the identity $\mathcal{U}_1=\mathcal{V}_1$ implies an identity of the form \eqref{L7.4}. If both symbols $T_k$, $P_n$ are empty, then it follows from the description of the semigroup $W_\lambda $ that $g(\mathcal{U}_1)=g(T)a=g(\mathcal{V}_1)$, and this equality contradicts $(*)$. If $T_k,P_n\in \{X,Y\}$, then $g(\mathcal{U}_1)=g(T)ag(T)=g(\mathcal{V}_1)$ in contradiction with \eqref{IN3}.

                    Assume that $g(\mathcal{V}_1)=0$. This means that  $\mathcal{V}_1$ contains a square of one of the letters $A,X,Y$ or $XY$ or $YX$. Replace $A$ with $y$ and $X$, $Y$ with $x$. We shall obtain the identity $(xy)^{n+1}=Px^2Q$ of the form \eqref{L7.1} or the identity $(xy)^{n+1}=Py^2Q$ of the form \eqref{L7.2}.
              \end{enumerate}
        \end{enumerate}
  \end{enumerate}
\end{proof}

\begin{lemma}\label{l8}
  If an identity $\mathcal{U}=\mathcal{V}$ fails in the semigroup  $N_3$ then this identity implies an identity of the form \eqref{1}.
\end{lemma}
\begin{proof}
  Examine all cases item by item.
  \begin{enumerate}
    \item $|\mathcal{U}|\geq3$ and $|\mathcal{V}|\geq3$. Then $\mathcal{U}=\mathcal{V}$ holds in $N_3$ in contradiction with the hypothesis.
    \item $|\mathcal{U}|<3$ or $|\mathcal{V}|<3$. If $|\mathcal{V}|<|\mathcal{U}|$ or $|\mathcal{V}|>|\mathcal{U}|$ then the $\mathcal{U}=\mathcal{V}\longrightarrow x^2=x^{n+2}$ where $n$ is equal to $\pm (|\mathcal{U}|-|\mathcal{V}|)$. Let $|\mathcal{V}|=|\mathcal{U}|$; if the identity $\mathcal{U}=\mathcal{V}$ is balanced (i.e., the multiplicity of occurrence of each letter in the word $\mathcal{U}$ is the same as the multiplicity of occurrence of this letter in the word $\mathcal{V}$), then this identity holds in the semigroup $N_3$ in contradiction with the assumption. Therefore, $\mathcal{U}=\mathcal{V}$ is not a balanced identity.

        There is only one unbalanced identity of degree 1 and only 7 unbalance identity of degree 2:

        $$x=y,\,\,\, x^2=y^2,\,\,\, x^2=xy,\,\,\, x^2=yx,\,\,\, xy=xz,\,\,\, xy=zx,\,\,\,xy=zy,\,\,\, xy=zt.$$
        The substitution $x\mapsto x,\,y\mapsto x^2$ for first four of this identities and the substitution $x,y\mapsto x$, $z,t\mapsto x^2$ makes them into identities of the form \eqref{1}.
  \end{enumerate}
\end{proof}
\begin{lemma}\label{l9}
    Let the identity \eqref{1} and an identity of the form \eqref{L2.1} hold in in the variety $V$. Then the identity \eqref{2} holds in the variety $V$ as well.
\end{lemma}
\begin{proof}
  Let the identity $xyx=P(yx)^2Q$ or the identity $xyx=P(xy)^2Q$ hold in $V$. It follows from \eqref{1} that $(yx)^2=(yx)^{2+n}$ and $(xy)^2=(xy)^{2+n}$; therefore, $xyx$ is divisible by $(yx)^3=yxyxyx$ or $(xy)^3=xyxyxy$. In both cases $xyx$ is divisible by $yxyxy$ and the elements $xyx$,  $yxyxy$ belong to the same $\mathcal{J}$-class. Let $t=(xyx)\cdot (yxyxy)\cdot (xyx)$. Since $xyx$ is divisible by $(xy)^{2+nk}$ for all $k$, $xyx$ is divisible by $t$, and $xyx\mathcal{J} t$.

  Let $S$ be the free semigroup of the variety $V$ with free generators $x$, $y$ and let $K_{xyx}$ be the principal factor of the element $xyx$. Since $z=yxyxy\mathcal{J}xyx$ and $t\mathcal{J}xyx$, the elements $z$, $t$ belong to $ K_{xyx}$ and $z \neq0$, $t \neq0$ in $K_{xyx}$. Since $xyx=(xyx)\cdot z\cdot (xyx)$, the semigroup $K_{xyx}$ is not a semigroup with zero multiplication. Any principal factor of a semigroup is a 0-simple semigroup. It follows from \eqref{1} that $K_{xyx}$ is periodic. According to the Mann's theorem, a 0-simple periodic semigroup is either a semigroup with $0$ multiplication or  a completely $0$-simple semigroup. So, $K_{xyx}$ is a completely 0-simple semigroup. Since $t=(xyx) \cdot z \cdot (xyx) \neq0$ in $K_{xyx}$, the elements $xyx$ and $t$ are contained in the same $\mathcal{H}$-class of the semigroup $K_{xyx}$. Therefore, the sets of their right unities and the sets of their left unities coincide. As the element $(xy)^n$ is a left unity for $t$, it is a left unity for $xyx$ too. Thus, $xyx=(xy)^{n+1}x$.
\end{proof}

\begin{lemma}\label{l10}
  The identities \eqref{1} and \eqref{2} hold in a semigroup variety $V$ for some $n$ if and only if $V$ does not contain any of indicator Burnside semigroups $(1)$---$(7)$.
\end{lemma}
\begin{proof}\textbf{Necessity.} Assume that the identities \eqref{1},
    \eqref{2} hold in $V$. Show that the semigroups $(1)$---$(7)$ are not contained in $V$.
    \begin{enumerate}
      \item The identity \eqref{2} does not hold in the semigroup $A$: $eye=z; (ey)^{n+1}e=0$.
      \item The identity \eqref{2} does not hold in the semigroup $B$: $xyx=c; (xy)^{n+1}x=0$.
      \item The identity \eqref{2} does not hold in the semigroup $C_\lambda$: $xyx=z; (xy)^{n+1}x=0$.
      \item The identity \eqref{2} does not hold in the semigroup $C_\rho$ (as in the case $(3))$.
      \item The identity \eqref{1} does not hold in the semigroup $N_3$: $a^2=b; a^{2+n}=0$.
      \item The identity \eqref{2} does not hold in the semigroup $D$: $xex\neq0; (xe)^{n+1}x=0$.
      \item For every $m$ the identity \eqref{2} does not hold in the semigroup $K_m$: $$xyx \neq0; (xy)^{n+1}x=x(yxy)(xy)^{n-1}x=0.$$
    \end{enumerate}
  \textbf{Sufficiency.} The semigroup $N_3\notin V$ is not contained in $V$; therefore, there is an identity of $V$ which does not hold in $N_3$, and by Lemma \ref{l8} there exists a positive integer $n$ such that the identity $x^2=x^{2+n}$ holds in $V$. Similarly, since  $B\notin V$, there is an identity of $V$ which does not hold in $B$. According to Lemma \ref{l2}, this identity implies an identity of one of the forms \eqref{L2.1}, \eqref{L2.2} or \eqref{L2.3}. Consider these three cases separately.
  \begin{enumerate}
    \item An identity of the form \eqref{L2.1} holds in $V$. By Lemma \ref{l9} an identity of the form \eqref{2} holds in $V$.
    \item An identity of the form \eqref{L2.2} holds in $V$. Assume that $xyx=Py^2xQ$ is an identity of $V$ (the case of the identity $xyx=Pxy^2Q$ can be studied in the same way).
        \begin{enumerate}
          \item $x \notin\chi(P)$. Then an identity of the form $xyx=y^mxP^*$ $(m\geq 2)$ is fulfilled in $V$. Since $y^2=y^{2+n}$ in $V$, we have for all $k$:
              $$xyx=y^mxP^*=y^{m+nk}xP^*=y^{nk}xyx.$$
              Replacing $x$ with $y$ and $y $ with $x$, we obtain the identity $yxy=x^{nk}yxy$. Therefore,
              $$xyx=y^{nk-1}\cdot(yxy)\cdot x=y^{nk-1}x^{nk}yxyx=y^{nk-1}x^{nk}(yx)^2$$
              is an identity of the variety $V$. It follows now from Lemma \ref{l9} that the identity $xyx=(xy)^{n+1}x$ of the form \eqref{2} holds in $V$.
          \item $x \in \chi(P)$.
            \begin{enumerate}
              \item $Q\neq\emptyset$ and $y\notin\chi(Q)$. The identity $xyx=Py^2xQ$ implies the identity $xyx=Py^2x^m$ with $m>1$. It follows: $xyx=Py^2x^m=Py^2x^{m+n}=xyx^{1+n}$. Replacing $y$ with $yx$ in the identity $xyx=Py^2xQ$, we obtain the identity $xyx^2=P(yx)^2xQ$. Hence, $xyx=xyx^{1+n}=xyx^2x^{n-1}=P(yx)^2xQx^{n-1}$ is an identity of $V$. By Lemma \ref{l9} the identity $xyx=(xy)^{n+1}x$ of the form \eqref{2} holds in $V$.
              \item $Q\neq\emptyset$ and $y\in\chi(Q)$. Our identity takes the form
                  \begin{equation}\label{FORMi}
                    xyx=P^*xQ^*y^2xSyT,
                  \end{equation}
                where $y\notin\chi(S)$.  Replace if needed all variables except $x$, $y$ with $x$; then $S=x^{k-1}$, $k\geq 1$, and we see that $xyx$ is divisible by $y^2x^ky$. Replacing $x$ with $y$ and $y$ with $x^k$, we obtain that $yx^ky$ is divisible by $x^{2k} y^kx^k$. Since the relation of divisibility is transitive, $xyx$ is divisible by $x^2y^kx$.

                If $k=1$ then substituting in \eqref{FORMi} the variable $y$ for $xy$ we conclude that $x^2yx$ is divisible by $(xy)^2$, and therefore, $xyx$ is divisible by $(xy)^2$. By Lemma \ref{l9}, the identity \eqref{1} is fulfilled in $V$.

                Let $k>1$. Replacing $y$ with $xy^k$ we conclude that $x^2y^kx$ is divisible by $x^2(xy^k)^kx$, and consequently $xyx$ is divisible by $(xy^k)^2$ (because $k>1$). But $(xy^k)^2=xy^kxy^k$, therefore, $xyx$ is divisible by $y^2xy$. Changing the variables $x$, $y$ respectively into $x$, $xy$ and $y$, $x$, we obtain that $x^2yx$ is divisible by $(xy)^2$ and $yxy$ is divisible by $x^2yx$. Thus, $xyx$ is divisible by $y^2xy$ and consequently by $yxy$, $yxy$ is divisible by $x^2yx$ and $x^2yx$ is divisible by $(xy)^2$; it follows that $xyx$ is divisible by $(xy)^2$, and Lemma \ref{l9} guarantees that the identity $xyx=(xy)^{n+1}x$ holds in $V$.
              \item $Q=\emptyset$. The identity \eqref{L2.2} has the form $xyx=P^*xQ^*y^2x$. If $y\in\chi(P^*)$ then $xyx$ is divisible by $yx^ky^2$ $(k>1)$, and the same argument as in the previous case (where $y^2x^ky$ divided $xyx$) shows that \eqref{2} holds in $V$.

                  Let  $y\notin\chi(P^*)$. If $P\neq\emptyset$, then $xyx=x^mQ^*y^2x$, $m>1$ and $xyx=x^mQ^*y^2x=x^{m+n}Q^*y^2x=x^nxyx=x^{n+1}yx$. Therefore, $xyx=x^{n+1}yx$  and consequently $xyx$ is  divisible by $x^2yx$. On the other hand, changing $y$ into $xy$ in the identity $xyx=x^mQ^*y^2x$ we see that $x^2yx$ is divisible by $(xy)^2$. Thus, $xyx$ is divisible by $(xy)^2$, and by Lemma \ref{l9} the identity \eqref{2} holds in $V$.

                  Let $P^*=\emptyset$. In this case the identity \eqref{L2.2} transforms into the identity $xyx=xQ^*y^2x$. If $Q$ contains a subword $yx$ then $xyx$ is divisible by $yx^ky^2$. As we have seen above, it implies the identity \eqref{2}.

                  If  $Q$ does not contain any subword of the form $yx$ then the identity \eqref{L2.2} implies the identity $xyx=x^my^kx$ $(k>1)$. If $m>1$ then $xyx=x^my^kx=x^{m+n}y^kx=x^{n+1}yx$. But we have just seen that this identity implies the identity \eqref{2}.

                  Let $m=1$. This means that for some $l>1$ the identity
                  \begin{equation}\label{L10}
                    xyx=xy^{l}x,
                  \end{equation}
                  where $l\geq2$ holds in the variety $V$. Without loss of generality we can assume that $n$ is the least integer for which the identity \eqref{1} is fulfilled in $V$ and that $l<n+2$. Show that $l=n+1$; indeed, let $l\leq n$. Replacing $y$ by $x$ in the identity \eqref{L10} we prove that $x^3=x^{l+2}$ and, consequently, $x^2=x^{n+2}=x^{3+n-l}$ are identities of the variety $V$. But $1\leq1+(n-l)<n$ in contradiction with the minimality of $n$.

                  Let at first $n=1$. In this case, $y^2=y^3=\cdots=y^k=\cdots$, and the identities $xyx=xy^qx$ hold in $V$ for all $q$. Since $D \notin V$, Lemma \ref{l4} guarantees the fulfillment of the identity $xyx=PxQ$ with $P\neq\emptyset$, $Q\neq\emptyset$. If $y \notin\chi(P)$ or $y \notin\chi(Q)$ then either $xyx=x^{n+1}yx$ or $xyx=xyx^{n+1}$. By the same  argument as in previous cases we conclude that the identity \eqref{2} holds in $V$. Let  $y \in \chi(P)$ and $y \in \chi(Q)$. In this case, $xyx$ is divisible by $yx^my$ $(m>1)$. Changing $y$ into $y^l$ we obtain that $xy^lx$ is divisible by $y^2x^my^2$.  Since $xy^lx=xyx$, $xyx$ is divisible by $yx^my^2$. This situation has been already studied, and we have shown there that the identity \eqref{2} holds in $V$.

                  Let $n>1$. Since $K_n\notin V$, one of the identities \eqref{L5.1}, \eqref{L5.2}, \eqref{L5.3} holds in $V$ by Lemma \ref{l5}. Examine all possibilities.
                  \begin{description}
                    \item[\eqref{L5.1}] $x^m=PyQ$. Substitute $x$, $y$ for $y$, $(xy)^2$; we obtain that $y^q$ is divisible by $(xy)^2$ for all $q\geq m$. Therefore, $xyx=xy^{1+mn}x$ is divisible by $(xy)^2$ and the identity \eqref{2} holds in $V$ by Lemma \ref{l9}.
                    \item[\eqref{L5.2}] $xy^mx=PxQ$, $P \neq\emptyset$, $Q\neq\emptyset$, $m\equiv 1\pmod{n}$. If $m\neq1$ then $y^l=y^{n+1}=y^m$ by \eqref{1}, and making use of \eqref{L10} we obtain that $xyx=xy^lx=xy^mx=PxQ$. If $m=1$ then $xyx=PxQ$, i.e., in any case we have got the identity: $xyx=PxQ$, where $P\neq\emptyset$, $Q\neq\emptyset$. It remains to use the same argument as in the case of $n=1$.
                    \item[\eqref{L5.3}] $xy^{m_1}x=xy^{m_2}x$, $m_1\equiv1\pmod{n}$, $m_2\not\equiv 1\pmod{n}$. But this case is not possible because it follows from \eqref{L5.3} that $x^{2+m_1}=x^{2+m_2}$ and, consequently, that $m_1-m_2=(2+m_1)-(2+m_2)$ is divisible by $n$.
                  \end{description}
            \end{enumerate}
        \end{enumerate}
    \item An identity of the form \eqref{L2.3} holds in $V$. Assume that the identity $xyx=Px^2yQ$ is an identity of $V$ (the case of the identity $xyx=Pyx^2Q $ can be studied in a similar way). We shall show that then an identity of the form \eqref{L2.2} is fulfilled in $V$, and we have already proved that in this case the identity \eqref{2} holds in $V$.
        \begin{enumerate}
          \item $y\in\chi(P)$. Then $xyx$ is divisible by $yx^ky$ for an integer $k>1$. Changing $x$ into $y$ and $y$ into $x^k$, we obtain that $yx^ky$ is divisible by $x^ky^kx^k $. Therefore, $xyx$ is divisible by $x^ky^kx^k$, $k >1$, and it means that an identity of the form \eqref{L2.2} holds in $V$.
          \item $y \notin\chi(P)$. Then $(L2.3)$ turns into the identity $xyx=x^myQ$ with $m>1$. If $Q$ contains a subword $xy$ then $xyx$ is divisible by $yx^ky$. Changing $x$ into $y$ and $y$ into $x^k$ in the identity $xyx=x^myQ$ we obtain that $yx^ky=y^mx^kQ^*$ is divisible by $y^2x$. Therefore, $xyx$ is divisible by $y^2x$ and an identity of the form \eqref{L2.2} holds in $V$.

              Assume now that $Q$ does not contain any subword of the form $xy$. It is possible only if the identity \eqref{L2.3} has the form of $xyx=x^my^kx^p$ or of $xyx=x^my^k$ (for an integer $m>1$). If $k>1$ then these identities are identities of the form \eqref{L2.2}. Let $k=1$. Since
              $$xyx=x^my\Rightarrow xyx=x^{m+n}y=x^nx^my=x^{n+1}yx,$$
              we can restrict us to the case of the identity $xyx=x^myx^p$.

              The semigroup $C_\lambda$ is not contained in $V$; therefore, it follows from Lemma \ref{l3} that one of the identities \eqref{L3.1} or \eqref{L3.2} is fulfilled in $V$. In the second case $$xyx=x^myx^p=x^{m+sn}yx^p=x^lx^syxx^{p-1}=x^lPyQyR^*,$$ where $l=m+sn-s\geq 2$. If $x\notin\chi(Q)$, then the last identity is an identity of the form \eqref{L2.2}. Otherwise, $xyx$ is divisible by $yx^ky$ ($k\geq 1$), and $yx^ky=y^mx^ky^{mp}$ is divisible by $y^2x$; therefore, $xyx$ is divisible by $y^2x$, and an identity of the form \eqref{L2.2} holds in $V$.

              Suppose, at last, that the identity \eqref{L3.1} holds in $V$. The semigroup $A$ does not belong to $V$; by Lemma \ref{l1}, the identity \eqref{L1} hold in $V$. Then
              $$x^syx=x^tyx^q\Rightarrow x^{s+k}yx=x^{k+t}yx^{q+ln}=x^tPyQyRx^{q+ln-l},$$
              and this identity is an identity of the form \eqref{L3.2}. But we have just seen that it implies an identity of the form \eqref{L2.2}.
        \end{enumerate}
  \end{enumerate}\end{proof}
\begin{lemma}\label{l11}
  Let the identities \eqref{1} and \eqref{2} be fulfilled in the variety $V$ for a positive integer $n$, and let $L_{2}^{1} \notin V$ and $R_{2}^{1}\notin V$. If the identity
  \begin{equation}\label{L11}
    (ax)^{k_1}(ay)^{k_2}\cdots (ax)^{k_p}=(ay)^{l_1}(ax)^{l_2}\cdots (ay)^{l_{q-1}}(ax)^{l_q},
  \end{equation}
  where $k_i>1$, $l_j>1$ for all $i$, $j$, holds in $V$, then the identities
  \begin{subequations}
    \begin{align}
      \label{L11.1}
      (ay)^{2n}axaya=axaya;\\
      \label{L11.2}
      axaya(ay)^{2n}=axaya
    \end{align}
  \end{subequations}
  hold in $V$ too.
\end{lemma}
\begin{proof}
  Suppose that \eqref{L11.1} does not hold. Then the identity $(ay)^{2n}axayaxa=axayaxa$ does not hold. Indeed, if it was fulfilled then we should obtain, taking into account the identity \eqref{2}:
  $$(ay)^{2n}axaya=(ay)^{2n}(a\cdot xay)^{n+1}a=(a\cdot xay)^{n+1}a= a\cdot xay\cdot a$$
  in contradiction with  the assumption. Thus, there exists a semigroup $S\in V$ and elements $a, x, y\in S$ such that
  \begin{equation}\label{L11*}
    (ay)^{2n}axayaxa \neq axayaxa,
  \end{equation}
  Then
  $$(ay)^{2n}(ax)^{n+1}(ay)^{n+1}(ax)^{n+1}a \neq(ax)^{n+1}(ay)^{n+1}(ax)^{n+1}a$$
  and, consequently,
  \begin{equation}\label{L11**}
    (ay)^{2n}(ax)^{n+1}(ay)^{n+1}(ax)^{n+1}\neq(ax)^{n+1}(ay)^{n+1}(ax)^{n+1}.
  \end{equation}
  Consider the subsemigroup $S^*$ of the semigroup $S$ generated by all elements $(ax)^k$ and $(ay)^l$ with $k\geq 2, l\geq 2$, and show that it is a regular semigroup. Let $Z$ be an element of $S^*$; then
  $$ Z=(at_1)^{k_1}(at_2)^{k_2}\cdots (at_m)^{k_m},\quad k_1, k_2,\cdots,k_m>1, \quad t_1,t_2,\cdots ,t_m \in\{x,y\}.$$
  Because of the identity \eqref{1} we can assume without loss of generality that $k_m>3$. Therefore, making use of the identity \eqref{2} we obtain:
  \begin{align}
    \nonumber
    Z=((at_1)^{k_1}(at_2)^{k_2}\cdots (at_{m})^{k_m-1})^{n+1}(at_m)=((at_1)^{k_1}\cdots (at_m)^{k_m-1})^nZ=\\ \nonumber ((at_1)^{k_1}\cdots (at_m)^{k_m-1+n})^nZ=(Z(at_m)^{n-1})^nZ.
  \end{align}
  Thus, $S^*$ is a regular semigroup.

  Remark that $V$ does not contain the semigroup $B_{2}^{1}$  because the identity \eqref{2} does not hold in $B_{2}^{1}$, and the semigroups $L_{2}^{1}$ and $R_{2}^{1}$ are not contained in $V$ by our hypothesis. According to Proposition 3 the semigroup $S^*$ can be embedded into the direct product of its principal factors, and by Mann's theorem these principal factors are completely 0-simple semigroups. It follows from the identity \eqref{1} that all principal factors of the semigroup $S^*$ belong to $CS^{0}(B_{n})$. Denote by $\mathcal{U}$ and $\mathcal{V}$ respectively the left and the right sides of \eqref{L11**}. Let $t=(ax)^{n+1}$ and $s=(ay)^{n+1}$. Making use of the identity \eqref{1} we obtain that $t=t^{n+1}$ and $s=s^{n+1}$. Since $\mathcal{U} \neq\mathcal{V}$ in $S^*$ and $S^*$ is the subdirect product of semigroups belonging to $CS^{0}(B_{n})$,there exists a homomorphism $\phi:S^* \rightarrow T \in CS^{0}(B_{n})$ such that
  \begin{equation}\label{L11***}
    f(\mathcal{U})\neq f(\mathcal{V}).
  \end{equation}
  Since $\mathcal{U}$ and $\mathcal{V}$ are divisible in $S^*$ by $ts$ and by $st$, the inequality \eqref{L11***} implies that $f(ts)\neq0$ and  $f(st)\neq0$ And it follows from \eqref{L11} that $t^{k_1}s^{k_2}\cdots t^{k_p}=s^{l_1}t^{l_2}\cdots t^{l_q}$. Denote the left and the right sides of the last inequality respectively by $P$ and $Q$. Since $t=t^{n+1}$, $t$ is divisible by $t^2$. Therefore, $f(t^2)\neq0$: otherwise we should have $f(t)=0$ and  $f(\mathcal{U})= f(\mathcal{V})=0$ in contradiction with \eqref{L11***}. Similarly, $f(s^2) \neq0$.

   Remind that if the product $z_1z_2\cdots z_r$ of elements of any completely 0-simple semigroup is equal to $0$ then $z_iz_{i+1}=0$ for at least one $i$, $1\leq i<r$. Therefore, since  $f(st),f(ts), f(t^2),f(s^2)\neq0$, we see that $f(P)=f(Q)\neq0$. We know from the matrix description of completely 0-simple semigroups that if an element of such semigroup is divisible on the right and on the left by another element of this semigroup, then both elements belong to the same $\mathcal{H}$-class. Hence, $f(t)\mathcal{H}f(P)$, and $f(t)\mathcal{H}f(Q)$ because $f(P)=f(Q)$.  Therefore, any right unity of the element $f(Q)$ is in the same time a right unity of the element $f(t)$. But it follows from \eqref{1} that $f(s)^{2n}$ is a right unity of $f(Q)$; hence, it is a right unity of the element $f(t)$ too. thus, $f(s)^{2n}f(t)= f(t)$, and $f(\mathcal{U})=f(s)^{2n}f(tst)=f(tst)=f(\mathcal{V})$ in contradiction with \eqref{L11***}. Therefore, our assumption was erroneous and the identity \eqref{L11.1} holds in $V$.

   The identity \eqref{L11.2} can be proved similarly.\end{proof}

\begin{lemma}\label{l12}
  Let the identities \eqref{1}, \eqref{2} and \eqref{L7.1} hold in $V$. Then the following identities hold in $V$ for all $k,l,m\neq0:$
  \begin{subequations}\label{L12}
    \begin{align}
      \label{L12.1}
      xyx=(xyx)^{n+1};\\
      \label{L12.2}
      x^{k}y^{l}x^{m}=(x^{k}y^{l})^{n+1}x^{m};\\
      \label{L12.3}
      x^{k}y^{l}x^{m}=x^{k}(y^{l}x^{m})^{n+1}.
    \end{align}
  \end{subequations}
\end{lemma}
\begin{proof}
  The identity $(xy)^{n+1}x=Py^2Qx$ follows from \eqref{L7.1}; according to \eqref{2}, $xyx=(xy)^{n+1}x=Py^2Q^*$. After the substitution $x \mapsto y$, $y\mapsto xyx$ we obtain  the identity $yxyxy=P^*(xyx)^2Q^*$. We see from the last identity that $xyx$ is divisible by $(xyx)^2$. This means that the principal factor $K_{xyx}$ is not a semigroup with zero  multiplication. According to Mann's theorem, $K_{xyx}$ is a completely 0-simple semigroup and $xyx$ is an element of a subgroup of the semigroup $K_{xyx}$; since all groups from $V$ satisfy the identity $x^2=x^{n+2}$, they are groups of exponent $n$, and $xyx=(xyx)^{n+1}$. The last identity is the identity \eqref{L12.1}.

  Let now $k,l,m>0$. By \eqref{2} the element $x^ky^lx^m$ is a left unity for $x^ky^lx^k$. Making use of \eqref{L12.1} and then of \eqref{1}, we obtain:
  \begin{align}
    \nonumber
    x^ky^lx^m=x^{k-1}(xy^lx)x^{m-1}=x^{k-1}(xy^lx)^{n+1}x^{m-1}=x^ky^lx^2Q=\\ \nonumber
    x^ky^lx^{2+kn}Q = x^ky^lx^kx^{2+k(n-1)}Q.
  \end{align}
  Therefore, $(x^ky^l)^n$ is  a left unity for the element  $x^ky^lx^m$ as well, i.e, $x^ky^lx^m=(x^ky^l)^{n+1}x^m$. The identity \eqref{L12.2} can be proved in a similar way.
\end{proof}

\begin{lemma}\label{l13}
  Let the identities \eqref{1} , \eqref{2} and \eqref{L7.1} hold in $V$. Then for any $k,l,s,t$ the identity
  \begin{equation}\label{L13}
    xyx=R(x^ky^l)^{n+1} (x^sy^t)^{n+1}T,
  \end{equation}
  where $R$ and $T$ are words or empty symbols, holds in $V$.
\end{lemma}
\begin{proof}
  By Lemma \ref{l12}, the identity \eqref{L12.1} holds in $V$. Hence,
  \begin{align}
    \nonumber
    xyx=(xyx)^2(xyx)^{n-1}=xyx^2yx(xyx)^{n-1}=x(yx^2y)^{n+1}x(xyx)^{n-1}=\\
    \nonumber
    x(yx^2y)(yx^2y)S=xyx^2y^2x^2yS=xyx^{2+kn}y^{2+ln}Rx^2yS=R_1x^ky^lR_2.
  \end{align}
  Therefore, making use several times of the identities \eqref{1}, \eqref{L12.1}, we obtain:
  \begin{align}
    \nonumber
    xyx= (xyx)^2(xyx)^{n-1}=R_1x^ky^lR_2R_1x^ky^lR_3=R_1(x^ky^lR_2R_1x^ky^l)^{n+1}R_3=\\
    \nonumber
    S_1 (x^ky^l)^2S_2=S_1(x^ky^l)^{n+1}(x^ky^l)S_2=S_1[(x^ky^l)^{n+1}x^k]^{n+1}S_2=\\
    \nonumber
    S_1(x^ky^l)^{n+1}x^{2k}S_3=S_1(x^ky^l)^{n+1}x^{2k+sn}S_3=S_1(x^ky^l)^{n+1}x^sS_4.
  \end{align}
  Using the identities \eqref{1}, \eqref{L12.1}, \eqref{L12.3} and again \eqref{L12.3}, obtain the needed identity \eqref{L13}:
  \begin{align}
    \nonumber
    xyx=S_1(x^ky^l)^{n}x^ky^lx^sS_4=S_1(x^ky^l)^{n}x^k(y^lx^s)^{n+1}S_4=S_1(x^ky^l)^{n}x^ky^lx^sy^lS_5=\\
    \nonumber
    S_1(x^ky^l)^{n}x^k(y^lx^sy^l)^{n+1}S_5=S_1(x^ky^l)^{n+1}x^sy^{2l}S_6=S_1(x^ky^l)^{n+1}x^sy^{2l+tn}S_6=\\
    \nonumber
    S_1(x^ky^l)^{n+1}x^sy^tS_7=S_1(x^ky^l)^{n+1}(x^sy^t)^{n+1}S_7.
  \end{align}
\end{proof}

\begin{lemma}\label{l14}
  Let the identities \eqref{1}, \eqref{2} and \eqref{L7.1} hold in $V$. If the variety $V$ does not contain any of the semigroups $F_{\lambda}$, $F_{\rho}$, $L_{2}^{1}$, $R_{2}^{1}$, then the following identities are fulfilled in the variety $V$:
  \begin{subequations}\label{L14}
    \begin{align}
      \label{L14.1}
      xyx=x^{n+1}yx;\\
      \label{L14.2}
      xyx=xyx^{n+1}.
    \end{align}
  \end{subequations}
\end{lemma}
\begin{proof}
    Lemma \ref{l12} asserts that the identities \eqref{L12.1}, \eqref{L12.2}, \eqref{L12.3} are fulfilled in $V$. Let $S=<x,y>$ be the free semigroup of the variety $V$. Consider the subsemigroup $S^*\subseteq S$ generated by the elements $(x^ky^l)^m$ and $xPx$, where $m\geq 2$ and $P$ is a word or an empty symbol. We shall show that $S^*$ is a regular semigroup. Let $z_1, z_2,\cdots , z_p$ be elements of the mentioned form, and let $W=z_1,z_2,\cdots ,z_p$. If $z_p=xPx$ then $W=xQx$ and by the identity \eqref{L12.1} we have $W=W^{n+1}\in WS^*W$. This means that $W$ is a regular element of $S^*$. Let $z_p=(x^ky^l)^m$, $(m\geq 2)$. Then
    \begin{align}
      \nonumber
      W\overset{\eqref{1}}\to=z_1\cdots z_{p-1}(x^ky^l)^{m+n-1}x^ky^l\overset{\eqref{L12.2}}\to=(z_1\cdots z_{p-1}(x^ky^l)^{m+n-1})^{n+1}x^ky^l=\\
      \nonumber
      (z_1\cdots z_{p-1}(x^ky^l)^{m+n-1})^{n}W\overset{\eqref{1}}\to=(z_1\cdots z_{p-1}(x^ky^l)^{m+3n-1})^nW=(W (x^ky^l)^{3n-1})^nW.
    \end{align}

    By definition $(x^ky^l)^{3n-1} \in S^*$. Therefore $W \in WS^*W$, and $W$ is again a regular element of $S^*$. Thus, $S^*$ is a regular semigroup.

    By Lemma \ref{l13} the identity \eqref{L13} holds in $V$. Using \eqref{1} and \eqref{2} we come to the identity
    \begin{equation}\label{L14*}
      xyx=(xy)^nR(x^ky^l)^{n+1}(x^sy^t)^{n+1}(x^sy^t)^nT(yx)^n.
    \end{equation}

    But the element $(xy)^nR   (x^ky^l)^n$ can be written as $(x^{k_1}y^{l_1})(x^{k_2}y^{l_2})\cdots (x^{k_p}y^{l_p})$. Using several times the identity \eqref{L12.2}, transform this element into
    \begin{equation}\label{L14**}
      (x^{k_1}y^{l_1})^{n+1} (x^{k_2}y^{l_2})^{n+1}\cdots (x^{k_p}y^{l_p})^{n+1}x.
    \end{equation}

    Substituting in \eqref{L14*} the product $(xy)^nR(x^ky^l)^n$ for \eqref{L14**}, we obtain  the identity
    $$xyx=P(x^ky^l)^{n+1}(x^sy^t)^{n+1})Q,$$
    where $P=(x^{k_1}y^{l_1})^{n+1}(x^{k_2}y^{l_2})^{n+1}\cdots (x^{k_p}y^{l_p})^{n+1}$, $Q \in xSx$. Both elements $P$, $Q$ are contained in $S^*$. Thus, for every $k,l,s,t\geq1$ the element $xyx$ is divisible in $S^*$ by $(x^ky^l)^{n+1}(x^sy^t)^{n+1}$.

    Since $F_\lambda \notin V$, Lemma \ref{l6} guarantees the fulfillment of the identity \eqref{L6}. Without loss of generality assume $P$ and $Q$ to be words over $<x,y>$. The element $x^n$ is a left unity for $x^kyQ$, because $x^k=x^{k+n}$ by \eqref{1}. Therefore, $x^n$ is a left unity for $xyP$; it follows that $xyP=x^{n+1}yP$ and $xyPxy=x^{n+1}yPxy$ or, after the substitution $y \mapsto yxy$, that $xyxP^*y=x^{n+1}yxP^*y$. The last equality can be written as follows:
    $$xyx^{k_1}y^{l_1}x^{k_2}y^{l_2}\cdots x^{k_s}y^{l_s}=x^{n+1}yx^{k_1}y^{l_1}\cdots x^{k_s}y^{l_s}.$$
    It implies that
    $$(xy)(x^{k_1}y^{l_1})(x^{k_2}y^{l_2})\cdots (x^{k_s}y^{l_s})(xy)^{n+1}=(x^{n+1}y)(x^{k_1}y^{l_1})\cdots (x^{k_s}y^{l_s})(xy)^{n+1};$$
    using \eqref{L12.2}, we obtain the identity
    \begin{align}
      \nonumber
      (xy)^{n+1}(x^{k_1}y^{l_1})^{n+1}\cdots (x^{k_s}y^{l_s})^{n+1}(xy)^{n+1}=
      \nonumber
      (x^{n+1}y)^{n+1}(x^{k_1}y^{l_1})^{n+1}\cdots(x^{k_s}y^{l_s})^{n+1}(xy)^{n+1}.
    \end{align}
    Denote by $r$ and $t$ respectively the left and the right sides of the last equality. By definition $r,t \in S^*$.

    Let $a$ and $b$ denote  the elements $(xy)^{n+1}$ and $(x^{n+1}y)^{2n}$. It is clear that $a,b \in S^*$. Prove that $a=ba$. Assume the contrary. By \eqref{2} $B_{2}^{1}\notin V$, and $L_{2}^{1}, R_{2}^{1} \neq S^*$ by the hypothesis of Lemma. According to Proposition 3, $S^*\in V$ can be approximated by its completely 0-simple principal factors. Hence, there exists a homomorphism $f:S^* \rightarrow M\in CS^{0}(B_{n})$ such that $f(a) \neq f(b)$. We have:
    \begin{align}
      \nonumber
      (xy)^{n+1}=xyx\mathcal{U}y\overset{\eqref{L12.1}}\to=(xyx)^{n+1}\mathcal{U}y=(xyx)(x\mathcal{V}y)=  (xyx)(x^{m_1}y^{n_1}\cdots x^{m_q} y^{n_q})\overset{\eqref{L12.2}}\to=\\
      \nonumber
      (xyx)(x^{m_1}y^{n_1})^{n+1}(x^{m_2}y^{n_2})^{n+1}\cdots (x^{m_{q-1}}y^{n_{q-1}})^{n+1}x^{m_q} y^{n_q}\overset{\eqref{L12.3}}\to=\\
      \nonumber
      =(xyx)(x^{m_1}y^{n_1})^{n+1}(x^{m_2}y^{n_2})^{n+1}\cdots (x^{m_{q-1}} y^{n_{q-1}})^{n+1}(x^{m_q} y^{n_q})^{n+1}.
    \end{align}
    Therefore, $a$ is divisible in  $S^*$ by $xyx$. But we have seen above that $xyx$ is divisible by  $(x^ky^l)^{n+1} \cdot (x^sy^t)^{n+1})$ for all $k,l,s,t \geq 1$; hence, $a$ is divisible by these elements too.

    Let us denote $(xy)^{n+1}$ by $z$, $(x^{k_i}y^{l_i})^{n+1}$ by $z_i$, $i=1,\cdots ,s$, and $(x^{n+1}y)^{n+1}$ by $z^*$; then $r=zz_1\cdots z_sz$, $t=z^*z_1\cdots z_sz$. If $f(r)=f(z)f(z_1)\cdots f(z_s)f(z)=0$, then at least one of the products  $f(z)f(z_1)$, $f(z_1)f(z_2)$,$\cdots$,$f(z_s)f(z)$ is equal to 0 because $M$ is a 0-simple semigroup. But we have seen that $a$ is divisible by all products $zz_1$, $z_1z_2$,$\cdots$,$z_sz$. Therefore, $f(a)$ is divisible in $M$ by $f(z)f(z_1)$, $f(z_1)f(z_2)$,$\cdots$,$f(z_s)f(z)$, and consequently, $f(a)=0$. Then $f(ba)=f(b)f(a)=f(b)\cdot 0=0=f(a)$ in contradiction with our assumption. Hence, $f(r)\neq0$. Then $f(a)\neq0$, because $r$ is divisible by $a$.

    The non-zero elements $f(r)$ and $f(a)$ are contained in the same $\mathcal{H}$-class of $M$, because $r$ is divisible by $a$ both on the left and on the right. Since $f(r)=f(t)$, we have also $f(a)\mathcal{H} f(t)$. The element $f(b)=f( (x^{n+1}y)^{2n}) $ is a left unity of the element $f(t)$. Therefore, $f(b)$ is a left unity for $f(a)$ too, i.e., $f(a)=f(b)f(a)=f(a)$ in contradiction with the choice of $f$.

    Hence, $a=ba$, i.e., $(xy)^{n+1}=(x^{n+1}y)^{2n}(xy)^{n+1}$. This identity implies the identity  $(xy)^{n+1}x=(x^{n+1}y)^{2n}(xy)^{n+1}x$. Making use of \eqref{L12.2} we obtain: $xyx=(x^{n+1}y)^{2n}xyx$. But it follows from \eqref{1} that $x^n$ is a left unity for the right side of the last equality; hence, it is a left unity for the left side, i.e., $xyx=x^{n+1}yx$.

    Dual reasonning enables us to prove the identity \eqref{L14.2}.\end{proof}

\begin{lemma}\label{l15}
  Let the identities \eqref{1}, \eqref{2} and $xyx=Py^2Q$ be fulfilled in $V$. If the variety $V$ does not contain any of the semigroups $L_{2}^{1}$, $R_2^1$, $F_\lambda$, $F_\rho$, then the identities
  \begin{subequations}
    \begin{align}
      \label{L15.1}
      axaya=(ay)^{2n}axaya;\\
      \label{L15.2}
      axaya=axaya(ay)^{2n},
    \end{align}
  \end{subequations}
  hold in $V$ too.
\end{lemma}
\begin{proof}
  Remark at first that an identity of the form $a^kx^lP=a^py^qQ$, where $a,x,y$ are different variables and $P,Q $ are words over $\{a,x,y\}$, hold in $V$ because $L_{2}^{1}\notin V$. The substitution $x \mapsto xa$; $y\mapsto ya$ transforms it into the identity $$a^kxa^{k_1}\mathcal{U}_1a^{k_2}\mathcal{U}_2\cdots a^{k_{s-1}}\mathcal{U}_{s-1}a^{k_s}=a^lxa^{l_1}\mathcal{V}_1a^{l_2}\mathcal{V}_2\cdots a^{l_{t-1}}\mathcal{V}_{t-1}a^{l_t},$$
  where $\mathcal{U}_1,\mathcal{U}_2,\cdots,\mathcal{U}_{s-1}\in\{x,y\}$ and $\mathcal{V}_1,\mathcal{V}_2,\cdots ,\mathcal{V}_{t-1} \in \{x,y\}$. Without loss of generality we can assume $k=l=n+1$, because otherwise the substitution $x \mapsto a^{(n-1)k+1}x$, $y \mapsto a^{(n-1)l+1}y$ should return us to this case. Further, Lemmas \ref{l12} and \eqref{L14} guarantee the fulfillment of the identities \eqref{L12.2}, \eqref{L12.3}, \eqref{L14.1} and \eqref{L14.2}. Making use of these identities obtain the identity
  \begin{align}
    \nonumber
    R=(ax)^{n+1}(a^{k_1}\mathcal{U}_1)^{n+1}\cdot\cdots\cdot(a^{k_s}\mathcal{U}_s)^{n+1}(ax)^{n+1}=\\
    \nonumber
    (ay)^{n+1}(a^{l_1}\mathcal{V}_1)^{n+1}\cdot\cdots\cdot(a^{l_t}\mathcal{V}_t)^{n+1}(ay)^{n+1}=T.
    \end{align}

    Show that if the identity \eqref{L15.1} does not hold in $V$ then
    $$(ax)^{n+1}\cdot(ay)^{n+1}\cdot(ax)^{n+1}\neq(ay)^{n+1}\cdot(ax)^{n+1}\cdot(ay)^{n+1}\cdot(ax)^{n+1}$$ in the free semigroup $S=S(a,x,y)$ of the variety $V$. In fact, if $$(ax)^{n+1}\cdot(ay)^{n+1}\cdot(ax)^{n+1}=(ay)^{n+1}\cdot(ax)^{n+1}\cdot(ay)^{n+1}\cdot(ax)^{n+1}$$ then by \eqref{L12.2} we should have $axayaxa=(ay)^{2n}axayaxa$ and consequently, $a(xaya)^{n+1}=(ay)^{2n}a(xaya)^{n+1}.$ Then it follows from \eqref{2} that $axaya=(ay)^{2n}axaya $ in contradiction with the assumption.

    Let $c$ denote  $(ax)^{n+1}\cdot(ay)^{n+1}\cdot(ax)^{n+1}$ and let $d$ denote $(ay)^{n+1}\cdot(ax)^{n+1}\cdot(ay)^{n+1}\cdot(ax)^{n+1}$. We have just shown that $c\neq d$.

    Let us consider the semigroup $S^*$ generated by the elements  $(a^px)^m$ and $(a^qy)^m$, where $m>2$. Remark that $B_{2}^{1}\notin V$, because \eqref{2} does not hold in $B_2^1$, and $L_2^1$, $R_2^1\notin V$ by the hypothesis. Therefore, we can prove exactly in the same way as in the proof of Lemma \ref{l14} that $S^*$ is a regular semigroup.

    Since $c,d\in S^*$ and $c\neq d$, there exists a homomorphism $f:S^*\rightarrow  M$, of the semigroup $S^*$ into a semigroup $M\in CS^0(V)$ such that $f(c)\neq f(d)$. Denote by $g_{p,q}$ the element $(a^px)^{n+1}(a^qy)^{n+1}$ and by $ h_{p,q}$ the element $(a^qy)^{n+1}(a^px)^{n+1}$. Show that for all $p$ and $q$ the element $c$ is divisible in $S^*$ by $ g_{p,q}$ and by $h_{p,q}$. Indeed, it follows from the identity \eqref{L12.2} that $$c=(ax)^{n+1}ay(ay)^n(ax)^{n+1}=ax(ay)^{n+1}ax(ax)^n=axay(ax)^{n+1}.$$ Further, making use of \eqref{1}, \eqref{L14.1} and \eqref{L14.2} obtain
    \begin{align}
      \nonumber
      c=axay(ax)^{n+1}=a^{n+1}xay(ax)^{n+1}=a^{n+1}xayax(ax)^{n}=\\
      \nonumber
      a^{n+1}xa^{n+1}y(ax)^{n+1}=a^{np+1}xa^{nq+1}y(ax)^{n+1}=a^{p^*}a^pxa^{q^*}a^qy(ax)^{n+1},
    \end{align}
    where $p^*=(n-1)p+1$, $q^*=(n-1)q+1$. Apply several times the identity\eqref{L12.2}:
    \begin{align}
      \nonumber
      c=a^{p^*}(a^pxa^{q^*}a^qy)(ax)^{n+1}=a^{p^*}(a^pxa^{q^*}a^qy)^{n+1}(ax)^{n+1}=\\
      \nonumber
      a^{p^*}a^pxa^{q^*}a^qya^pxaW=a^{p^*+p}xa^{q^*}((a^qy)^{n+1}(a^px)^{n+1})^{n+1}aW=\\
      \nonumber
      (a^{p^*+p}x)^{n+1}(a^{q^*+q}y)^{n+1}(a^qy)^{n}(a^px)^{n+1}(a^qy)^{n+1}(a^px)^{n+1}W_1.
    \end{align}

    It is obvious that $aW$ and $W_1$ are contained in $S^*$; therefore, $c$ is divisible in $S^*$ by $(a^px)^{n+1}(a^qy)^{n+1}=g_{p,q}$ and $(a^qy)^{n+1}(a^px)^{n+1}=h_{p,q}$.

    If $f(R)=0$ then there exist $p$, $q$ such that $f(g_{p,q})=0$ or $f(h_{p,q})=0$ because of 0-simplicity. But in this case $f(c)=0$ and consequently, $f(d)=0$, i.e., $f(c)=f(d)$ in contradiction with the hypothesis. Thus, $f(R) \neq0$.

    Since the nonzero elements $f(s)$ and $f(R)$ have a common two-side unity $f((ax)^{2n}))$, we have $f(c)\mathcal{H}f(R)$. The elements $f(R)=f(T)$ and $f(T)$ have a common left unity $f((ay)^{2n})$, therefore, $f((ay)^{2n})$ is a left unity for $f(c)$, i.e., $f(c)=f((ay)^{2n})\cdot f(c)$. This implies the equality $f(c)=f(d)$ which contradicts the hypothesis.

    So, we have proved that the identity \eqref{L15.1} holds in $V$. The fulfillment of \eqref{L15.2} can be proved similarly.\end{proof}

The following result was obtained and proved by the author in \cite{Hall1}.

\begin{proposition}\label{p4}
  Let $V$ be a semigroup variety satisfying the identities \eqref{1}, \eqref{2}, \eqref{L15.1} and \eqref{L15.2} or the identities \eqref{1}, \eqref{2} and \eqref{3}. Then for any semigroup $S$ of $V$ and for any regular elements $\mathcal{U}$, $\mathcal{V}$ there exists  a principal factor of $S$ which is a completely 0-simple semigroup and an epimorphism $f:S\rightarrow M$ such that $f(\mathcal{U})\neq f(\mathcal{V})$.\qed
\end{proposition}

\begin{proposition}\label{p5}
  The semigroup $S_0$ is not residually in the class of completely 0-simple semigroups and semigroups with zero multiplication.
\end{proposition}
\begin{proof}
  Let $K=CS^0\cup H$ where $H$ is the class of all semigroups with zero multiplication and $CS^0$ is the class of all completely 0-simple semigroups. Assume that $S_0$ is a residually  $K$-semigroup; then there exists a homomorphism $f:S_0\rightarrow M\in K$ such that $f(c)\neq f(0)$. It follows immediately that $M$ is not a semigroup with zero multiplication. Therefore, the elements $f(a)$, $f(b)$, $f(c)$ belongs to a same $\mathcal{H}$-class  $P$ of the semigroup $M$ and this class is not a semigroup with zero multiplication. Hence, $P$ is a subgroup of $M$, and $f(0)=f(a)^3\in P$, i.e., $f(0)$ is the unity of the group $P$. Thus, $f(c)=f(0)$.
\end{proof}

Let $\Lambda =\{1,2\}$, $X=\{a,b\}$; denote by $T_l$ the set
$(\Lambda\times X)\cup\Lambda\cup X\cup\{0\}$ and define a binary
operation on $T_l$ by the following formulae:
$$(m,x)(n,y)=(\min(m,n),y);$$
$$\alpha \beta=0;$$
$$(m,x)n=(m,x)y=0;$$
$$x(n,y)=y;$$
$$m(n,y)=\begin{cases}
  m, &\text{ if } m\leq n,\\
  y, &\text{ otherwise}
\end{cases}$$
$$z\cdot 0=0\cdot z=0$$
(in these formulae $z\in T_{l}; \alpha,\beta\in \Lambda \cup X; m,n\in \Lambda$; $x,y\in X$). It is easy to check that this operation is associative, i.e., $T_l$ is a semigroup. Denote by $T_r$ the semigroup which is antiisomorphic to $T_l$.

\begin{proposition}\label{p6}
  The semigroups $T_l$, $T_r$ are not residually completely 0-simple semigroups.
\end{proposition}
\begin{proof}
  Assume the contrary. Then there exists a homomorphism $f$ of the semigroup $T_l$ into a completely 0-simple semigroup $M$ such that $f(a)\neq f(b)$. Since $a(1,b)=b$, $b(1,a)=a$, we obtain that $f(a)\neq0$, $f(b)\neq0$. Further, $a(n,a)=a$, therefore, $f(1,a)\neq0$, $f(2,a)\neq0$. It follows from the definition of the multiplication on $T_l$ that $(1,a)(2,a)=(2,a)(1,a)=(1,a)$; hence, $f(1,a)$ is divisible by $f(2,a)$ on the left and on the right, i.e. $f(1,a)\mathcal{H} f(2,a)$. But $f(1,a)$ and $f(2,a)$ are nonzero idempotents, therefore, $f(1,a)=f(2,a)$, and $f(2)=f(2)f(2,a)=f(2)f(1,a)=f(a)$. The same argument proves that $f(2)=f(b)$, i.e., $f(a)=f(b)$ in contradiction with the above inequality $f(a)\neq f(b)$.
\end{proof}

\begin{proposition}\label{p7}
  If the identity $xy=xy^{n+1}$ holds in a semigroup variety $V$ and $T_l\notin V$, then an identity of one of the following forms holds in $V$:
  \begin{subequations}
    \begin{align}
      \label{P7.1}
      xR=xPxQ;\\
      \label{P7.2}
      Px=Qy,
    \end{align}
  \end{subequations}
  or the identity \eqref{L5.1}, where $P$, $Q$ are words or empty symbolds, $x$, $y$ are different variables and $x\notin \chi(R)$.
\end{proposition}
\begin{proof}
  Assume that none of the identities \eqref{L5.1}, \eqref{P7.1}, \eqref{P7.2} holds in $V$. Since $T_l\notin V$, there is an identity $\mathcal{U}=\mathcal{V}$ of the variety $V$ which fails in the semigroup $T_l$. It is obvious that $\chi (\mathcal{U})=\chi(\mathcal{V})$: otherwise an identity of the form \eqref{L5.1} should hold in $V$. Denote by $Y=\{x_1,\dots,x_n\}$ the set $\chi (\mathcal{U})=\chi (\mathcal{V})$. Since the identity $\mathcal{U}=\mathcal{V}$ does not hold in $V$ there is a mapping $f:Y\rightarrow T_l$ such that $f(\mathcal{U})\neq f(\mathcal{V})$. Let us consider possible cases.
  \begin{enumerate}
    \item The first letters $x_i$, $x_j$ of the words $\mathcal{U}$, $\mathcal{V}$ are distinct. Since $\chi (\mathcal{U})=\chi (\mathcal{V})$, the words $\mathcal{U}$, $\mathcal{V}$ have the form $\mathcal{U}=x_iPx_jQ$, $\mathcal{V}=x_jP^*x_iQ^*$. The identity $xy=xy^{n+1}$ is fulfilled in $V$, therefore, $f(\mathcal{U})=f(x_iPx_jQ)=f(x_iPx_j^{n+1}Q)$ and similarly $f(\mathcal{V})=f(x_jP^*x_i^{n+1}Q^*)$. Hence, without loss of generality we can assume that the letter $x_j$ occurs in the word $\mathcal{U}$ at least twice and the letter $x_i$ occurs in the word $\mathcal{V}$ at least twice.
        \begin{enumerate}
          \item $f(x_i)\in\Lambda\cup X\cup\{0\}$, $f(x_j)\in\Lambda\cup X\cup\{0\}$. This case is not possible because $I=\Lambda \cup X\cup\{0\}$ is an ideal of $T_l$ and the product of any two elements of $I$ is equal to 0; therefore, $f(\mathcal{U})=f(\mathcal{V})=0$.
          \item $f(x_i)\in \Lambda \cup X\cup \{0\}$, $f(x_j)\notin\Lambda\cup X\cup \{0\}$ (or \textit{vice versa}). Since $x_i$ occurs in $\mathcal{V}$ at least twice, we conclude as above that $f(\mathcal{V})=0$. Therefore, $f(\mathcal{U})\neq0$; it follows that $x_i$ occurs in $\mathcal{U}$ only once. Replacing $x_i$ with $x$ and all other variables with $y$ obtain the identity $xy^m=y^pxP'xQ'$, in which $P'$, $Q'$ are words over $\{x,y\}$ or empty symbols. Multiply both sides of this identity by $y^{(n-1)m+1}$ on the right and make use of the identity $xy=xy^{n+1}$; our identity turns into the identity
              \begin{equation}\label{P7*}
                xy=xy^{mn+1}=y^pxP'xQ,
              \end{equation}
              in which $Q=Q'y^{(n-1)m}+1$. Substitute in \eqref{P7*} the variable $x$ for $y^p$ and the variable $y$ for $x$; we come to the identity $y^px=x^py^pT=xP$. Therefore, the identity $$xy=y^pxP'xQ=xPP'xQ$$ of the form \eqref{P7.1} holds in $V$ in contradiction with the hypothesis.
          \item $f(x_i)\notin \Lambda \cup X\cup \{0\}$, $f(x_j)\notin\Lambda \cup X\cup \{0\}$. It means that $f(x_i),f(x_j)\in \Lambda \times X$. If $f(x_k)\in \Lambda\cup X\cup \{0\}$ for a variable $x_k\in Y$ then we should have:
              $$f(\mathcal{U})=f(x_iP'x_kQ')=f(x_iP'x_k^{n+1}Q')=0,$$
              and similarly $f(\mathcal{V})=0$, i.e., $f(\mathcal{U})=f(\mathcal{V})$ in contradiction with the assumption. Therefore, $f(x_k)\in \Lambda \times X$ for all $x_k\in Y$. It follows from the definition of the multiplication in the semigroup $T_l$ that $f(\mathcal{U})=(m,f(z))$, where $z$ is the last letter of $\mathcal{U}$ and $m$ is the minimum of the first components of the pairs $f(x_s)\in \Lambda\times X$ ($x_s\in Y$). Similarly, $f(\mathcal{V})=(m,f(t))$, where $t$ is the last letter of $\mathcal{U}$. Since the identity \eqref{P7.2} fails in $V$, the last letters $z$, $t$ of the words $\mathcal{U}$, $\mathcal{V}$ coincide, and $f(\mathcal{U})=f(\mathcal{V})$ in contradiction with the assumption.
        \end{enumerate}
    \item The words $\mathcal{U}$, $\mathcal{V}$ begin with the same letter $x$. If $x$ occurs only once in $\mathcal{U}$ and it occurs at least twice in $\mathcal{V}$ or \textit{vice versa}, then the identity $\mathcal{U}=\mathcal{V}$ has the form \eqref{P7.2} in contradiction with the assumption. If $x$ occurs two or more times in each of the words $\mathcal{U}$ and $\mathcal{V}$ then, as in previous case, we obtain $f(\mathcal{U})=f(\mathcal{V})$. Therefore, we can assume that $x$ occurs in each of the words $\mathcal{U}$, $\mathcal{V}$ only once.
        \begin{enumerate}
          \item $f(x)\in \Lambda \cup X\cup \{0\}$. If $f(x_k)\in \Lambda \cup X\cup \{0\}$ for some $x_k\in Y$, $x_k\neq x$, then $f(\mathcal{U})=0=f(\mathcal{V})$. If $f(x_k)\notin \Lambda \cup X\cup \{0\}$ for all $x_k\in Y$, $x_k\neq x$, then $f(x_k)\in \Lambda \times X$ and it follows easily from the definition of the multiplication on $T_l$ that either $f(\mathcal{U})=f(x)$, $f(\mathcal{V})=f(x)$ or $f(\mathcal{U})=f(z)$, $f(\mathcal{V})=f(t)$, where $z$, $t$ are the last letters of the words $\mathcal{U}$, $\mathcal{V}$. Since the identity \eqref{P7.2} fails in $V$, the last letters of the words $\mathcal{U}$, $\mathcal{V}$ coincide, and in all cases $f(\mathcal{U})=f(\mathcal{V})$ in contradiction with the assumption.
          \item $f(x)\notin \Lambda \cup X\cup \{0\}$. Exactly as in the previous case  we obtain that $f(\mathcal{U})=f(\mathcal{V})$.
        \end{enumerate}
  \end{enumerate}\end{proof}
Denote by $M_l$ the semigroup
$$M_l=\{c,a,b,e,f,0|c=ae=af=be;\,ce=cf=c;\,ef=e^2=e;\, fe=f^2=f\}$$
(all products which are not mentioned here are equal to 0). Further, denote
by $M_r$ the semigroup antiisomorphic to $M_l$.

\begin{proposition}\label{p8}
  The semigroups $M_l$, $M_r$ are not residually $B_2$-semigroups.
\end{proposition}
\begin{proof}
  Assume the contrary. Then there exists a $B_2$-semigroup $M$ and a homomorphism $f:M_l\rightarrow M$ such that $f(c)\neq f(0)$. Let $\theta $ be the zero element of the semigroup $M$. If $f(0)\neq\theta $ then the $\mathcal{H}$-class of the idempotent $f(0)$ is a group and $f(0)$ is the unity of this group. Since $(f(c))^2=f(c^2)=f(0)\neq\theta $, the elements $f(c)$ and $f(0)$ belong to the same $\mathcal{H}$-class and $f(c)=f(c)f(0)=f(0)$ in contradiction with the assumption. Therefore, $f(0)=\theta $.

  For any element $x\in M_l$ such that $f(x)\neq0$ we shall denote the element $f(x)$ from the completely simple semigroup $M\backslash\theta $ by $(i_x,g_x,\lambda_x)$. Let $(P_{\lambda ,i})$ be the sandwich-matrix defining this semigroup. Since $f(c)\neq\theta $, we have: $P_{\lambda_a,i_e}\neq\theta $, $P_{\lambda _a,i_f}\neq\theta $, $P_{\lambda _b,i_e}\neq\theta $. But $bf=0$; hence $P_{\lambda _b,i_f}=0$ and $M$ is not a $B_2$-semigroup.
\end{proof}
\begin{proposition}\label{p9}
  Let $V$ be a semigroup variety and let $M_l\notin V$. Then an identity of one of the following forms is fulfilled in $V$: \eqref{L5.1},
  \begin{subequations}
    \begin{align}
      \label{P9.1}
      xy^n=PxQxR;\\
      \label{P9.2}
      axP=ayQ,
    \end{align}
  \end{subequations}
  (here $a$, $x$, $y$ are different variables, $P$, $Q$, $R$ are words or empty symbols and $a\notin\chi (P)\cup \chi (Q)$).
\end{proposition}
\begin{proof}
  Let $M_l\notin V$. Then an identity $\mathcal{U}=\mathcal{V}$ of the variety $V$ does not hold in $M_l$. Assume that this identity does not imply any of the identities of the forms \eqref{L5.1}, \eqref{P9.1}, \eqref{P9.2}. Then obviously $\chi (\mathcal{U})=\chi (\mathcal{V})$. If the first letters $x$, $y$ of the words $\mathcal{U}$, $\mathcal{V}$ are different, then $a\mathcal{U}=a\mathcal{V}$, where the letter $a$ does not belong to $\chi (\mathcal{U})\cup \chi (\mathcal{V})$, is an identity of the form \eqref{P9.2}. Therefore, the words $\mathcal{U}$, $\mathcal{V}$ begin with the same letter $a$. Let us consider all possible cases.
  \begin{enumerate}
    \item The letter $a$ occurs at least twice in each of the words $\mathcal{U}$, $\mathcal{V}$. In this case the identity $\mathcal{U}=\mathcal{V}$ holds in $M_l$, and it contradicts the assumption.
    \item The letter $a$ occurs only once in $\mathcal{U}$ and at least twice in $\mathcal{V}$ (or \textit{vice versa}). In this case the identity $\mathcal{U}=\mathcal{V}$ implies, evidently, an identity of the form \eqref{P9.1}, and this contradicts the hypothesis.
    \item The letter $a$ occurs only once both in $\mathcal{U}$ and $\mathcal{V}$. If the second letters of the words $\mathcal{U}$, $\mathcal{V}$ coincide as well, then the identity $\mathcal{U}=\mathcal{V}$ holds in $M_l$; therefore, the second letters $x$, $y$ of the words $\mathcal{U}$, $\mathcal{V}$ are different and they differ from $a$. Hence, $\mathcal{U}=\mathcal{V}$ is an identity of the form \eqref{P9.2} in contradiction with the assumption.
  \end{enumerate}\end{proof}

\begin{proposition}\label{p10}
  Let $V$ be the semigroup variety defined by the identities $xy=xy^{n+1}$, $axay=ayax$, $xabc=xacb$. Then any semigroup $S$ of the variety $V$ is residually completely 0-simple.
\end{proposition}
\begin{proof}
  It is not difficult to show that $V$ is a Rees---Sushkevich variety. Let $x\neq x^*$ be arbitrary elements of the semigroup $S\in V$.

  If the elements $x$, $x^*$ are regular elements, then, by Proposition \ref{p4}, they can be separated by a homomorphism into a completely 0-simple semigroup. Assume that the element $x\neq0$ is not regular. If $x\notin S^2$ then the homomorphism $S\rightarrow\{0,1\}$ taking $x$ into 1 and taking all other elements into 0 separates $x$ from all other elements. Therefore, without loss of generality, we can assume that $x\in S^2$; then $x=uv$ for $u,v\in S$, and $x=uv=uv^{n+1}=xv^n\in xS$.

  Denote by $I_x$ the set of all elements $y\in S$ such that $x\notin S^1yS^1$; it is clear that $I_x$ is a two-side ideal of $S$. Since $x\notin I_x$, the elements $x$, $x^*$ remains distinct in the factor semigroup $S/I_x$; therefore, we can replace $S$ with $S/I_x$ and assume, without loss of generality, that any nonzero element of $S$ is a divisor of $x$.

  Let $L=\{a\in S|x=xa^{2n}\}$, $H=S\backslash L$. The element $x$ is not contained in $L$ (otherwise $x$ should be a regular element of $S$); therefore, $x\in H$. Since the identity $(ab)^m=a^mb^m$ holds in $V$ (it easily follows from the identities defining the variety $V$), $L$ is a subsemigroup of $S$. Let $y\in H$, $z\in S$; if $zy\neq0$, then $zy$ divides $x$, i.e., there exist elements $u,v\in S^1$ such that $x=uzyv$. Making us of the defining identities of the variety $V$ we obtain: $x=uzyv=uzy^{n+1}v=uzy^{2n}v=uzyvy^{2n}=xy^{2n}$, and $y\in L$. But it contradicts the assumption $y\in H=S\backslash L$. Therefore, $zy=0$ for every $y\in H$, $z\in S$. Thus, $H$ is the right annihilator of $S$ and, consequently, $H$ is an isolated two-side ideal.

  Let $y\in H$, $y\neq0$. The element $y$ divides $x$; therefore, there are elements $u,v\in S^1$ such that $x=uyv$. If $u\in S$ then $uy=0$ and $x=0$; hence, $u\notin S$, i.e., $u=1$ and $x=yv=yv^{n+1}=yv^{2n+1}=xv^{2n}$. Thus, we have proved that for any nonzero element $y\in H$ there is an element $v\in L$ such that $x=yv$.

  If $x^*\in L$ then the elements $x$, $x^*$ are separated by the homomorphism $f:S\rightarrow \{0,1\}$ defined by the formulae: $f(z)=0$ if $z\in H$, $f(z)=1$ if $z\in L$.

  Let $x^*\notin L$; assume at first that $x^*=0$. Introduce a binary relation $\lambda $ on the set $S\backslash \{0\}$:
  $$u\lambda v\Longleftrightarrow \text{ there exists }w\in S\text{ , such that }u\in wS^1,v\in wS^1.$$

  Show that if $u\in H$, $u\lambda v$, then $v\in H$. Indeed, otherwise $v\in L$. There exist elements $a,b\in S^1$ such that $u=wa$, $v=wb$. Hence, $wb=v\in L$, and $w\notin H$ (because $H$ is a two-side ideal), i.e., $w\in L$. Since $wa=u\neq0$, the element $a$ does not belong to the right annihilator $H$ of the semigroup $S$; it follows that $a=1$ or $a\in L$. In both cases $u=wa$ is contained in the semigroup $L$ in contradiction with the hypothesis.

  Let $z\in S$, $zu=0$, $u\lambda v$; show that $zv=0$. In fact, if $u\in H$ then $v\in H$ and $zv=0$. If $u\notin H$ then $v\notin H$, i.e., both elements $u$, $v$ are contained in $L$. If $zv\neq0$ then there are elements $c,d\in S^1$ such that $x=czvd$; since $u\lambda v$, there are elements $w\in S$, $a,b\in S^1$ such that $u=wa$, $v=wb$. Then $$x=czwbd=czwbda^{2n}=czwa^{2n}bd=czua^{2n-1}bd=0$$
  in contradiction with the assumption $x\neq0$.

  Let $\lambda^t$ be the transitive closure of the relation $\lambda $. It is clear that the relation $\lambda ^t$ is an equivalence and that it has (together with $\lambda $) the properties:
  \begin{enumerate}
    \item if $u\in H$ and $u\lambda ^tv$ then $v\in H$;
    \item if $z\in S$, $zu=0$, $u\lambda ^tv$ then $zv=0$.
  \end{enumerate}

  Let $M=M^0(\{1\}), \Lambda ,I,P)$ be the semigroup of matrix type with $\Lambda =S\backslash\{0\}$, $I=\Lambda /\lambda ^t$ and with the matrix $P$ defined by the rule: for any elements $z_1\in \Lambda $, $\bar z_2=z_2\pmod{\lambda ^t}\in I$ the entry $P_{z_1\bar z_2}$ is equal to 1, if $z_1z_2\neq0$ and $P_{z_1\bar z_2}=0$ if $z_1z_2=0$ (this definition does not depend on the choice of a representative $z_2$ of the equivalency class $\bar z_2$).

  Define a mapping $f:S\rightarrow M$ by the following rule:
  $$f(z)=\begin{cases}
  0, &\text{ if }z=0;\\
  (\bar z,1,x), &\text{ if }z\in L;\\
  (\bar x,1,x), &\text{ if }z\notin L\quad,z\in zL;\\
  (\bar x,1,z), &\text{ if }z\notin L\quad,z\notin zL.
  \end{cases}$$

  We shall prove that $f$ is a homomorphism of semigroups; this homomorphism separates $x\neq0$ and $x^*=0$, because $f(x)\neq0$, $f(x^*)=0$.

  To verify that $f$ is an homomorphism take elements $z_1,z_2\in Z$, denote $z=z_1z_2$ and consider all possible cases. If at least one of the elements $z_1$, $z_2$ is equal to 0 then $z=0$ and $f(z)=0=f(z_1)f(z_2)$. Let $z_1\neq0$, $z_2\neq0$.
  \begin{enumerate}
    \item The elements $z_1$, $z_2$ are contained in $H$. Then $z=0$, and $f(z_1)f(z_2)=(\bar x,1,u)(\bar x,1,v)$, where $u,v\in \{z_1,z_2,x\}\subseteq H$. Hence, $ux=0$ and consequently $P_{u,\bar x}=0$; therefore, $f(z_1)f(z_2)=0=f(z)$.
    \item The element $z_1$ is not contained in $H$ and $z_2\in H$. Then $z\in z_1H=0$ and $z_1\in L$. The entry $P_{x\bar x}$ of the matrix $P$ is equal to 0 because $x^2\in xH=0$ and $f(z_1)f(z_2)=(\bar z_1,1,x)(\bar x,1,v)=0=f(z)$.
    \item The element $z_1$ is contained in $H$ and and the element $z_2$ is not contained in $H$.
        \begin{enumerate}
          \item $x$ divides $z_1$. There is elements $u,v\in S^1$ such that $z_1=uxv$. Since $x\in H$ and $H$ annihilates $S$ from the right, $u\notin S$, i.e., $u=1$ and $z_1=xv$. We have for any $w\in L$: $$z_1w^{2n}=xvw^{2n}=xw^{2n}vw^{2n}=xw^{4n}v=xv=z_1.$$
              It follows that $z_1w\neq0$ for any $w\in L$; in particular, $z_1z_2\neq0$. But $H$ is a two-side ideal of $S$; therefore, $z_1z_2\in H$. Further, $z_1z_2=z_1z_2^{n+1}\in z_1z_2L$. By the definition of the mapping $f$ we have: $f(z_1)f(z_2)=(\bar x,1,x)(\bar z_2,1,x)=(\bar x,P_{x,\bar z_2},x)$, $f(z_1z_2)=(\bar x,1,x)$. But $xz_2\neq0$, because $0\neq x=xz_2^{2n}$; hence, $P_{x,\bar z_2}=1$ and $f(z_1z_2)=f(z_1)f(z_2)$.
          \item $x$ does not divide $z_1$.
           \begin{enumerate}
                \item $z_1z_2=0$. If $z_1\in z_1L$, then $z_1=z_1b$ for an element $b\in L$; since $z_1\neq0$, $0\neq x=z_1v=z_1bv=z_1bvz_2^{2n}=z_1bz_2^{2n}v=z_1z_2^{2n}v=0$. Therefore, $z_1\notin z_1L$ and $f(z_1)f(z_2)=(\bar x,1,z_1)(\bar z_2,1,x)=0$, because $P_{z_1,\bar z_2}=0$, and $f(z_1z_2)=f(0)=0$.
                \item $z_1z_2\neq0$. In this case $f(z_1)f(z_2)=(\bar x,1,u)(\bar z_2,1,x)$, where $u$ is one of the elements $x$, $z_1$. But both elements $xz_2$, $z_1z_2$ are not equal 0; hence, $uz_2\neq0$ and $P_{u,\bar z_2}=1$. Therefore, $f(z_1)f(z_2)=(\bar x,1,x)$. Since $z_1z_2\notin L$ and $z_1z_2=z_1z_2^{n+1}\in z_1z_2L$, the image $f(z_1z_2)$ of the element $z_1z_2$ is also equal to $(\bar x,1,x)$.
              \end{enumerate}
        \end{enumerate}
    \item Both elements $z_1$, $z_2$ are not contained in $H$. In this case $z_1\in L$, $z_2\in L$, $z_1z_2\in L$, and $f(z_1z_2)=(\overline{z_1z_2},1,x)$. Since $xz_2\neq0$, the entry $P_{x,\bar z_2}$ is equal to 1 and $f(z_1)f(z_2)=(\bar z_1,1,x)(\bar z_2,1,x)=(\bar z_1,1,x)$. But it follows from definition that $z_1\lambda z_1z_2$; therefore, $\bar z_1=\overline{z_1z_2}$, and $f(z_1)f(z_2)=f(z_1z_2)$.
  \end{enumerate}

  Let now $x^*\neq0$. If $x$ does not divide $x^*$ then denote by $S^*$ the factor semigroup of $S$ over the two-side ideal of $S$ generated by $x$. The semigroup $S^*$ belongs to the variety $V$, the image of $x^*$ in it is distinct from 0 and the image of $x$ in it is equal to 0; hence, we came to the situation which we have already examined above (with permutated $x$, $x^*$). Therefore, it remains to consider the case of $x$ dividing $x^*$.

  Introduce a binary relation $\tau $ on the semigroup $L$ by the rule
  $$a\tau b\Longleftrightarrow xa=xb.$$
  It follows immediately from this definition that if $a\tau b$, $t\in L$, then $at\tau bt$. Moreover, in this situation
  $$xta=xt^{2n+1}a=xt^{2n}at=xat=xbt=xt^{2n}bt=xt^{2n+1}tb=xtb,$$
  i.e., $ta\,\tau \,tb$. Thus, $\tau $ is a congruence on the semigroup $L$.

  Remark that $x=xa^{2n}=xb^{2n}$, $xb=xa^{2n}b=xba^{2n}$, i.e., $a^{2n}\tau b^{2n}$, $b\tau a^{2n}b\tau ba^{2n}$ for every $a,b\in L$. It means that the factor semigroup $G=L/\tau $ is a group. We shall denote by $|a|$ the $\tau$-class of the element $a\in L$; as above, we denote by $\bar z$ the $\lambda ^t$-class of the element $z\in S\backslash \{0\}$.

  Let $M$ be the semigroup of matrix type $M^0(G,\Lambda ,I,P)$, where $\Lambda=S\backslash \{0\}$, $I=\Lambda /\lambda ^t$ and the matrix $P$ is defined by the rule:
  $$P_{z_1,\bar z_2}=\begin{cases}
  0, &\text{ if }z_1z_2=0,\\
  1, &\text{ if }z_1z_2\neq0,z_1\in z_1L,\\
  |b^{2n-1}z_2^{2n-1}|, &\text{ if }z_1z_2\neq0,\,z_1\notin z_1L,\,x=z_1z_2b\in z_1z_2L,\\
  1 &\text{ otherwise.}\end{cases}$$

  Define a mapping $f:S\rightarrow M$:
  $$f(z)=\begin{cases}
  0, &\text{ if }z=0,\\
  (\bar z,|z^{2n+1}|,x), &\text{ if }z\in L,\\
  (\bar x,|b^{2n-1}|,x), &\text{ if }z\notin L,z\in zL,x=zb,\\
  (\bar x,1,z), &\text{ if }z\notin L,z\notin zL.\end{cases}$$

  Show that these definitions do not depend on the choice of the element $b$. Indeed, let $x=zb_1=zb_2$ for $b_1,b_2\in L$, and let $z\in zL$; then there is an element $c\in L$ such that $z=zc$ and $$xb_1=xb_1b_2^{2n}=zcb_1b_1b_2^{2n}=zcb_2b_1^2b_2^{2n-1}=xb_1^2b_2^{2n-1}.$$

  It follows that $xb_1b_2=xb_1^2$, i.e., $|b_1||b_2|=|b_1b_2|=|b_1^2|=|b_1|^2$, and, since $G$ is a group, $|b_1|=|b_2|$.

  Now we shall prove that $f$ is a homomorphism of semigroups. If $z_1=0$ or $z_2=0$ then $f(z_1z_2)=f(0)=0=f(z_1)f(z_2)$. Let $z_1\neq0$, $z_2\neq0$; consider all possible cases.
  \begin{enumerate}
    \item The elements $z_1$, $z_2$ are contained in $H$. Then $z_1z_2=0$ and $f(z_1z_2)=0$, $f(z_1)f(z_2)=(\bar x,g_1,u)(\bar x,g_2,v)$, where $u,v\in \{x,z_1,z_2\}$. Therefore, $ux=0$ and $P_{u,\bar x}=0$, $f(z_1)f(z_2)=(\bar x,0,v)=0$.
    \item $z_2$ is contained in $H$ and $z_1$ is not contained in $H$. Hence, $z_1\in L$. In this case $z_1z_2=0$ again. Since $x^2=0$, the entry $P_{x,\bar x}$ of the matrix $P$ is equal to 0, and $f(z_1)f(z_2)=(\bar x,g_1,x)(\bar x,g_2,v)=0=f(z_1z_2)$.
    \item $z_1$ is contained in $H$ and $z_2$ is not contained in $H$, i.e., $z_2\in L$.
        \begin{enumerate}
          \item $z_1\in z_1L$. There are elements $b_1,b,c\in L$ such that $x=z_1b_1=z_1z_2b$, $z_1=z_1c$. Then $f(z_1)f(z_2)=(\bar x,|b_1^{2n-1}|,x)(\bar z_2,|z_2^{2n+1}|,x)$. Remark that $x=xz_2^{2n}\in xL$; therefore, $P_{x,\bar z_2}=1$ and $f(z_1)f(z_2)=(\bar x,|b_1^{2n-1}z_2^{2n+1}|,x)$. On the other hand, $f(z_1z_2)=(\bar x,|b^{2n-1}|,x)$, because $z_1z_2\notin L$, $z_1z_2=z_1z_2^{n+1}\in z_1z_2L$. But we have:  $xb_1=z_1z_2bb_1=z_1cz_2bb_1=z_1cb_1z_2b=z_1b_1z_2b=xz_2b$; therefore, $|b_1|=|z_2b|$ and $|b_1^{2n-1}z_2^{2n+1}|=|z_2^{4n}b_1^{2n-1}|=|b_1^{2n-1}|$, because $xz_2^{4n}b_1^{2n-1}=xb_1^{2n-1}$. Thus, $f(z_1)f(z_2)=f(z_1z_2)$.
          \item $z_1\notin z_1L$. If $z_1z_2=0$ then $P_{z_1\bar z_2}=0$ and $$f(z_1)f(z_2)=(\bar x,1,z_1)(\bar z_2,|z_2|^{2n+1},x)=0=f(0)=f(z_1z_2).$$
              Let $z_1z_2\neq0$; since $z_1z_2\in H$, there is an element $b\in L$ such that $x=z_1z_2b$. Further, $z_1z_2=z_1z_2^{n+1}\in z_1z_2L$. It follows from the definition that $P_{z_1,\bar z_2}=|b^{2n-1}z_2^{2n-1}|$; therefore,
              \begin{align}
                \nonumber
                f(z_1)f(z_2)=(\bar x,1,z_1)(\bar z_2,|z_2^{2n+1}|,x)=(\bar x,|b^{2n-1}z_2^{2n-1}z_2^{2n+1}|,x)=\\
                \nonumber
                (\bar x,|b^{2n-1}||z_2^{4n}|,x)=(\bar x,|b^{2n-1}|,x)=f(z_1z_2).
              \end{align}
        \end{enumerate}
    \item Both elements $z_1$, $z_2$ are not contained in $H$. Then $z_1,z_2,z_1z_2\in L$. Remark that $z_1\in z_1S^1$, $z_1z_2\in z_1S^1$, therefore, $z_1\lambda z_1z_2$ and $\bar z_1=\overline {z_1z_2}$. Since $P_{x,\bar z_2}=1$, we obtain:
        \begin{align}
          \nonumber
          f(z_1)f(z_2)=(\bar z_1,|z_1^{2n+1}|,x)(\bar z_2,|z_2^{2n+1}|,x)=(\bar z_1,|z_1^{2n+1}||z_2^{2n+1}|,x)=\\
          \nonumber
          (\overline{z_1z_2},|(z_1z_2)^{2n+1}|,x)=f(z_1z_2).
        \end{align}
  \end{enumerate}
  Thus, in all cases $f(z_1z_2)=f(z_1)f(z_2)$, i.e., $f$ is a homomorphism.

  Suppose that $f(x)=f(x^*)$. Recall that $x$ divides $x^*$; since $x\in H$, it means that $x^*=xd$ for an element $d\in S$. But $d\notin H$ (otherwise $x^*=xd=0$), therefore, $d\in L$ and $x^*=xd=xd^{2n}d=x^*d^{2n}\in x^*L$. The element $x^*=xd$ is contained (together with the element $x$) in the ideal $H$. Therefore, there exist elements $b,c\in L$ such that $x=xb$, $x=x^*c$. Then $f(x)=(\bar x,|b^{2n-1}|,x)$, $f(x^*)=(\bar x,|c^{2n-1}|,x)$. Since $f(x)=f(x^*)$, we obtain that $|b^{2n-1}|=|c^{2n-1}|$. But $|b|=|b^2|$ because $xb=xb^2$, therefore, $|b|=1$ and $|c^{2n-1}|=1$. Hence, $|c^{2n}|=|c|$ and $xc=xc^{2n}=x$.

  On the other hand, $x^*=xd$ and $x=x^*c=xdc=xc^{2n}dc=xc^{2n}cd=xcd=xd=x^*$. It contradicts the assumption $x\neq x^*$. Therefore, $f(x)\neq f(x^*)$.
\end{proof}

\begin{proposition}\label{p11}
  The semigroups $L_2^1$, $R_2^1$, $F_\lambda $, $F_\rho $ are not residually completely 0-simple.
\end{proposition}
\begin{proof}
  Since completely 0-simple semigroups do not contain any subsemigroup isomorphic with $L_2^1$ or $R_2^1$, it is easy to show that the 3-element semigroups $L_2^1$, $R_2^1$ are not residually completely 0-simple.

  Assume that the semigroup $F_{\lambda}$ is residually completely 0-simple. Let $a=xy$, $b=x^2y$. It follows from this definition of the semigroup $F_\lambda$ that $a\neq b$. According to our hypothesis there exists a completely 0-simple semigroup $M$ and a homomorphism $f:F_\lambda \rightarrow M$ such that $f(a)\neq f(b)$. If $f(x^2)=0$ then $f(b)=f(x^2)f(y)=0$ and $f(a)=f(xy)=f(xyx)=f(xyx^2)=f(xy)f(x^2)=0$. Therefore, $f(x^2)\neq0$, and it follows from the equality $x^2=x^3$ that $(f(x))^2=(f(x))^3$. Since $M$ is a completely 0-simple semigroup, we obtain that $f(x)=(f(x))^2$ and $f(x)=f(x^2)$. Hence, $f(a)=f(xy)=f(x)f(y)=f(x^2)f(y)=f(x^2y)=f(b)$ in contradiction with the assumption. Thus, $F_\lambda $ is not a residually completely 0-simple semigroup.

  Similarly we can prove that $F_\rho $ is not a residually completely 0-simple semigroup.\end{proof}

\section{The proof of the main results}

\begin{theorem}\label{t1}
  A semigroup variety $V$ is a Rees---Sushkevich variety if and only if it contains none of indicator Burnside semigroups (1)--(13).
\end{theorem}
\begin{proof}
  \textbf{Necessity.} Let $V$ be a Rees---Sushkevich variety. Lemma \ref{l10} asserts that $V$ does not contain any of the  semigroups (1)--(7).  The remaining semigroups $F_\lambda$, $F_\rho$, $W_\lambda$, $W_\rho$, $L_{2}^{1}$ and $R_{2}^{1}$ are not contained in $V$ because the identity \eqref{3} is not fulfilled in these semigroups.

  \textbf{Sufficiency.} Let $V$ be a semigroup variety which does not contain any of the semigroups (1)--(13). According to Lemma \ref{l10}, the identities \eqref{1} and \eqref{2} hold in $V$. By Lemma \ref{l7} one of the identities \eqref{L5.1}, \eqref{L7.1},\eqref{L7.2}, \eqref{L7.3}, \eqref{L7.4} holds in $V$ as well. Let us prove that in all cases the identities \eqref{L15.1} and \eqref{L15.2} hold in $V$ as well.
  \begin{description}
    \item[The identity \eqref{L7.1}.] Making use of \eqref{2}, we obtain the identity $xyx=Py^2Q^*$ which implies, by Lemma \ref{l15}, the identities \eqref{L15.1} and \eqref{L15.2}.
    \item[The identity \eqref{L7.2}.] In this case  $(xy)^{n+1}=(xy)^{2n+1}$ because of \eqref{1}; therefore, $(xy)^{n+1}$ is divisible by  $(yx)^{n+1}=P^*y^2Q^*$. Thus, \eqref{L7.2} implies \eqref{L7.1}, and it was shown above that \eqref{L15.1} and \eqref{L15.2} hold in $V$.
    \item[The identity \eqref{L7.3}.] Remark that the only difference between the identities \eqref{L7.3} and \eqref{L11} is that the degrees $k_i$, $l_j$ are greater or equal to 1 in the first identity and they are strictly greater than 1 in the second identity. Making use of \eqref{2} we easily transform \eqref{L7.3} into $(L11)$, and Lemma \ref{l11} guarantees the fulfillment of \eqref{L11.1} and \eqref{L11.2}; but these identities coincide with \eqref{L15.1} and \eqref{L15.2}.
    \item[The identity \eqref{L7.4}.] Let $p>1$ be an integer such that $k+p\equiv 1\pmod{n}$. Multiplying \eqref{L7.4} on the right by $x(ax)^{p-1}$, making use of \eqref{1} and then replacing $a$ with $x$, $x$ with $y$ obtain the identity $(xy)^{m-1}xy^2(xy)^{p-1}=(xy)^{n+1}$ of the form \eqref{L7.1}. As we have seen in the first case, the last identity implies the identities \eqref{L15.1} and \eqref{L15.2}.
    \item[The identity \eqref{L5.1}.] The substitution $x\mapsto xy$, $y\mapsto y^2$ transforms this identity into the identity $(xy)^n=P^*y^2Q^*$ which implies the identity $(xy)^{n+1}=P^*y^2Q^*xy$ of the form \eqref{L7.1} and, consequently, the identities \eqref{L15.1} and \eqref{L15.2}.
  \end{description}
  If the identity \eqref{3} is not satisfied in $V$ then the following inequality is true in the free semigroup $S=S(x,y,z,h)$ of the variety $V$:
  $$xyz\cdot(xhz)^n\neq(xhz)^n\cdot xyz.$$
  Let $R$ and $T$ denote respectively the left and the right sides of this inequality. It was shown in \cite{Mashevitzky6} that $R$ and $T$ are regular elements of $S$. By Proposition \ref{p4} there exists a homomorphism $f:S\rightarrow M \in CS^0(V)$ such that $f(R) \neq f(T)$. But the identity \eqref{3} holds in $M$, therefore, $f(R)=f(T)$. This contradiction proves that \eqref{3} holds in $V$. Thus, the identities \eqref{1}, \eqref{2} and \eqref{3} hold in $V$, i.e., $V$ is a Rees---Sushkevich variety.\end{proof}

\begin{theorem}\label{t2}
  Let $V$ be a periodic semigroup variety. The variety $V$ is generated by completely 0-simple semigroups if and only if $V$ is a Rees---Sushkevich variety and one of the following condition is fulfilled:
  \begin{enumerate}
    \item $B_{2}\in V$, $A_{0}\notin V$;
    \item $A_{2}\in V$;
    \item $N_{2}\notin V$.
  \end{enumerate}
\end{theorem}
\begin{proof}
  \textbf{Necessity.} If $V$ is generated by its 0-simple semigroups then any free semigroup of the variety $V$ is a residually $CS^0(V)$-semigroup, i.e., $V$ is a Rees---Sushkevich variety.

  If $B_2 \notin V$ then all completely 0-simple semigroups are completely simple semigroups (possibly with the outer zero). Let us show that in this case $N_{2} \notin V$. If the contrary is true then the identity $x=x^{n+1}$ does not hold in the free semigroup $S$ of the variety $V$. As we have seen above, there exists a homomorphism $f$ of the semigroup  $S$ into a semigroup $M^0 \in CS(V)$, such that $f(x) \neq f(x^{n+1})$. But the identity $z=z^{n+1}$ holds in the semigroup $M^0$, and therefore, $f(x)=(f(x))^{n+1}=f(x^{n+1})$ in contradiction with the previous statement.

  Assume that neither (2) nor (3) is true and show that then (1) is fulfilled. Assume the contrary. We have seen above that if $B_2\notin V$ then $N_{2}\notin V$, i.e., (3) is true; therefore, $B_2\in V$. But we assumed that (1) fails; it is possible only if $A_0\in V$.  Therefore, the identity
  \begin{equation}\label{t2*}
    x^2y^2=(x^2y^2)^{n+1}
  \end{equation}
  does not hold in $V$, and $x^2y^2\neq(x^2y^2)^{n+1}$ in the free semigroup $S=S(x,y)$ of the variety $V$.  The variety $V$ is generated by completely 0-simple semigroups; hence, there is a semigroup $M\in CS^0(V)$ and a homomorphism $f:S\rightarrow M$ such that $f(x^2y^2)\neq f((x^2y^2)^{n+1})$.

  Since $A_2\notin V$, the sandwich-matrix of $M$ (as the sandwich-matrix of any 0-simple semigroup of $V$) consists of diagonal blocks, and every block contains only nonzero elements. The identity $(*)$ holds in $B_2$; therefore, both sides of the inequality $f(x^2y^2)\neq f((x^2y^2)^{n+1})$ lay in the domain of action of the same block (or both sides are equal to 0). This means that the values of the left and the right sides of this inequality are not equal already in a completely simple semigroup. But it is not possible because all completely simple semigroups satisfy \eqref{t2*}.

  \textbf{Sufficiency.} Let $V$ be a Rees---Sushkevich variety. We shall prove that if any of the conditions (1)-(3) hold then $V$ is generated by completely 0-simple semigroups.
  \begin{enumerate}
    \item $B_2\in V$, $A_0\notin V$. Let us show that free semigroups of $V$ are approximated by completely 0-simple semigroups of the variety $V$. Let $S$ be a free semigroup of $V$ and let $X$ be the set of its free generators. Further, let $\mathcal{U},\mathcal{V}\in S$ and $\mathcal{U}\neq\mathcal{V}$.

        If the elements $\mathcal{U}$, $\mathcal{V}$ are regular then the existence of a homomorphism $f$ of the semigroup $S$ into a semigroup $M\in CS^0(V)$, such that $f(\mathcal{U})\neq f(\mathcal{V})$, is ensured by Proposition \ref{p4}. Assume that one of the elements $\mathcal{U}$, $\mathcal{V}$ is not regular. Without loss of generality we can suppose that $\mathcal{U}$ is not regular.
        \begin{enumerate}
          \item $\chi (\mathcal{U})\neq\chi (\mathcal{V})$. Then there exists a canonical homomorphism of $S$ onto the semigroup $\{0,1\}$ such that $f(\mathcal{U})\neq f(\mathcal{V})$. It remains to note that $\{0,1\}\subset B_2\in V$.
          \item $\chi (\mathcal{U})= \chi (\mathcal{V})$. Prove the statement by induction on $m=|\mathcal{U}|+|\mathcal{V}|$. It is true if $m=2$. Indeed, in this case $\mathcal{U}=x$, $\mathcal{V}=x$ and $\mathcal{U}=\mathcal{V}$; therefore, the inequality $\mathcal{U}\ne\mathcal{V}$ implies the existence of the needed homomorphism because the first term of this implication always fails.

              Let $m>2$ and assume that the needed homomorphism  exist for any words $P$, $Q$ which are not equal in $S$ and the sum of the lengths of which is less than $m$.

              Let at first $|\mathcal{U}|=1$; then $|\mathcal{V}|\geq 2$ and $\mathcal{V}\in S^2$. Take for $f$ the homomorphism $S\rightarrow B_2$ defined by the formulae
              $$f(x)=\begin{cases}
                    \left(
                      \begin{array}{cc}
                        0 & 1 \\
                        0 & 0 \\
                      \end{array}
                    \right), &\text{ if } x=\mathcal{U}\\
                    \left(
                      \begin{array}{cc}
                        0 & 0 \\
                        0 & 0 \\
                      \end{array}
                    \right), &\text{ otherwise.}
              \end{cases}$$
              Then $f(\mathcal{U})\neq0$, $f(\mathcal{V})=0$, i.e., $f(\mathcal{U})\neq f(\mathcal{V})$. If $\mathcal{U}=x\in S^2$ then the identity of the form $x=x^{n+1}$ holds in $V$ and, consequently, $S$ is a regular semigroup. Therefore, $\mathcal{U}=\mathcal{U}^{n+1}$ and $\mathcal{U}$ is a regular element in contradiction with the assumption.

              Before considering another case let us make prove some auxiliary facts. Recall that $A_0\notin V$; it is easy to prove that then an identity of the kind of $x^ky^l=PyxQ$, where $P$, $Q$ are words or empty symbols holds in $V$. Moreover, let $p,q\ge 2$ be arbitrary integers. Then
              \begin{equation}\label{t2**}
                x^py^q=x^{p+nk}y^{q+nl}=x^{p+(n-1)k}x^ky^ly^{q+(n-1)l}=P^*yxQ^*.
              \end{equation}

              Let now $z$ be an arbitrary word in which every letter occurs more than once. Show that $z$ is a regular element. In fact, the word $z$  decomposes into the product $z_1z_2\dots z_q$, all factors $z_i$ of which can be covered with cycles. But any element of a semigroup from the Rees---Sushkevich variety which can be covered with cycles, is regular (see \cite{Mashevitzky6}). Thus, it remains to prove that the product of regular elements is a regular element. Let $c=ab$, $a$ and $b$ being regular elements of $S$. Then there exist elements $a',b'\in S$ such that $a=aa'a$, $b=bb'b$. Making use of \eqref{2} and \eqref{t2**} (for $p=q=n+1$) we obtain:
              $$c=ab=(aa'a)(bb'b)=a(a'a)^{n+1}(bb')^{n+1}b=aP^*(bb')(a'a)Q^*b.$$
              We see that $c$ is covered with cycles; therefore, the element $c$ is regular.

              Let now $\mathcal{U}\geq 2$.  We shall call a letter $t$ of the word $\mathcal{U}$ blocking if $\mathcal{U}=\mathcal{U}_1t\mathcal{U}_2$, $\chi (\mathcal{U}_1)\cap\chi (\mathcal{U}_2)=\emptyset$ and $t\notin \chi (\mathcal{U}_1)\cup \chi (\mathcal{U}_2)$; thus, every blocking letter occurs in $\mathcal{U}$ only once.

              If there is no blocking letters in the word $\mathcal{U}$ then, using \eqref{2}, transform $\mathcal{U}$ into a word in which every letter occurs at least twice. But we have just seen that any such word is a regular element of $S$. Therefore, $\mathcal{U}$ is a regular element in contradiction with the assumption.

              Hence, there is a blocking letter $t$ in the word $\mathcal{U}$, i.e., $\mathcal{U}=\mathcal{U}_1t\mathcal{U}_2$, $\chi (\mathcal{U}_1)\cap\chi (\mathcal{U}_2)=\emptyset$ and $t\notin \chi (\mathcal{U}_1)\cup \chi (\mathcal{U}_2)$. Let $t$ enter in $\mathcal{V}$ at least twice; then $\mathcal{V}=PtQtR$, and $t\notin\chi (Q)$ (maybe, $Q$ is the empty word).  The semigroup $B_2$ is generated by two elements $a$, $b$ such that $a^2=b^2=0$, $aba=a$, $bab=b$; let $f$ be the homomorphism of $S$ into the semigroup $B_2\in V$ which acts on the free generators $x\in X$ of the semigroup $V$ by formulae:

              $$f(x)=\begin{cases}
              a, &\text{if }  x=t; \\
              ab, &\text{if } x\in \chi (\mathcal{U}_1); \\
              ba, &\text{if } x\in \chi (\mathcal{U}_2);\\
              ba &\text{in all other cases. }
              \end{cases}$$
              Then $f(\mathcal{U})=a\neq0$, $f(\mathcal{V})=0$, and $f(\mathcal{U})\neq f(\mathcal{V})$.

              Let now $t$ occur in $\mathcal{V}$ only once; then $\mathcal{V}=\mathcal{V}_1t\mathcal{V}_2$, $t\notin \chi (\mathcal{V}_1)\cup \chi (\mathcal{V}_2)$. If $\chi(\mathcal{U}_1)\cap\chi (\mathcal{V}_2)\ne\emptyset$ or $\chi (\mathcal{U}_2)\cap \chi (\mathcal{V}_1)\neq\emptyset$, then we have for the same mapping $f:S\rightarrow B_2$:
              $$f(\mathcal{U})=a\neq0=f(\mathcal{V}).$$
              Let $\chi (\mathcal{U}_1)\cap\chi (\mathcal{V}_2)=\chi (\mathcal{U}_2)\cap \chi (\mathcal{V}_1)=\emptyset$. But $\chi (\mathcal{U}_1)\cup \chi (\mathcal{U}_2)=\chi \mathcal{U}=\chi \mathcal{V}=\chi (\mathcal{V}_1)\cup \chi (\mathcal{V}_2)$; therefore, $\chi (\mathcal{U}_1)=\chi (\mathcal{V}_1)$, $\chi (\mathcal{U}_2)=\chi (\mathcal{V}_2)$. Since $\mathcal{U}\neq\mathcal{V}$, at least one of the equalities $\mathcal{U}_1t=\mathcal{V}_1t$, $t\mathcal{U}_2=t\mathcal{V}_2$ fails (otherwise $\mathcal{U}=\mathcal{U}_1t\mathcal{U}_2=\mathcal{V}_1t\mathcal{U}_2=\mathcal{V}_1t\mathcal{V}_2=\mathcal{V}$). Assume that $\mathcal{U}_1t\ne\mathcal{V}_1t$ (another case can be treated similarly). By the inductive hypothesis, $f(\mathcal{U}_1t)\neq f(\mathcal{U}_2t)$ for a homomorphism $f$ of the semigroup $S$ into a semigroup $M\in CS^0(V)$. There exists a right unity $e\in M$ of the element $f(t)$. Let $\hat f:S\rightarrow M$ be the homomorphism whose action on the free generators $x\in X$ of the semigroup $S$ is defined by the formulas
              $$\hat f(x)=\begin{cases}
              f(x), &\text{if }  x=t\text{ or } x\in \chi (\mathcal{U}_1); \\
              e,    &\text{in all other cases.}
              \end{cases}$$

              Then
              \begin{align}
                \nonumber
                \hat f(\mathcal{U})=\hat f(\mathcal{U}_1)\hat f(t)\hat f(\mathcal{U}_2)=f(\mathcal{U}_1)f(t)e\dots e=f(\mathcal{U}_1)f(t)=\\
                \nonumber
                f(\mathcal{U}_1t)\neq f(\mathcal{V}_1t)=f(\mathcal{V}_1)f(t)e\dots e=\hat f(\mathcal{V}_1t\mathcal{V}_2)=\hat f(\mathcal{V})
              \end{align}

              Thus, we have proved, that free semigroups of $V$ are approximated by completely 0-simple semigroups of $V$ if the condition (1) is fulfilled.
        \end{enumerate}
    \item $A_{2}\in V$. Let $\mathcal{U}$, $\mathcal{V}$ be words from $S$ and $\mathcal{U}\ne\mathcal{V}$.
        \begin{enumerate}
          \item $\chi (\mathcal{U})\neq\chi (\mathcal{V})$. As in the previous case there exists a canonical homomorphism of $S$ onto the semigroup $\{0,1\}$ such that $f(\mathcal{U})\neq f(\mathcal{V})$. It remains to note that $\{0,1\}\subset B_2\in V$.
          \item $\chi (\mathcal{U})= \chi (\mathcal{V})$. We can assume without loss of generality that $\mathcal{U}$ is not a regular element and that $|\mathcal{U}|>1$, $|\mathcal{V}|>1$. Then $\mathcal{U}=\mathcal{U}_1\mathcal{U}_2$, where $\mathcal{U}_1$ and $\mathcal{U}_2$ are not empty words and $\chi (\mathcal{U}_1)\cap \chi (\mathcal{U}_2)=\emptyset$.

              If a 2-letter subword $ab$ such that $a\in \chi (\mathcal{U}_2)$, $b\in\chi (\mathcal{U}_1)$ occurs in the word $\mathcal{V}$, then it is obvious that there exists a homomorphism $f$ of $S$ into the semigroup $A_2\in V$, such that $f(\mathcal{U})\neq f(\mathcal{V})$. Therefore, we can assume that any such 2-letter subword does not occur in the word $\mathcal{V}$; since $\chi(\mathcal{U})=\chi (\mathcal{V})$, it is possible only if $\mathcal{V}=\mathcal{V}_1\mathcal{V}_2$, where $\chi (\mathcal{V}_1)=\chi (\mathcal{U}_1)$, $\chi (\mathcal{V}_2)=\chi (\mathcal{U}_2)$. The words $\mathcal{U}=\mathcal{U}_1\mathcal{U}_2$ and $\mathcal{V}=\mathcal{V}_1\mathcal{V}_2$ represent different elements of $S$; therefore, $\mathcal{U}_1\ne\mathcal{V}_1$ or $\mathcal{U}_2\neq\mathcal{V}_2$. Both cases can be treated in the same way; restrict us to the case $\mathcal{U}_1\ne\mathcal{V}_1$.

              Assume at first that the last letters $u$, $v$ of the words $\mathcal{U}_1$, $\mathcal{V}_1$ do not coincide. The semigroup $A_2$ is generated by the elements $a$, $b$ connected with the defining relations $aba=a$, $bab=b$, $a^2=a$, $b^2=0$. Then the homomorphism $f:S\rightarrow A_2$ which sends $u$ into $a$, $v$ into $ab$ and all other letters of $X$ into $ab$ separates the elements $\mathcal{U}$ and $\mathcal{V}$, because $f(\mathcal{U})=a$ and $f(\mathcal{V})=f(\mathcal{V}_1)f(\mathcal{V}_2)=ab\,f(\mathcal{V}_2)=0$ (recall that the word $\mathcal{V}_2$ is not empty because the word $\mathcal{U}_2$ is not empty and $\chi(\mathcal{V}_2)=\chi (\mathcal{U}_2)$).

              Let now $\mathcal{U}_1$ and $\mathcal{V}_1$ terminate with the same letter $c$. By the inductive hypothesis there is a homomorphism $f$ of $S$ into a semigroup $M\in CS^0(V)$ such that $f(\mathcal{U}_1)\neq\mathcal{V}_1$. Let $e\in M$ be a right unity of the element $f(c)$. Define the mapping $\hat f:X\rightarrow M$ by the formulae
              $$\hat f(z)=\begin{cases}
              f(z), &\text{if } x\in\chi(\mathcal{U}_1)=\chi(\mathcal{V}_1);\\
              e,    &\text{otherwise. }
              \end{cases}$$

              Extend this mapping till a homomorphism $\hat f:S\rightarrow M$. Then
              $$\hat f(\mathcal{U})=\hat f(\mathcal{U}_1)\neq\hat f(\mathcal{V}_1)=\hat f(\mathcal{V}).$$
              Thus, we have proved, that free semigroups of $V$ are approximated by completely 0-simple semigroups of $V$ if the condition (2) is fulfilled.
        \end{enumerate}
    \item $N_{2}\notin V$. In this case the identity $x=x^{n+1}$ holds in $V$; it follows that all semigroups of the variety $V$ are regular semigroups. The semigroups $L_2^1$, $R_2^1$, $B_2^1$ are not contained in $V$; therefore, by Proposition \ref{p3}, all semigroups of $V$ can be approximated by completely 0-simple semigroups of the variety $V$.
  \end{enumerate}\end{proof}
\begin{corollary1}
  Let $T$ be a class of periodic semigroups of totally bounded period. If the variety $Var(T)$ is generated by $0$-simple semigroups than it is generated by the principal factors of the semigroups from $T$.
\end{corollary1}
\begin{proof}
  Assume the the variety $Var(T)$ is generated by completely 0-simple semigroups of bounded period. Then the identities \eqref{1}, \eqref{2}, \eqref{3} hold in $Var(T)$ for some integer $n$. Let $T^{*}$ stand for the class of principal factors of semigroups from $T$. Let us consider possible cases.
  \begin{enumerate}
    \item $A_2\in Var (T)$.
        \begin{enumerate}
          \item $A_2\not\in Var (T^{*})$. Let $S$ be an arbitrary semigroup from $T$, $a,x,y$ are an arbitrary elements from $S$. Denote by $\mathcal{U},\mathcal{V}$ the elements $axyax$, $axyxay$.

              Suppose that $\mathcal{U}^2{\mathcal{V}}^2\neq(\mathcal{U}^2{\mathcal{V}}^2)^{n+1}$.
              We can consider the elements $\mathcal{U}^2{\mathcal{V}}^2$ and $(\mathcal{U}^2{\mathcal{V}}^2)^{n+1}$ as the words over $\{ a,x,y\}$. This words are covered by cycles. Therefore they are regular elements of $S$ (see \cite{Clifford1}). From Proposition \ref{p4} there exist a homomorphism $f$ from $S$ onto a principal factor $K$ of $S$ such that $f(\mathcal{U}^2{\mathcal{V}}^2)\neq f(\mathcal{U}^2{\mathcal{V}}^2)^{n+1}$. Note $K$ is completely $0$-simple semigroup. Since $ A_2\not\in Var(T^{*})$ and $K\in T^{*}$ then $ A_2\notin Var (K)$. It follows that $K$ is $B_2$-semigroup. In every $B_2$-semigroup the identity $t^2s^2=(t^2s^2)^{n+1}$ holds.

              Hence the equality  $f(\mathcal{U}^2)\cdot f({\mathcal{V}}^2)=f^{n+1}(\mathcal{U}^2{\mathcal{V}}^2)$ holds in $K$, a contradiction. Therefore the equality  $\mathcal{U}^2{\mathcal{V}}^2=(\mathcal{U}^2{\mathcal{V}}^2)^{n+1}$ holds in $S$. It follows that the identity $(axyax)^2\cdot (axyxay)^2=[(axyax)^2\cdot (axyxay)^2]^{n+1}$ holds in $Var (T)$. It is easy to see that $A_2$ does not satisfy the last identity, that is $A_2\notin Var (T)$, a contradiction. Hence the case $1.1.$ is impossible.
          \item  $A_2\in Var (T^{*})$. Let $F$ be a free semigroup from $Var (T)$ over countable set of generators. Assume ${\mathcal{U},V}\in F$ and $\mathcal{U}\neq{\mathcal{V}}$. From the definitions follows that there exists a semigroup $S$ from $T$ and homomorphism $\varphi$ from $F$ onto $S$ such that $\varphi (\mathcal{U})\neq\varphi ({\mathcal{V}})$.

              If $\mathcal{U}$ and ${\mathcal{V}}$ are regular elements in $F$ then $\varphi (\mathcal{U})$ and $\varphi ({\mathcal{V}})$ are regular in $S$ also. From Proposition \ref{p4} there exists a homomorphism $f$ from $S$ onto the some principal factor $K$ of $S$ such that $f[\varphi (\mathcal{U})]\neq f[\varphi ({\mathcal{V}})]$. Note that $K\in T^{*}$. Basing on it we can to prove that $F$ is residualy in class $T^{*}\cup \{ A_2\}$ by the some method as well as in the Theorem \ref{t2}. It follows that $Var(T)=Var(T^{*})$.
        \end{enumerate}
    \item $A_2\not\in Var (T)$.
        \begin{enumerate}
          \item $B_2\in Var (T)$.
            \begin{enumerate}
              \item $B_2\not\in Var (T^{*})$. Let $S$ be an arbitrary semigroup from $T$, $x,y$ are the arbitrary elements of $S$. Denote the elements $(xy)^2, (xyx^2)^{n+1}(xy)^2$ respectively by $\mathcal{U}$ and $\mathcal{V}$.

                  Suppose that $\mathcal{U}\neq{\mathcal{V}}$. We can consider the elements ${\mathcal{U},V}$ as words over $\{ x,y\}$. This words are covered by cycles. Therefore they are regular elements of $S$. From Proposition \ref{p4} there exists a homomorphism $f$ from $S$ onto a principal factor $K$ of $S$ such that $f(\mathcal{U})\not= f({\mathcal{V}})$, where $K$ is a completely $0$-simple semigroup.

                  Since $B_2\notin Var (T^{*})$ and $K\in T^{*}$ then $B_2\notin Var(K)$. It follows that $K$ is completely simple semigroup possibly with zero adjoined. Thus $K$ satisfies the identity $(xy)^2=(xyx^2)^{n+1}(xy)^2$. It follows that $f (\mathcal{U})= f ({\mathcal{V}})$, a contradiction. Hence the identity $(xy)^2=(xyx^2)^{n+1}(xy)^2$ holds in $Var (T)$. It easy to see that $B_2$ does not satisfy the last identity, that is $B_2\notin Var (T)$, a contradiction. Hence this case is impossible.
              \item $B_2\in Var (T^{*})$. Let $F$ be a free semigroup from $Var (T)$. We can to prove that $F$ is residualy in the class $T^{*}\cup \{ B_2\}$ the some method as well as in the above cases. It follows that $Var (T)= Var (T^{*})$.
            \end{enumerate}
          \item $B_2\not\in Var (T)$. As it was shown in the Theorem \ref{t2} in this case the identity $x=x^{n+1}$ holds in $Var (T)$. Then from Proposition \ref{p4} follows that $Var (T)= Var (T^{*})$.
        \end{enumerate}
  \end{enumerate}\end{proof}
\begin{theorem}\label{t3}
  Let $V$ be a periodic semigroup variety. The following conditions are equivalent.
  \begin{enumerate}
    \item Any free semigroup of the variety $V$ is the subdirect product of completely 0-simple semigroups.
    \item Any finitely generated free semigroup of the variety $V$ is the the subdirect product of completely 0-simple semigroups.
    \item Any semigroup of the variety $V$ is the subdirect product of completely 0-simple semigroups.
    \item Any finitely generated semigroup of the variety $V$ is approximated by completely 0-simple semigroups of the variety $V$.
    \item Any finite semigroup of the variety $V$ is approximated by completely 0-simple semigroups of the variety $V$.
    \item Any finite semigroup of the variety $V$ is the subdirect product of completely 0-simple semigroups.
    \item $V=qVar(W)$ for a class of completely 0-simple semigroups.
    \item $V$ is a Rees---Sushkevich variety and $N_{2}\neq V$.
    \item The identities $x=x^{n+1}$ and $(axyb)^n=(ayxb)^n$ are fulfilled in the variety $V$.
    \item Any semigroup of $V$ is the normal band of groups.
    \item $V$ does not contain the semigroups $N_{2}$, $S_0$, $S_{il}$, $S_{ir}$ $(i=1,2,3)$, $L_2^1$, $R_2^1$.
  \end{enumerate}
\end{theorem}
\begin{proof}
  The implications $(3)\Rightarrow (1)\Rightarrow (2)$, $(3)\Rightarrow (4)$,  $(3)\Rightarrow (6)\Rightarrow (5)$ are trivial, and the implications $(9)\Rightarrow (10)\Rightarrow (11) \Rightarrow (9)$ follow from the results of the papers \cite{Kublanovsky6} and \cite{Kublanovsky7}.

  Show that $(2)\Rightarrow (8)$. Let $(2)$ hold for a periodic semigroup variety $V$. It follows from this condition that the variety $V$ consists of residually completely 0-simple semigroups. By Theorem \ref{t2}, $V$ is a Rees---Sushkevich variety and one of the conditions (1)-(3) of Theorem \ref{t2} is fulfilled. Let $S$ denote the free semigroup of the variety $V$ with free generators $x$, $y$. We shall prove that $xy=(xy)^{n+1}$, where $n$ is the period of $V$; then the semigroups $A_2$, $B_2$ are not contained in $V$, i.e., the conditions (1) and (2) of Theorem \ref{t2} can not be satisfied Therefore, the statement (3): $N_{2}\neq V$ is true.

  If $xy\neq(xy)^{n+1}$ then according to (2) there is an epimorphism $f:S\rightarrow M\in CS^0(V)$ such that $f(xy)\neq(f(xy))^{n+1}$. Denote $f(x)$ and $f(y)$ by $x^*$ and $y^*$. The element $x^*y^*$ is a regular element of $M$; since the semigroup $M$ is generated by $x^*$ and $y^*$, the quasi-inverse element of the element $x^*y^*$ is a word $P(x^*,y^*)$ over $x^*$, $y^*$, and $x^*y^*=x^*y^*P(x^*,y^*)x^*y^*$. Hence, the element $x^*y^*$ is divisible by $y^*x^*$. If the element $y^*x^*$ was equal to 0 then we should have $x^*y^*=0$ and $(x^*y^*)^{n+1}=0=x^*y^*$ in contradiction with the assumption. Therefore, $y^*x^*\neq0$. Further, the inequality $x^*y^*\neq(x^*y^*)^{n+1}$ implies the inequality $x^*y^*\neq0$.

  Show that $(x^*y^*)^{n+1}\neq0$. Indeed, if the contrary was true we should have $$x^*y^*x^*y^*\dots x^*y^*=0.$$
  But the product of several elements of a completely 0-simple semigroup is 0 only if the product of two adjacent factors is equal to 0; in our case it means that $x^*y^*=0$ or $y^*x^*=0$. But we have seen that $x^*y^*\neq0$, $y^*x^*\neq0$.

  In any $n$-periodic completely 0-simple semigroup $a=a^{n+1}$ if $a^{n+1}\neq0$. Therefore, $x^*y^*=(x^*y^*)^{n+1}$ in contradiction with the assumption. This concludes the proof of the implication $(2)\Rightarrow(8)$.

  Let us show now that $(8)\Rightarrow (9)$. Assume that $V$ is a Rees---Sushkevich variety and $N_{2}\notin V$; then the semigroup $B_2\supset N_{2}$ does not belong to $V$. Therefore, all completely 0-simple semigroups of $V$ are completely simple (may be, with the adjoined outer 0). Hence, the identities $x=x^{n+1}$, $(axyb)^n=(ayxb)^n$ hold in $CS^0(V)$. Since, by Theorem \ref{t2}, the Rees---Sushkevich variety $V$ is generated by $CS^0(V)$, these identities hold in $V$ as well. Thus, we have proved that $(8)\Rightarrow (9)$.

  To prove the implication $(10)\Rightarrow (3)$ remark that any normal band of groups is the subdirect product of its principal factors, and that all these principal factors are completely simple semigroups, possibly with the adjoined outer 0 (see \cite{Petrich1}).

  Assume now that the condition (4) is satisfied, i.e., every finitely generated semigroup of $V$ is approximated with completely 0-simple semigroups of the variety $V$. By Theorem \ref{t2}, $V$ is a Rees---Sushkevich variety satisfying one of the conditions (1)-(3). By Proposition \ref{p5}, $S_0\notin V$. It was shown in \cite{Kublanovsky6} that in this case one of the identities $xy=(xy)^{n+1}$, $xy=x^{n+1}y$, $xy=xy^{n+1}$ holds in the variety $V$. Then $B_2\notin V$ and consequently $A_2\supset B_2$ is not contained in $V$; by Theorem \ref{t2}, $N_{2}\notin V$. Thus, $(4)\Rightarrow (8)$.

  If (5) holds then $S_0\notin V$ by Proposition \ref{p5}. As above it implies that $B_2\notin V$. Therefore, any completely 0-simple semigroup of the variety $V$ is a completely simple semigroup, possibly with the adjoined outer 0, and the identities $x=x^{n+1}$, $(axyb)^n=(ayxb)^n$ hold in $CS^0(V)$. But it follows from  (5) that the variety $V$ is generated by semigroups from  $CS^0(V)$; therefore, these identities are fulfilled in every finite semigroup of the variety $V$, and the semigroups $N_{2}$, $S_0$, $S_{il}$, $S_{ir}$ $(i=1,2,3)$, $L_2^1$, $R_2^1$ are not contained in $V$. Thus, $(5)\Rightarrow (11)$.\end{proof}

\begin{theorem}\label{t4}
  Let $V$ be a periodic semigroup variety. The following statements are equivalent.
  \begin{enumerate}
    \item Any free semigroup of the variety $V$ decomposes into the subdirect product of its principal factors.
    \item Any finitely generated free semigroup of the variety $V$ decomposes into the subdirect product of its principal factors.
    \item Any semigroup of the variety $V$ decomposes into the subdirect product of its principal factors.
    \item Any finite semigroup of the variety $V$ decomposes into the subdirect product of its principal factors.
    \item Any finitely generated semigroup of the variety $V$ decomposes into the subdirect product of its principal factors.
    \item Any semigroup of the variety $V$ decomposes into the subdirect product of completely 0-simple semigroups and semigroups with 0 multiplication.
    \item The identities $xy=x^{n+1}y^{n+1}$ and $(axyb)^n=(ayxb)^n$ are fulfilled in $V$.
    \item Any semigroup of the variety $V$ is residually in the class of completely 0-simple semigroups and semigroups with 0 multiplication.
    \item Any finite semigroup of the variety $V$ is residually in the class of completely 0-simple semigroups and semigroups with zero multiplication.
    \item The semigroups $S_0$, $S_{1l}$, $S_{1r}$, $S_{3l}$, $S_{3r}$ are not contained in $V$.
    \item The variety $V$ is finitely separated from one-side ideals.
    \item Any semigroup of the variety $V$ is the inflation of a normal band of groups.
    \item $V=qVar(W)$ for a set $W$ of completely 0-simple semigroups and semigroups with zero multiplication.
  \end{enumerate}
\end{theorem}
\begin{proof}
  Let $(9)$ hold in $V$;  by Proposition \ref{p5}, $S_0\notin V$. It was shown in \cite{Kublanovsky6} that in this case one of the identities $xy=(xy)^{n+1}$, $xy=x^{n+1}y$, $xy=xy^{n+1}$ holds in the variety $V$. Therefore, $B_2\notin V$, and any completely 0-simple semigroup of the variety $V$ is a completely simple semigroup, possibly with the adjoined outer 0, and the identities $xy=x^{n+1}y^{n+1}$, $(axyb)^n=(ayxb)^n$ hold in $CS^0(V)$. The semigroups with zero multiplication satisfy these identity too. But it follows from  $(9)$ that the variety $V$ is generated by semigroups from  $CS^0(V)$ and semigroups with zero multiplication; therefore, these identities are fulfilled in every finite semigroup of the variety $V$, and, consequently, the semigroups  $S_0$, $S_{1l}$, $S_{1r}$, $S_{1l}$, $S_{1r}$ are not contained in $V$. Thus, $(9)\Rightarrow (10)$.

  Prove that $(2)\Rightarrow (7)$. Let the variety $V$ satisfy $(2)$. The same argument as at the beginning of the proof of Theorem 3 shows that the identity $xy=(xy)^{n+1}$ holds in $V$. Therefore, $B_2\notin V$, and, as above, any completely 0-simple semigroup of the variety $V$ is a completely simple semigroup, possibly with the adjoined outer 0. Hence, the identities $xy=x^{n+1}y^{n+1}$, $(axyb)^n=(ayxb)^n$ hold in $CS^0(V)$. The semigroups with zero multiplication satisfy these identity too. The principal factors of any periodic semigroup are either semigroups with zero multiplication or completely 0-simple semigroups, and it follows from $(2)$ that the same identities are fulfilled in any finitely generated free semigroup of the variety $V$ and consequently in any semigroup of $V$.

  Let (7) be true in $V$ and let $S$ be any semigroup of $V$. Let $\mathcal{U}$, $\mathcal{V}$ be elements of $S$ and $\mathcal{U}\ne\mathcal{V}$. If $\mathcal{U}\notin S^2$ then the canonical epimorphism $f$ of $S$ onto the principal factor of the element $\mathcal{U}$ in $S$ defined by the rule: $f(\mathcal{U})=\mathcal{U}$, $f(x)=0$ for $x\neq\mathcal{U}$, separates $\mathcal{U}$ and $\mathcal{V}$. Similarly, if $\mathcal{V} \notin S^2$ then there exists an epimorphism of $S$ onto a principal factor of the semigroup $S$ which separates $\mathcal{U}$ and $\mathcal{V}$.

  Let now $\mathcal{U},\mathcal{V}\in S^2$. Denote by $S_r$ the set $\{x^{n+1}|x\in S\}$. It follows from the identities (7) that $S_r$ is a semigroup and that $a=a^{n+1}$ for each $a\in S_r$; therefore, $S_r$ is a regular semigroup. Applying the identities (7) once more we obtain: $\mathcal{U}=\mathcal{U}^{n+1}\in S_r$, $\mathcal{V}=\mathcal{V}^{n+1}\in S_r$. Remark that the semigroups $B_2^1$, $R_2^1$, $L_2^1$ are not contained in $V$ because at least one of the identities (7) is not fulfilled in any of these semigroups. By Proposition \ref{p3}, the regular semigroup $S_r\in V$ decomposes into the subdirect product of its principal factors; hence, there is an epimorphism $\varphi $ of $S_r$ onto its principal factor separating $\mathcal{U}$ and $\mathcal{V}$. But the mapping $\psi$ defined by the rule $\psi(x)=x^{n+1}$ is an epimorphism of $S$ onto $S_r$ because of (7), and any principal factor of $S_r$ is a principal factor of $S$. Therefore, $\varphi \psi$ is an epimorphism of $S$ on its principal factor separating $\mathcal{U}$ and $\mathcal{V}$.

  Thus, the semigroup $S$ is approximated by surjective homomorphisms onto its principal factors; it means that $S$ decomposes into the subdirect product of its principal factors. Hence, we proved that $(7)\Rightarrow (3)$.

  It was shown in \cite{Kublanovsky7} that the statement (5) is equivalent to each of the statements $(10)$, $(11)$, $(12)$. The implications $(1)\Rightarrow (3)\Rightarrow (2)\Rightarrow (1)$, $(3)\Rightarrow (4)\Rightarrow(9)$,  $(3)\Rightarrow (5)\Rightarrow (9)$, $(3)\Rightarrow (6)\Rightarrow (9)$, $(3)\Rightarrow (7)\Rightarrow (9)$, $(8)\Rightarrow (13)\Rightarrow (8)$ are trivial.\end{proof}

\begin{theorem}\label{t5}
    For every variety of semigroup $V$ the following conditions are equivalent.
  \begin{enumerate}
    \item Every semigroup from $V$ is embeds into a direct product of completely 0-simple semigroups.
    \item Every finitely generated semigroup from $V$ is embeds into a direct product of completely 0-simple semigroups.
    \item $V$ satisfies one of following systems of identities:
    \begin{enumerate}
      \item $[xy = x^{n+1}y^{n+1}, (axyb)^{n} = (ayxb)^{n}]$;
      \item $[xy = xy^{n+1}, axay = ayax, abxy = abyx]$;
      \item $[xy = x^{n+1}y, xaya = yaxa, xyab = yxab]$.
    \end{enumerate}
  \end{enumerate}
\end{theorem}
\begin{proof}
  The implication $(1)\Longrightarrow (2)$ is trivial, and the implication $(3)\Longrightarrow (2)$ follows from Proposition \ref{p10} and Theorem \ref{t4}. It remains to prove that $(2)\Longrightarrow (3)$. Assume that (2) is fulfilled. By Proposition \ref{p5}, the semigroup $S_0$ is not contained in $V$. It was shown in \cite{Kublanovsky6} that then one of the following identities holds in $V$: $xy=(xy)^{n+1}$, $xy=xy^{n+1}$, $xy=x^{n+1}y$.  It follows that the semigroup $B_2$ is not contained in $V$. Further, by Proposition \ref{p11} the semigroups $L_2^1$, $R_2^1$, $F_\lambda $, $F_\rho $ are not residually completely 0-simple; therefore, these semigroups are not contained in $V$, either. We can apply now Proposition \ref{p2} and obtain that the identities \eqref{1} and \eqref{2} are fulfilled in $V$.

  Denote by $V_r$ the intersection of the variety $V$ and the variety defined by the identity $x=x^{n+1}$. According to Proposition 3, each semigroup of the variety $V_r$ is the subdirect product of its principal factors. Therefore, by Theorem \ref{t4}, the identity $(axyb)^n=(ayxb)^n$ holds in $V_r$. Consider different possibilities.
  \begin{enumerate}
    \item The identity $xy=(xy^{n+1})$ holds in $V$. It follows from this identity and the identity $(axyb)^n=(ayxb)^n$ of the variety $V_r$ that the identity
        $$(a_1a_2x_1x_2y_1y_2b_1b_2)^n=(a_1a_2y_1y_2b_1b_2)^n$$
        is fulfilled in $V$. Making use of \eqref{1} and \eqref{2}, we obtain the identity $$xyx=(xy)^{5n}W=((xy)(xyx)(yxy)(xy))^nW=((xy)(yxy)(xyx)(xy))^nW=Py^2Q,$$
        where $P$, $Q$ are words on the alphabet $\{x,y\}$. Therefore, we can apply Lemma \ref{l14}. It follows from this lemma that the identities \eqref{L14.1}, \eqref{L14.2} hold in $V$. These identities and the identity $xy=(xy)^{n+1}$ imply: $$xy=xyxy(xy)^{n-1}y=x^{n+1}yxy(xy)^{n-1}=x^n(xy)^{n+1}=x^nxy=x^{n+1}y$$ and similarly $xy=xy^{n+1}$. Therefore,  $xy=x^nxy=x^nxy^{n+1}=x^{n+1}y^{+n+1}$. The last identity and the identity \eqref{1} imply the identity $(axyb)^n=(ayxb)^n$. Thus, the set of identities \textit{(a)} holds in $V$.
    \item The identity $xy=xy^{n+1}$ holds in $V$. We have just considered the case in which the identity $xy=x^{n+1}y^{n+1}$ holds in $V$. Hence, we can assume that this identity does not hold in $V$.

        Let $G$ be any group contained in the semigroup variety $V$, and let $H$ be a subgroup of $G$. Denote by $A=\{gH|g\in G\}$ the set of left cosets of the group $G$ with respect to the subgroup $H$. Extend the multiplication on the group $G$ till a binary operation on the set $S=G\cup A\cup \{0\}$ by setting for every $g_1,g_2\in G$:
        \begin{align}
          \nonumber
          g_1\cdot (g_2H)=g_1g_2H;
          \nonumber
          (g_1H)\cdot (g_2H)=0;
          \nonumber
          (g_1H)\cdot g_2=0.
        \end{align}
        It is easy to check that this binary operation is associative, i.e., $S$ is a semigroup. We shall denote this semigroup by $(G:H)_l$. Similarly we can construct the semigroup $(G:H)_r$ antiisomorphic to the semigroup $(G:H)_l$. These semigroups were introduced by the author in the paper \cite{Kublanovsky7}. It was shown in this paper that for any semigroup variety $V$ in which the identity $xy=xy^{n+1}$ holds and the identity $xy=x^{n+1}y^{n+1}$ fails, any group $G\in V$ and any subgroup $H\subseteq G$ the semigroup $(G:H)_r$ is contained in $V$.

        Let us prove now that all groups from the semigroup variety $V$ are abelian. Assume the contrary. Denote by $V_g$ the variety of all groups contained in $V$. Since $xy=xy^{n+1}$ in $V$, all groups of the variety $V_g$ are periodic of exponent $n$.

        Let $G^*$ be the free group of the variety $V_g$ with two free generators $x^*$, $y^*$. It is sufficient to prove that the group $G^*$ is abelian. If it is not true, then the commutator $[x^*,y^*]$ is not equal to the unity $e^*$ of the group $G^*$. Prove that $[x^*,y^*]$ is not contained in the cyclic group $\langle y^*\rangle$ generated by $y^*$. In fact, otherwise $[x^*,y^*]=(y^*)^k$ for an integer $k\neq0$. But every relation between free generators of a free group of a variety is an identity of this variety. The substitution $x^*\mapsto e^*$ turns this identity into the identity $(y^*)^k=e^*$. Therefore, we obtain that $[x^*,y^*]=e^*$ in contradiction with the above assumption.

        Let $H^*$ be a subgroup of $G^*$ which is maximal with respect to the following properties: $y^*\in H^*$, $[x^*,y^*]\not\in H^*$. Such subgroup exists by the Zorn's lemma. Let $N$ be the greatest normal subgroup of $G^*$ contained in $H^*$, $G=G^*/N$, $H=H^*/N$. Denote by $e,x,y\in G$ the images of the elements $e^*,x^*,y^*\in H^*$ under the canonical epimorphism $G^*\rightarrow G^*/N=G$. If the element $y^*\in H$ was contained in $N$ then we should have $[x^*,y^*]\in N\subseteq H^*$. Hence, $y^*\not\in N$, and  $y\not=e$.

        The subgroup $H$ of the group $G$ has the following properties:
        \begin{itemize}
          \item $H$ does not contain any nontrivial normal subgroup of $G$;
          \item the intersection of all subgroups of $G$ which strictly contain the subgroup $H$ does not coincide with $H$ (because this intersection contains the element $[x,y]\not\in H$).
        \end{itemize}
        It was shown in \cite{Kublanovsky7} that in this situation the semigroup $(G:H)_r$ is subdirectly indecomposable.

        But the semigroup $(G:H)_r$  is contained in $V$; therefore, $(G:H)_r$ is a residually completely 0-simple semigroup. Since $(G:H)_r$ is subdirectly indecomposable, it follows that the semigroup $(G:H)_r$ can be embedded into a completely 0-simple semigroup $M$. Let $f$ be an injective homomorphism of $(G:H)_r$ into $M$.  The elements $f(y)$ and $f(e)$ belong to the same subgroup $f(G)$ of the semigroup $M$; therefore, $f(e)\mathcal{H}f(y)$.

        On the other hand, $f(H)f(e)=f(He)=f(H)\neq0$, $f(H)f(y)=f(Hy)=f(H)\neq0$. Making use of the matrix description of completely 0-simple semigroups we deduce from these relations that $f(e)=f(y)$ in contradiction with the injectivity of the mapping $f$. Therefore, all groups from the semigroup variety $V$ are abelian.

        By Proposition \ref{p6}, the semigroup $T_l$ is not contained in $V$. It follows from Proposition \ref{p7} that an identity of one of the forms $x^m=PyQ$, $xR=xPxQ$ ($x\notin \chi (R)$) or $Px=Qy$ holds in $V$. Consider these cases separately.
        \begin{enumerate}
          \item An identity of the form $x^m=PyQ$ holds in $V$. Then $xy=xy^{n+1}=xy^{nm+1}=P^*(xy)^2Q^*$. It follows that the principal factor of the element $xy$ is not a semigroup with zero multiplication and, consequently, this principal factor is a 0-simple semigroup. By Mann's theorem, any 0-simple periodic semigroup is a completely 0-simple semigroup. Therefore, $xy$ is a regular element, i.e., $xy\in S_r$ and $xy=(xy)^{n+1}$ in contradiction with the assumption. Thus, this case is impossible.
          \item An identity of the form $xR=xPxQ$ ($x\notin \chi (R)$) holds in $V$. Then $xy^{n+1}=xPxQ^*$ and $xy=xy^{n+1}=xPxQ^*$. If $y\in \chi (P)\cup \chi(Q^*)$ then $xy$ is divisible by $(yx)^{n+1}=(yx)^{2n+1}$ or $xy$ is divisible by $(xy)^{n+1}$. In both cases $xy$ is divisible by $(xy)^{n+1}$; it follows that the principal factor of the element $xy$ is not a semigroup with zero multiplication and, by Mann's theorem, this principal factor is a completely 0-simple semigroup. Therefore, $xy=(xy)^{n+1}$, and it contradicts the assumption.

              Let $y\notin \chi (P)\cup \chi(Q^*)$. Then an identity of the form $x^m=P'yQ'$ holds in $V$. But we have seen above that this leads to a contradiction. Thus, this case is not possible, either.
          \item An identity of the form $Px=Qy$ holds in $V$. The semigroups $B_2^1$, $L_2^1$, $R_2^1$ are not contained in $V$. Therefore, by Proposition \ref{p3}, any regular semigroup of $V$ is the subdirect product of its principal factors. Since $B_2\notin V$ and $Px=Qy$ is an identity of $V$, all completely 0-simple semigroups of the variety $V$ are left groups. We have seen that these groups are abelian. Therefore, the identity $byx=bxy$ is fulfilled in completely 0-simple semigroups of the variety $V$ and, consequently, in regular semigroups of $V$.

              Introduce a binary relation $\rho $ on the free semigroup $S$ of the variety $V$ by the rule:
              $$x\rho y\Longleftrightarrow ax=ay \text{ for all }a\in S.$$
              It is obvious that $\rho $ is a congruence on $S$ and that the identity $z=z^{n+1}$ holds in $S/\rho $. Therefore, $S/\rho $ is a regular semigroup. Hence, the identity $bxy=byx$ is fulfilled in $S/\rho $. It follows immediately that the identity $abxy=abyx$ holds in $S$. The same argument as in the case \textit{1} enables us to conclude that an identity of the form $xyx=Py^2Q$ holds in $V$.

              It was shown above that the semigroups $L_2^1$, $R_2^1$, $F_\lambda $, $F_\rho $ are not contained in $V$. We can apply Lemma \ref{l14} and conclude that the identity $xyx=x^{n+1}yx$ is fulfilled in $V$. Making use of this identity and the identity $abxy=abyx$ we obtain:
              $$axay=a^{n+1}xay=aa^nxay=aa^nyax=a^{n+1}yax=ayax.$$
              We have proved that all identities of the set \textit{(b)} hold in $V$.
        \end{enumerate}
    \item The identity $xy=x^{n+1}y$ holds in $V$. This case can be treated exactly as the previous case.
  \end{enumerate}\end{proof}
\begin{theorem}\label{t6}
  Let $V$ be a semigroup variety. The following statements are equivalent.
  \begin{enumerate}
    \item Every semigroup of the variety $V$ is a residually $B_2$-semigroup.
    \item Every finitely generated semigroup of the variety $V$ is a residually $B_2$-semigroup.
    \item One of the following sets of identities holds in the variety $V$:
    \begin{enumerate}
      \item $xy=x^{n+1}y^{n+1}$, $(axyb)^n=(ayxb)^n$;
      \item $xy=xy^{n+1}$, $axy=ayx$;
      \item $xy=x^{n+1}y$, $xya=yxa$.
    \end{enumerate}
    \item Every semigroup of the variety $V$ is finitely approximated by Green relations $\mathcal{L}$ and $\mathcal{R}$.
    \item Every semigroup of the variety $V$ is finitely approximated by Green relation $\mathcal{H}$.
    \item The variety $V$ is residually finite or every semigroup of the variety $V$ is an inflation of a normal band of groups.
  \end{enumerate}
\end{theorem}
\begin{proof}
  The equivalences $(3)\Longleftrightarrow (4)\Longleftrightarrow (5)\Longleftrightarrow (6)$ were proved by the author in \cite{Kublanovsky7}, the implication $(1)\Longrightarrow (2)$ is trivial.

  Let the statement (2) be true. It follows from Theorem \ref{t5} that one of the following sets of identities holds in the variety $V$:
  \begin{subequations}
    \begin{align}
    \label{5.3.1}
    xy=x^{n+1}y^{n+1}, (axyb)^n=(ayxb)^n;\\
    \label{5.3.2}
    xy=xy^{n+1}, axay=ayax, abxy=abyx;\\
    \label{5.3.3}
    xy=x^{n+1}y, xaya=yaxa, xyab=yxab.
    \end{align}
  \end{subequations}
  Assume that the identities \eqref{5.3.2} holds in $V$ but at least one of the identities \eqref{5.3.1} fails. By Proposition \ref{p8}, the semigroup $M_l$ does not belong to $V$. It follows from Proposition \ref{p9} that an identity of one of the form \eqref{L5.1}, \eqref{P9.1}, \eqref{P9.2} holds in $V$. The argument of the proof of Theorem \ref{t5} shows that in our situation each of the identities \eqref{L5.1}, \eqref{P9.1} imply all identities of the set \eqref{5.3.1}. Let \eqref{P9.2} hold in $V$. Making use of the identity $abxy=axby$ we obtain the identity $ax^my^k=ay^px^q$. It follows that $ax^{n+1}y^{n+1}=ay^{p^{*}}x^{q^*}$, i.e., $axy=ay^{p^{*}}x^{q^*}$.

  Let $n$ be the period of the variety $V$. Then it is evident that $p^*\equiv1\pmod{n}$, $q^*\equiv 1\pmod{n}$; therefore,  the last identity takes the form $axy=ayx$. Thus, we have proved that \eqref{5.3.2}$\Longrightarrow$ $xy=xy^{n+1}$, $axy=ayx$. The implication \eqref{5.3.3}$\Longrightarrow$ $xy=x^{n+1}y$, $xya=yxa$ can be proved in the same way. Hence, we proved that $(2)\Longrightarrow (3)$.

  It remains to prove that $(3)\Longrightarrow (1)$. Let the identities $xy=xy^{n+1}$, $axy=ayx$ hold in $V$. It was shown in \cite{Kublanovsky7} that in this case only the following semigroups are subdirectly indecomposable in $V$:
  \begin{enumerate}
    \item cyclic group $Z(p)$ of prime order $p$, and the semigroup $Z^0(p)$ obtaines from  $Z(p)$ by adjoining outer 0;
    \item 2-element semigroup of left cycles $L_2$ and the semigroup $L_2^0$ obtaines from  $L_2$ by adjoining outer 0;
    \item $(G:H)_r$, where $G$ is a cyclic group of prime order and $H$ is its unity subgroup.
  \end{enumerate}
  It is evident that the semigroups 1, 2 are $B_2$-semigroups and consequently they are residually $B_2$-semigroups. The semigroups of the form 3 can be embedded into a Brandt semigroup over the group $G$. In fact, let $B(G)=M^0(G,P,I,\Lambda )$, where $I=\Lambda=\{1,2\}$, $P=\left(
                                                              \begin{array}{cc}
                                                                1 & 0 \\
                                                                0 & 1 \\
                                                              \end{array}
                                                            \right)$.
  Define a mapping $f:(G:H)_r\rightarrow B(G)$ by the formulae:
  $$f(g)=(1,g,1),\quad f(Hg)=(2,g,1) \quad \text{for every }g\in G.$$
  It is easy to check that $f$ is an injective homomorphism of semigroups.

  Thus, any semigroup of the variety $V$ is the subdirect product of subdirectly indecomposable semigroups, each of which can be embedded into a $B_2$-semigroup. It means that every semigroup of the variety $V$ is residually $B_2$-semigroup.

  The case when the identities $xy=x^{n+1}y$, $xya=yxa$ are fulfilled in $V$ can be studied in the same way. Thus, in all cases $(3)\Longrightarrow(1)$.\end{proof}
\begin{theorem}\label{t7}
  \begin{enumerate}
    \item Every semigroup of the variety $V$ is a residually Brandt semigroup.
    \item Every finitely generated semigroup of the variety $V$ is a residually Brandt semigroup.
    \item One of the following sets of identities holds in the variety $V$:
        \begin{enumerate}
          \item $xy=x^{n+1}y^{n+1}$, $axy^nb=ay^nxb$;
          \item $xy=xy^{n+1}$, $axy=ayx$;
          \item $xy=x^{n+1}y$, $xya=yxa$.
        \end{enumerate}
    \item The variety $V$ is residually finite or every semigroup of the variety $V$ is an inflation of a normal orthodox band of groups.
  \end{enumerate}
\end{theorem}
\begin{proof}
  The proof of this Theorem is a slight modification of the proof of Theorem \ref{t6}.
\end{proof}


\begin{thebibliography}{10}

\bibitem{Clifford1}
A.H. Clifford and G.B. Preston, \emph{The algebraic theory of semigroups}, vol.
  I,II, Amer. Math. Soc., Providence, R.I., 1961,1967.

\bibitem{Golubov1}
E.A.Golubov and M.V.Sapir, \emph{Varieties of residually finite semigroups},
  Docl. Akad. Nauk. SSSR \textbf{5} (1979), no.~247, 1037--1041.

\bibitem{Hall1}
T.~Hall, S.~Kublanovsky, S.~Margolis, M.~Sapir and P.G.~Trotter, \emph{Decidable and undecidable problems related to finite 0-simple semigroups}, J. Pure Appl
  Algebra \textbf{119} (1997), 75--96.

\bibitem{Ivanov1}
S.V. Ivanov, \emph{The Burnside problem for all sufficiently large exponents},
  Inter. J. Algebra Comput \textbf{4} (2000), no.~1,2, 1--300.

\bibitem{Mashevitzky6}
G.I. Mashevitzky, \emph{Varieties generated by completely 0-simple semigroups},
  Semigroup Varieties and Semigroups of Endomorphisms (1991), 53--62, Leningrad
  State Pedagogical Institute, Leningrad,.

\bibitem{McKenzie2}
R.~McKenzie, \emph{Residually small varieties of semigroups}, Algebra
  Universalis \textbf{2} (1981), no.~13, 171--201.

\bibitem{Petrich1}
M.Petrich, \emph{Structure of regular semigroups}, Cahiers Mathematiques,
  Montpellier, 1977.

\bibitem{Murskii1}
V.L. Murskii, \emph{Some examples of varieties of semigroups}, Mat. Zametki
  \textbf{3} (1968), no.~6, 663--670.

\bibitem{Rees1}
D.~Rees, \emph{On semi-groups}, Proc. Cambridge Philos. Soc. \textbf{36}
  (1940), 387--400.

\bibitem{Shevrin1}
L.N. Shevrin and M.V. Volkov, \emph{Identities of semigroups}, Izv. Vyssh.
  Uchebn. Zaved. Mat. (1985), no.~11, 3--47, [Russian; Engl. translation:
  Soviet Math. Izv. VUZ 29, No.11, 1--64].

\bibitem{Adian1}
S.I.Adian, \emph{The Burnside problem and identities in groups}, Nauka (1975),
  English transl., Springer-Verlag, 1979.

\bibitem{Kublanovsky7}
S.I.Kublanovskii, \emph{Finite approximability of prevarieties of semigroups
  with respect to predicates}, Modern Algebra, Leningrad Gos. Ped. Inst.,
  Leningrad (1980), 59--78.

\bibitem{Kublanovsky6}
S.I.Kublanovskii, \emph{Finite approximability and algorithmic questions}, Algebraic actions and orderings, Leningrad Gos. Ped. Inst., Leningrad (1983), 58--88.

\bibitem{Sushkevich1}
A.K. Sushkevich, \emph{\"{U}ber die endlichen gruppen ohne gesetz das
  eindeutigen umkehrbarkeit}, Math. Ann. \textbf{99} (1928), 30--50.

\end{thebibliography}
\end{document}